\documentclass[11pt]{amsart}
\usepackage[margin=1in]{geometry}
\usepackage{latexsym}
\usepackage{amsfonts}
\usepackage{amsmath}
\usepackage{amssymb}
\usepackage{amsthm}
\usepackage{enumerate}
\usepackage[hang,flushmargin]{footmisc}
\usepackage{caption}
\usepackage{tabu}
\usepackage{mathrsfs}
\usepackage{amsaddr}
\usepackage{subfig}
\usepackage{mathtools}

\usepackage{xcolor}

\definecolor{MyBlue}{rgb}{0,0,1}
\definecolor{MyRed}{rgb}{1,0,0}
\definecolor{MyGreen}{rgb}{0,1,0}
\definecolor{MyIndigo}{rgb}{0.7254,0,1}
\definecolor{MyOrange}{rgb}{1,0.4431,0}

\usepackage{yhmath}
\usepackage{graphicx}
\usepackage{epstopdf}
\usepackage{epsfig}
\usepackage{caption}

\DeclareFontFamily{U}{mathx}{\hyphenchar\font45}
\DeclareFontShape{U}{mathx}{m}{n}{
      <5> <6> <7> <8> <9> <10>
      <10.95> <12> <14.4> <17.28> <20.74> <24.88>
      mathx10
      }{}
\DeclareSymbolFont{mathx}{U}{mathx}{m}{n}
\DeclareFontSubstitution{U}{mathx}{m}{n}
\DeclareMathAccent{\widecheck}{0}{mathx}{"71}
 
\usepackage{bm} 
\usepackage{cite}

\makeatletter

\newtheorem{theorem}{Theorem}[section]
\newtheorem{lemma}[theorem]{Lemma}

\newtheorem{corollary}[theorem]{Corollary}

\newtheorem{conjecture}[theorem]{Conjecture}
\newtheorem{question}[theorem]{Question}
\newtheorem{problem}[theorem]{Problem}

\theoremstyle{definition}
\newtheorem{definition}[theorem]{Definition}
\newtheorem{remarkx}[theorem]{Remark}
\newenvironment{remark}
  {\pushQED{\qed}\remarkx}
  {\popQED\endremarkx}
\newtheorem{examplex}[theorem]{Example}
\newenvironment{example}
  {\pushQED{\qed}\examplex}
  {\popQED\endexamplex}

\DeclareMathOperator{\InsCl}{\overline{Ins}}
\DeclareMathOperator{\Var}{Var}
\DeclareMathOperator{\Av}{Av}
\DeclareMathOperator{\SW}{SW}
\DeclareMathOperator{\con}{con}
\DeclareMathOperator{\VHC}{VHC}
\DeclareMathOperator{\NC}{NC}

\DeclareMathOperator{\skel}{skel}
\DeclareMathOperator{\des}{des}
\DeclareMathOperator{\peak}{peak}
\DeclareMathOperator{\black}{black}
\DeclareMathOperator{\hook}{hook}
\DeclareMathOperator{\Comp}{Comp}
\DeclareMathOperator{\tl}{tl}
\DeclareMathOperator*{\Res}{Res}
\DeclareMathOperator{\EDP}{EDP}
\DeclareMathOperator{\ALT}{ALT}

\newcommand{\dfn}[1]{\textcolor{blue}{\emph{#1}}}

\begin{document}
\title{Troupes, Cumulants, and Stack-Sorting}
\author{Colin Defant}
\address{Princeton University \\ Department of Mathematics \\ Princeton, NJ 08544}
\email{cdefant@princeton.edu}

\begin{abstract}
In several cases, a sequence of free cumulants that counts certain binary plane trees corresponds to a sequence of classical cumulants that counts the decreasing versions of the same trees. Using two new operations on colored binary plane trees that we call \emph{insertion} and \emph{decomposition}, we prove that this surprising phenomenon holds for families of trees that we call \emph{troupes}. We give a simple characterization of troupes, showing that they are plentiful. Troupes provide a broad framework for generalizing several of the results that are known about West's stack-sorting map $s$. 
Indeed, we give new proofs of some of the main theorems underlying techniques that have been developed recently for understanding $s$; these new proofs are far more conceptual than the original ones, explain how the objects called \emph{valid hook configurations} arise very naturally, and generalize to the context of troupes. To illustrate these general techniques, we enumerate $2$-stack-sortable and $3$-stack-sortable alternating permutations of odd length and $2$-stack-sortable and $3$-stack-sortable permutations whose descents are all peaks. 

The unexpected connection between troupes and cumulants provides a powerful new tool for analyzing the stack-sorting map that hinges on free probability theory. We give numerous applications of this method. For example, we show that if $\sigma\in S_{n-1}$ is chosen uniformly at random and $\des$ denotes the descent statistic, then the expected value of $\des(s(\sigma))+1$ is \[\left(3-\sum_{j=0}^n\frac{1}{j!}\right)n.\] Furthermore, the variance of $\des(s(\sigma))+1$ is asymptotically $(2+2e-e^2)n$. We obtain similar results concerning the expected number of descents of postorder readings of decreasing colored binary plane trees of various types. We also obtain improved estimates for $|s(S_n)|$ and an improved lower bound for the degree of noninvertibility of $s:S_n\to S_n$. The combinatorics of valid hook configurations allows us to give two novel formulas that convert from free to classical (univariate) cumulants. The first formula is given by a sum over noncrossing partitions, and the second is given by a sum over $231$-avoiding valid hook configurations. We pose several conjectures and open problems. 
\end{abstract}

\maketitle

\bigskip

\section{Introduction}\label{Sec:Intro} 

Given a sequence $(m_n)_{n\geq 1}$ of elements of a field $\mathbb K$, called a \emph{moment sequence}, one can consider the corresponding sequence $(c_n)_{n\geq 1}$ of \emph{classical cumulants}, as well as the corresponding sequence $(\kappa_n)_{n\geq 1}$ of \emph{free cumulants}. Cumulants are the fundamental combinatorial tools used in noncommutative probability theory. Each of these three sequences determines the other two via summation formulas involving partition lattices and noncrossing partition lattices.   

If $(\kappa_n)_{n\geq 1}$ is a sequence of free cumulants defined by $\kappa_n=-C_{n-1}$, where $C_r=\frac{1}{r+1}\binom{2r}{r}$ is the $r^\text{th}$ Catalan number, then the corresponding sequence of classical cumulants $(c_n)_{n\geq 1}$ is given by $c_n=-(n-1)!$. Indeed, this is equivalent to the fact that the sequences $((-1)^{n-1}C_{n-1})_{n\geq 1}$ and $((-1)^{n-1}(n-1)!)_{n\geq 1}$ give the M\"obius invariants of noncrossing partition lattices and partition lattices, respectively. On the other hand, $C_{n-1}$ is the number of binary plane trees with $n-1$ vertices, while $(n-1)!$ is the number of decreasing binary plane trees with $n-1$ vertices. This might seem like a mere coincidence; one of the primary goals of this paper is to show that it is not. 

We will give a vast generalization of the above observation by developing a theory of  \emph{troupes}. These are families of colored binary plane trees that are closed under two new operations that we call \emph{insertion} and \emph{decomposition}, which resemble a product and a coproduct on trees. We will see that many classical families of rooted plane trees found in the literature are troupes. In fact, we will give a characterization of troupes, which will show that there are many of them. More precisely, we will prove that every troupe is uniquely determined by its \emph{branch generators}, which play the role of ``indecomposable'' or ``prime'' elements. We also define \emph{insertion-additive} tree statistics, some natural examples of which are (essentially) the statistic that counts right edges and the statistic that counts vertices with $2$ children. 

We will prove that if ${\bf T}$ is a troupe, then a sequence of free cumulants that counts trees in ${\bf T}$ according to some insertion-additive tree statistics corresponds to a sequence of classical cumulants that counts decreasing versions of the trees in ${\bf T}$ according to the same statistics. Some very specific manifestations of this surprising phenomenon are as follows. Free cumulants given by Narayana polynomials correspond to classical cumulants given by Eulerian polynomials. Free cumulants given by aerated Catalan numbers correspond to classical cumulants given by tangent numbers. Free cumulants given by Motzkin polynomials, which are the $\gamma$-polynomials of associahedra, correspond to classical cumulants given by $\gamma$-polynomials of permutohedra. Free cumulants given by large Schr\"oder numbers correspond to classical cumulants that count cyclically ordered set partitions. 

Our proof requires three main ingredients: the Refined Tree Decomposition Lemma, the Refined Tree Fertility Formula, and the VHC Cumulant Formula. The Refined Tree Decomposition Lemma generalizes the Refined Decomposition Lemma that the author has used to answer several questions about West's stack-sorting map in \cite{DefantCounting, DefantEnumeration, DefantMonotonicity}. The proof given here is new and is more conceptual than the original proof; it also generalizes to the setting of troupes. From the Refined Tree Decomposition Lemma, we will derive the Refined Tree Fertility Formula. This is a generalization of the Refined Fertility Formula that the author has used to answer several other questions about West's stack-sorting map \cite{DefantCatalan, DefantClass, DefantFertility, DefantFertilityWilf, DefantEngenMiller, DefantPreimages}. Again, our proof is new, is far more conceptual than the original proof, and generalizes to troupes. Our new proof also explains how the combinatorial objects called \emph{valid hook configurations}, which appear in the formula, arise naturally. Special cases of the Refined Tree Fertility Formula also imply new results about the stack-sorting map. For example, we will obtain a formula for the number of alternating permutations in $s^{-1}(\pi)$ when $\pi$ is an arbitrary permutation of odd length and $s$ denotes the stack-sorting map. We will also obtain a formula for the number of permutations in $s^{-1}(\pi)$ whose descents are all peaks when $\pi$ is an arbitrary permutation. The VHC Cumulant Formula is a result that 
converts from free cumulants to the corresponding classical cumulants via a sum over valid hook configurations. 

\medskip

The combination of the Refined Tree Fertility Formula and the VHC Cumulant Formula provides a new method for analyzing the stack-sorting map, which we will illustrate with several applications. One application is a result that was originally proven in \cite{DefantEngenMiller}, which states that \emph{uniquely sorted} permutations (i.e., permutations with exactly one preimage under $s$) are enumerated by the absolute values of the classical cumulants of the standard semicircular law, known as Lassalle numbers. This is interesting because it was once an open problem to find a combinatorial interpretation of the Lassalle numbers, and uniquely sorted permutations provide arguably the most natural such interpretation.   

For another application, we consider the problem of computing the expected value $\mathbb E(D_n)$, where $D_n=\des(s(\sigma))+1$ and $\sigma$ is chosen uniformly at random from $S_{n-1}$. Here, $\des$ denotes the permutation statistic that counts descents. One can view $\des(s(\sigma))+1$ as a measure of how far $s(\sigma)$ is from the identity permutation $123\cdots (n-1)$. It is not at all clear how one could use standard methods to prove that the limit $\lim\limits_{n\to\infty}\dfrac{\mathbb E(D_n)}{n}$ even exists. Using free probability, we will not only show that this limit exists, but will see that it is equal to $3-e$. In fact, this will follow from the shockingly simple exact formula 
\[\mathbb E(D_n)=\left(3-\sum_{j=0}^n\frac{1}{j!}\right)n.\] 
Moreover, we will provide an algorithm for computing the generating functions of the moments of the random variables $D_n$. As a consequence, we will see that the variance of $D_n$ is asymptotically $(2+2e-e^2)n$. 

Using elementary methods, we will also prove that the probability that $1$ is a descent of $s(\sigma)$ is asymptotically $3-e$. The same does not appear to be true for the probability that $i$ is a descent in $s(\sigma)$ when $i\geq 2$ is fixed. Thus, there seems to be a mysterious connection between the first index and a random index when we examine the stack-sorting image of a random permutation. 

Understanding the stack-sorting image of a random permutation is equivalent to understanding the postorder reading $\mathcal P(\mathcal T)$ of a random decreasing binary plane tree $\mathcal T$. The methods that we use to understand the random variables $D_n$ generalize immediately, allowing us to study descents in postorder readings of random trees taken from other troupes. For example, we will show that if $n$ is even and $\mathcal T$ is chosen uniformly at random from the set of decreasing full binary plane trees with $n-1$ vertices and label set $\{1,\ldots,n-1\}$, then 
\[\mathbb E(\des(\mathcal P(\mathcal T))+1)=\left(1-\frac{E_n}{nE_{n-1}}\right)n\sim\left(1-\frac{2}{\pi}\right)n,\] where $E_n$ denotes the $n^\text{th}$ Euler number. This can be rephrased in terms of stack-sorting because $\mathbb E(\des(\mathcal P(\mathcal T))+1)$ is also the expected value of $\des(s(\sigma))+1$ when $\sigma$ is chosen uniformly at random from the set of alternating permutations in $S_{n-1}$. We will also show that if $\mathcal T$ is chosen uniformly at random from the set of decreasing Motzkin trees with $n-1$ vertices and label set $\{1,\ldots,n-1\}$ ($n$ could be even or odd), then 
\[\mathbb E(\des(\mathcal P(\mathcal T))+1)\sim\left(1-\dfrac{3\sqrt 3}{2\pi}\left(e^{\frac{\pi}{3\sqrt 3}}-1\right)\right)n.\] This result can also be rephrased in terms of stack-sorting because $\mathbb E(\des(\mathcal P(\mathcal T))+1)$ is also the expected value of $\des(s(\sigma))+1$ when $\sigma$ is chosen uniformly at random from the set of permutations in $S_{n-1}$ whose descents are all peaks. As a final example of these methods, we will show that if $\mathcal T$ is chosen uniformly at random from the set of decreasing Schr\"oder $2$-colored binary trees (defined in Section~\ref{SecTreesEtc}) with $n-1$ vertices and label set $\{1,\ldots,n-1\}$, then 
\[\mathbb E(\des(\mathcal P(\mathcal T))+1)\sim\left(1-\frac{1}{2\log 2}\right)n.\] 

We will also consider $|s(S_n)|$, the number of \emph{sorted} permutations in $S_n$. Bousquet-M\'elou \cite{Bousquet} found a recurrence for these numbers, but their asymptotic behavior is not known. Using Bousquet-M\'elou's recurrence and a strengthening of Fekete's lemma due to de Bruijn and Erd\H{o}s, we will prove that the limit $\displaystyle\lim_{n\to\infty}\left(\frac{|s(S_n)|}{n!}\right)^{1/n}$ exists and is greater than $0.68631$. Free probability theory will allow us to show that the number of valid hook configurations of permutations in $S_n$ is asymptotically $n!/c^{n+1}$, where $c\approx 1.32874$ is the smallest positive real root of $\displaystyle 1-z\,{}_1\hspace{-.03cm}F_2\left(\frac{1}{2};\frac{3}{2},2;-z^2\right)$ (where $\,{}_1\hspace{-.03cm}F_2$ denotes a generalized hypergeometric function). Every sorted permutation has a valid hook configuration, so this result will imply that $\displaystyle\lim_{n\to\infty}\left(\frac{|s(S_n)|}{n!}\right)^{1/n}\leq 1/c\approx 0.75260$. 

As a final application, we consider the \emph{degree of noninvertibility} of the stack-sorting map. Given a finite set $X$ and a function $f:X\to X$, Propp and the author \cite{DefantPropp} defined \[\deg(f:X\to X)=\frac{1}{|X|}\sum_{x\in X}|f^{-1}(x)|^2\] as a measure of how far the function $f$ is from being invertible. They showed that the limit $\displaystyle\lim_{n\to\infty}\deg(s:S_n\to S_n)^{1/n}$ exists and lies in the interval $[1.12462,4]$, and they conjectured that it actually lies in the interval $(1.68,1.73)$. Free probability will allow us to obtain a lower bound of $1.62924$. 

\medskip

Associated to every valid hook configuration $\mathcal H$ are two set partitions, denoted $\vert\mathcal H$ and $\underline{\mathcal H}$. The first partition is connected, while the second is noncrossing. The connected partitions $\vert\mathcal H$ play a fundamental role in the VHC Cumulant Formula and its proof. By considering the noncrossing partitions $\underline{\mathcal H}$, we will obtain two new combinatorial formulas that express (univariate) classical cumulants $c_n$ in terms of the corresponding free cumulants $\kappa_n$. The first formula states that 
\begin{equation}\label{Eq59}
-c_n=\sum_{\eta\in\NC(n)}|\mathcal L(K(\eta))|(-\kappa_\bullet)_{\eta},
\end{equation} where $\mathcal L(K(\eta))$ can be seen as the set of linear extensions of a certain poset associated to the Kreweras complement $K(\eta)$ of the noncrossing partition $\eta$. The second formula states that 
\begin{equation}\label{Eq60}
-c_n=\sum_{\mathcal H\in\VHC(\Av_{n-1}(231))}\mathscr T_{\mathcal H}(-\kappa_\bullet)_{\underline{\mathcal H}}
\end{equation} where $\VHC(\Av_{n-1}(231))$ is the set of valid hook configurations of $231$-avoiding permutations in $S_{n-1}$ and $\mathscr T_{\mathcal H}$ is the number of linear extensions of a rooted tree poset associated to $\mathcal H$, which can be computed using the hook length formula for rooted tree posets. Let us remark that there are other notions of cumulants in noncommutative probability theory; the task of finding combinatorial formulas that convert between different types of cumulants was undertaken in \cite{Arizmendi, Belinschi, Celestino, Ebrahimi, Lehner, Josuat}. 

\medskip

One of the central notions in the study of the stack-sorting map is that of a \emph{$t$-stack-sortable permutation}, which is a permutation $\pi$ such that $s^t(\pi)$ is increasing ($s^t$ denotes the $t$-fold iterate of $s$). The enumeration of $2$-stack-sortable permutations in particular has received a huge amount of attention \cite{BonaSimplicial, Bousquet98, Branden3, Cori, DefantCounting, Dulucq, Dulucq2, Egge, Fang, Goulden}. The longstanding open problem of finding a polynomial-time algorithm for enumerating $3$-stack-sortable permutations was only solved very recently in \cite{DefantCounting}. The Refined Tree Decomposition Lemma and the Refined Tree Fertility Formula allow one to straightforwardly generalize many of the results that the author has proven about the stack-sorting map to the more general context of troupes. To illustrate this, we will show how the Refined Tree Decomposition Lemma gives a general method for enumerating $2$-stack-sortable permutations belonging to sets of permutations that are associated with troupes. For two very concrete applications, we enumerate $2$-stack-sortable alternating permutations of odd length and $2$-stack-sortable permutations whose descents are all peaks. We will also prove that the generating function that counts $2$-stack-sortable permutations associated to a troupe is algebraic whenever the generating function counting the trees in the troupe is algebraic. This is a far-reaching generalization of the fact that the generating function counting $2$-stack-sortable permutations is algebraic. Furthermore, we will show that these methods provide recurrences that count $3$-stack-sortable permutations associated to troupes. 

\medskip

As a final result, we prove that the sequence enumerating the trees in a troupe is determined by the sequence enumerating the branch generators of the troupe. In many cases, the latter sequence is much simpler than the former. This yields a new transform on nonnegative integer sequences, which we call the \emph{troupe transform}. 

As evidenced by the articles \cite{Arizmendi, Belinschi, Celestino, Ebrahimi, Lehner, Josuat}, combinatorial formulas that convert from one type of cumulant sequence to another have become popular in recent years and have appealed to a wide audience outside of enumerative combinatorics. Therefore, we believe that our main results concerning cumulant conversion formulas in Corollary~\ref{Cor4}, Theorem~\ref{Thm5}, Corollary~\ref{Cor15}, and Theorem~\ref{Thm25} will be attractive to a broad audience. Moreover, as far as we are aware, this is the first time that such cumulant conversion formulas have found further applications making use of the specific combinatorial objects that are involved. This is because of the surprising fact that valid hook configurations feature prominently in both the VHC Cumulant Formula and the Refined Tree Fertility Formula. Therefore, we believe that enumerative combinatorialists will find our specific applications to postorder readings and the stack-sorting map to be interesting in their own right and that the very existence of such applications will be interesting to a wider group of researchers.  

\subsection{Outline}
In Section~\ref{SecTreesEtc}, we introduce insertion and decomposition, define and characterize troupes, and give necessary background on the stack-sorting map and valid hook configurations. Section~\ref{SecRefinedDecomposition} is devoted to the proof of the Refined Tree Decomposition Lemma. In Section~\ref{Sec:TreeFertility}, we prove the Refined Tree Fertility Formula and detail some applications by choosing specific troupes. Section~\ref{Sec:VHCCumulant} provides necessary background on the combinatorics of free probability theory and states the VHC Cumulant Formula. In Section~\ref{Sec:TroupsAndCumulants}, we prove that a sequence of free cumulants that counts trees in a troupe according to insertion-additive tree statistics corresponds to a sequence of classical cumulants that counts the decreasing versions of the same trees according to the same statistics. We then explain in more detail how this applies to some specific troupes. In Section~\ref{Sec:Applications}, we outline several applications of the Refined Tree Fertility Formula and the VHC Cumulant Formula to the study of the stack-sorting map and, more generally, postorder readings of decreasing colored binary plane trees. Section~\ref{Sec:OtherFormulas} is devoted to proving the new cumulant conversion formulas \eqref{Eq59} and \eqref{Eq60}. Section~\ref{Sec:2-stack} provides a method for enumerating $2$-stack-sortable permutations associated with troupes, explicitly enumerates $2$-stack-sortable alternating permutations of odd length and $2$-stack-sortable permutations whose descents are all peaks, proves the algebraicity of the generating functions that count $2$-stack-sortable permutations associated with troupes counted by algebraic generating functions, and gives a recurrence for counting $3$-stack-sortable permutations associated to troupes. In Section~\ref{Sec:Transform}, we prove that the sequence enumerating the trees in a troupe is determined by the sequence enumerating the branch generators of the troupe, and we use this theorem to define the troupe transform. In Section~\ref{Sec:Conclusion}, we accumulate numerous open problems and conjectures from throughout the article. 

\subsection{Notation and Terminology}
For easy reference, we record some of the notation and terminology that we will use throughout the article. 
\begin{itemize}
\item Let $[n]$ denote the set $\{1,\ldots,n\}$. \item Given elements $a_{i_1,\ldots,i_r}$ of a field $\mathbb K$, we can consider the generating function $A(x_1,\ldots,x_r)=\displaystyle\sum_{i_1,\ldots,i_r}a_{i_1,\ldots,i_r}x_1^{i_1}\cdots x_r^{i_r}$, which is a formal power series in the variables $x_1,\ldots,x_r$. We write $[x_1^{i_1}\cdots x_r^{i_r}]A(x_1,\ldots,x_r)$ for the coefficient $a_{i_1,\ldots,i_r}$ of $x_1^{i_1}\cdots x_r^{i_r}$ in this series. 
\item A \dfn{composition} of a positive integer $b$ into $a$ parts is an $a$-tuple of positive integers that sum to $b$. If $(u_n)_{n\geq 1}$ is a sequence of elements of a field $\mathbb K$ and ${\bf q}=(q_1,\ldots,q_a)$ is a composition, then we write $u_{\bf q}$ for the product $\prod_{t=1}^au_{q_t}$. 
\item A \dfn{set partition} of a finite set $X$ is a set of pairwise-disjoint nonempty subsets of $X$ whose union is $X$. If $(u_n)_{n\geq 1}$ is a sequence of elements of a field $\mathbb K$ and $\rho$ is a set partition, then we let $(u_\bullet)_\rho=\prod_{B\in\rho}u_{|B|}$. 
\item A \dfn{permutation} is an ordering of a finite set of positive integers, which we write in one-line notation. Let $S_n$ denote the set of permutations of $[n]$. A \dfn{descent} of a permutation $\pi=\pi_1\cdots\pi_n$ is an index $i\in[n-1]$ such that $\pi_i>\pi_{i+1}$. A \dfn{peak} of $\pi$ is an index $i\in\{2,\ldots,n-1\}$ such that $\pi_{i-1}<\pi_i>\pi_{i+1}$. Let $\des(\pi)$ and $\peak(\pi)$ denote the number of descents of $\pi$ and the number of peaks of $\pi$, respectively. We say $\pi$ is \dfn{alternating} if its descents are precisely the even elements of $[n-1]$. Let $\ALT$ be the set of alternating permutations. Let $\EDP$ denote the set of permutations in which every descent is a peak. 
\end{itemize}

\section{Trees, Permutations, and Valid Hook Configurations}\label{SecTreesEtc}

\subsection{Troupes}\label{Subsec:Troupes}
A \dfn{rooted plane tree} is a rooted tree in which the children of each vertex are linearly ordered from left to right. Such trees have been studied extensively in combinatorics and computer science \cite{Bender, Bona, BonaTrees, Flajolet, Gu, Hivert, Loday1, Loday2, Postnikov, Stanley}. We restrict our attention to \dfn{binary plane trees}; these are rooted plane trees in which each vertex has at most two children and every child is designated as either a left or a right child (but not both). Let $\mathsf{BPT}$ be the set of binary plane trees. The theory we will develop is quite general if we restrict our attention to the set $\mathsf{BPT}$, but we will obtain even more general results if we allow ourselves to color the vertices of trees. Throughout this article, we fix a finite set ${\bf C}$ of colors. It does not matter too much what the set ${\bf C}$ actually is, but we do want to require that it is finite and contains the colors black and white. A \dfn{colored binary plane tree} is a tree obtained from a binary plane tree by assigning the vertices colors from ${\bf C}$ (i.e., it is a binary plane tree along with a function from the set of vertices of the tree to ${\bf C}$). Let $\mathsf{CBPT}$ denote the set of colored binary plane trees. Given a set ${\bf T}\subseteq\mathsf{CBPT}$, we let ${\bf T}_n$ denote the set of all trees in ${\bf T}$ that have $n$ vertices. We make the convention that binary plane trees are just colored binary plane trees in which all of the vertices are black. Thus, $\mathsf{BPT}\subseteq\mathsf{CBPT}$. Our setup is able to model several different families of trees because we have the freedom to color the vertices in many different ways.

\begin{figure}[ht]
  \begin{center}{\includegraphics[height=1.3cm]{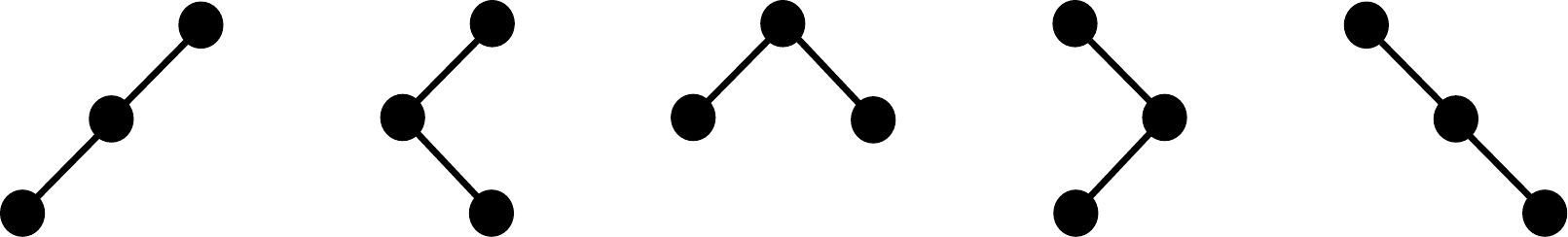}}
  \end{center}
  \caption{The $5$ binary plane trees with $3$ vertices.}\label{Fig2}
\end{figure}

We are going to describe two new operations defined on colored binary plane trees, which we call \emph{insertion} and \emph{decomposition}. These operations are very simple, but they will play a huge role in the remainder of the paper. 

To define insertion, suppose we are given two nonempty colored binary plane trees $T_1$ and $T_2$ along with a specific vertex $v$ in $T_1$. Replace $v$ with two vertices that are connected by a left edge. This produces a new tree $T_1^*$ with one more vertex than $T_1$. We call the lower endpoint of the new left edge $v$, identifying it with the original vertex $v$ and giving it the same color as the original $v$. We denote the upper endpoint of the new left edge by $v^*$, and we color $v^*$ black. For example, if \[T_1=\begin{array}{l}\includegraphics[height=1.3cm]{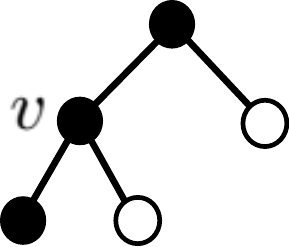}\end{array},\] where $v$ is as indicated, then \[T_1^*=\begin{array}{l}\includegraphics[height=1.9cm]{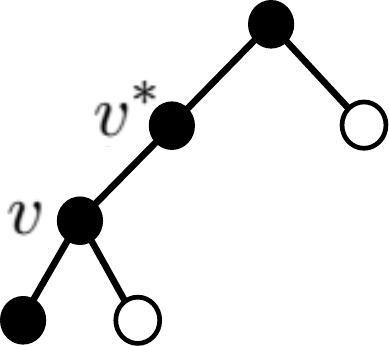}\end{array}.\] The \dfn{insertion} of $T_2$ into $T_1$ at $v$, denoted $\nabla_v(T_1,T_2)$, is the tree formed by attaching $T_2$ as the right subtree of $v^*$ in $T_1^*$. For example, if $T_1$ and $T_1^*$ are as above and \[T_2=\begin{array}{l}\includegraphics[height=1.3cm]{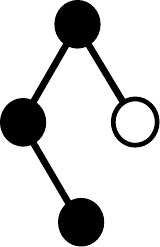}\end{array},\] then \[\nabla_v(T_1,T_2)=\begin{array}{l}\includegraphics[height=2.5cm]{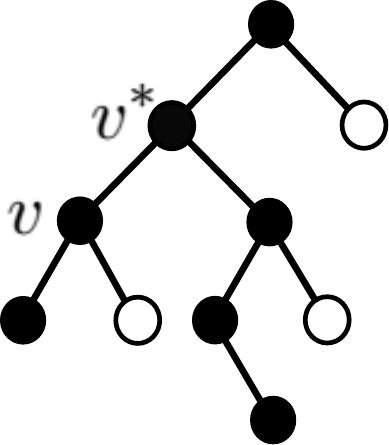}\end{array}.\] 

We can easily reverse the above procedure as follows. Let $T$ be a colored binary plane tree, and suppose $v^*$ is a black vertex in $T$ with $2$ children. Let $v$ be the left child of $v^*$ in $T$, and let $T_2$ be the right subtree of $v^*$ in $T$. Let $T_1^*$ be the tree obtained by deleting $T_2$ from $T$, and let $T_1$ be the tree obtained from $T_1^*$ by contracting the edge connecting $v$ and $v^*$ into a single vertex. We call this contracted vertex $v$, identifying it with the original $v$. We say the pair $(T_1,T_2)$ is the \dfn{decomposition} of $T$ at $v^*$ and write $\Delta_{v^*}(T)=(T_1,T_2)$.  

\begin{definition}\label{Def2}
We say a collection ${\bf T}$ of colored binary plane trees is 
\begin{itemize}
\item \dfn{insertion-closed} if for all nonempty trees $T_1,T_2\in{\bf T}$ and every vertex $v$ of $T_1$, the tree $\nabla_v(T_1,T_2)$ is in ${\bf T}$; 
\item \dfn{decomposition-closed} if for every $T\in{\bf T}$ and every black vertex $v^*$ of $T$ that has $2$ children, the pair $\Delta_{v^*}(T)$ is in ${\bf T}\times{\bf T}$; 
\item \dfn{black-peaked} if for every $T\in{\bf T}$, the vertices with $2$ children in $T$ are all black.
\end{itemize} A \dfn{troupe} is a set of colored binary plane trees that is insertion-closed, decomposition-closed, and black-peaked.\footnote{If we view insertion as analogous to a binary operation with decomposition as its inverse, then troupes are sets of trees that are analogous to groups.}
\end{definition}

Troupes are the sets to which our later theorems will apply. We now show that there are several troupes by giving a simple characterization of them. This characterization will not be needed in the remaining sections of the paper, so a reader primarily interested in cumulants and/or stack-sorting can safely skip to Example~\ref{ExamBinary}. Before proving the characterization, we need a little more terminology and a lemma. 

The \dfn{insertion closure} of a set ${\bf T}\subseteq\mathsf{CBPT}$, denoted $\InsCl({\bf T})$, is the smallest (under containment) insertion-closed subset of $\mathsf{CBPT}$ that contains ${\bf T}$. This is well-defined because the intersection of a collection of insertion-closed sets is insertion-closed. Note that $\InsCl({\bf T})$ is the set of trees obtained by starting with ${\bf T}$ and performing all possible sequences of insertions.    

\begin{lemma}\label{Lem3}
If a set ${\bf T}^{(0)}$ of colored binary plane trees is decomposition-closed and black-peaked, then its insertion closure $\InsCl({\bf T}^{(0)})$ is a troupe.  
\end{lemma}

\begin{proof}
Suppose ${\bf T}^{(0)}\subseteq\mathsf{CBPT}$ is decomposition-closed and black-peaked. Let ${\bf T}^{(1)}$ be the union of ${\bf T}^{(0)}$ with the set of all trees that can be written as $\nabla_v(T_1,T_2)$ for some nonempty $T_1,T_2\in{\bf T}^{(0)}$ and some vertex $v$ of $T_1$. It is clear that ${\bf T}^{(1)}$ is black-peaked because the new vertices with $2$ children that are produced from insertion are always black. We will show that ${\bf T}^{(1)}$ is also decomposition-closed. Choose $T\in{\bf T}^{(1)}$, and let $u^*$ be a (necessarily black) vertex in $T$ with $2$ children. We need to prove that $\Delta_{u^*}(T)\in{\bf T}^{(1)}\times{\bf T}^{(1)}$. If $T\in{\bf T}^{(0)}$, then this follows from the assumption that ${\bf T}^{(0)}$ is decomposition-closed. Thus, we may assume $T\in{\bf T}^{(1)}\setminus{\bf T}^{(0)}$. This means that there exist nonempty trees $T_1,T_2\in{\bf T}^{(0)}$ and a vertex $v$ of $T_1$ such that $T=\nabla_v(T_1,T_2)$. Let $v^*$ be the parent of $v$ in $T$. Let $u$ be the left child of $u^*$ in $T$. We consider three cases. 

\noindent{\bf Case 1.} Assume $u^*=v^*$. In this case, $\Delta_{u^*}(T)=\Delta_{v^*}(\nabla_v(T_1,T_2))=(T_1,T_2)\in{\bf T}^{(1)}\times{\bf T}^{(1)}$, as desired. 

\noindent{\bf Case 2.} Assume $u^*\neq v^*$ and $u^*\neq v$. Since $u^*\neq v^*$, the vertex $u^*$ is in either $T_1$ or $T_2$. We will assume $u^*$ is a vertex in $T_1$, the proof in the other case is completely analogous. Because $u^*$ is black and has $2$ children in $T$, it must also be black and have $2$ children in $T_1$. This means that we can decompose $T_1$ at $u^*$ to form the pair $\Delta_{u^*}(T_1)=(T_3,T_4)$. The trees $T_3$ and $T_4$ are in ${\bf T}^{(0)}$ because ${\bf T}^{(0)}$ is decomposition-closed and contains $T_1$. Because $u^*\neq v$, the vertex $v$ is in either $T_3$ or $T_4$. We assume $v$ is in $T_3$; the case in which $v$ is in $T_4$ is similar. Note that $u\neq v$ because $u^*\neq v^*$. It follows immediately from the definition of insertion that $\nabla_v(\nabla_u(T_3,T_4),T_2)=\nabla_u(\nabla_v(T_3,T_2),T_4)$. Therefore, \[\Delta_{u^*}(T)=\Delta_{u^*}(\nabla_v(T_1,T_2))=\Delta_{u^*}(\nabla_v(\nabla_u(T_3,T_4),T_2)
)=\Delta_{u^*}
(\nabla_u(\nabla_v(T_3,T_2),T_4))
\] \[=(\nabla_v(T_3,T_2),T_4).\] 
The trees $T_2$ and $T_3$ are in ${\bf T}^{(0)}$, so $\nabla_v(T_3,T_2)\in{\bf T}^{(1)}$. Since $T_4\in{\bf T}^{(0)}\subseteq{\bf T}^{(1)}$, this completes the proof in this case. 

\noindent{\bf Case 3.} Assume $u^*=v$. In this case, $u^*$ must be in $T_1$. As in the previous case, $u^*$ is black and has $2$ children in $T_1$, so we can write $\Delta_{u^*}(T_1)=(T_3,T_4)$ for some $T_3,T_4\in{\bf T}^{(0)}$. One can readily check that $\Delta_{u^*}(T)=(\nabla_u(T_3,T_2),T_4)$ (see Figure~\ref{Fig12} for an example). The trees $T_2$ and $T_3$ are in ${\bf T}^{(0)}$, so $\nabla_u(T_3,T_2)\in{\bf T}^{(1)}$. Since $T_4\in{\bf T}^{(0)}\subseteq{\bf T}^{(1)}$, this completes the proof in this final case. 

We have shown that ${\bf T}^{(1)}$ is decomposition-closed and black-peaked. Now let ${\bf T}^{(2)}$ be the union of ${\bf T}^{(1)}$ with the set of all trees that can be written as $\nabla_v(T_1,T_2)$ for some nonempty $T_1,T_2\in{\bf T}^{(1)}$ and some vertex $v$ of $T_1$. By the exact same argument as above, ${\bf T}^{(2)}$ is decomposition-closed and black-peaked. Repeating this construction, we obtain an infinite chain ${\bf T}^{(0)}\subseteq {\bf T}^{(1)}\subseteq\cdots$, where ${\bf T}^{(i+1)}$ is the union of ${\bf T}^{(i)}$ with the set of all trees that can be written as $\nabla_v(T_1,T_2)$ for some nonempty $T_1,T_2\in{\bf T}^{(i)}$ and some vertex $v$ of $T_1$. It follows by induction on $i$ that ${\bf T}^{(i)}$ is decomposition-closed and black-peaked for every nonnegative integer $i$. Therefore, $\bigcup_{i\geq 0}{\bf T}^{(i)}$ is decomposition-closed and black-peaked. It is straightforward to see that $\bigcup_{i\geq 0}{\bf T}^{(i)}=\InsCl({\bf T}^{(0)})$, so $\InsCl({\bf T}^{(0)})$ is decomposition-closed and black-peaked. The set $\InsCl({\bf T}^{(0)})$ is insertion-closed by definition, so it is a troupe.  
\end{proof}

\begin{figure}[ht]
  \begin{center}{\includegraphics[height=5.5cm]{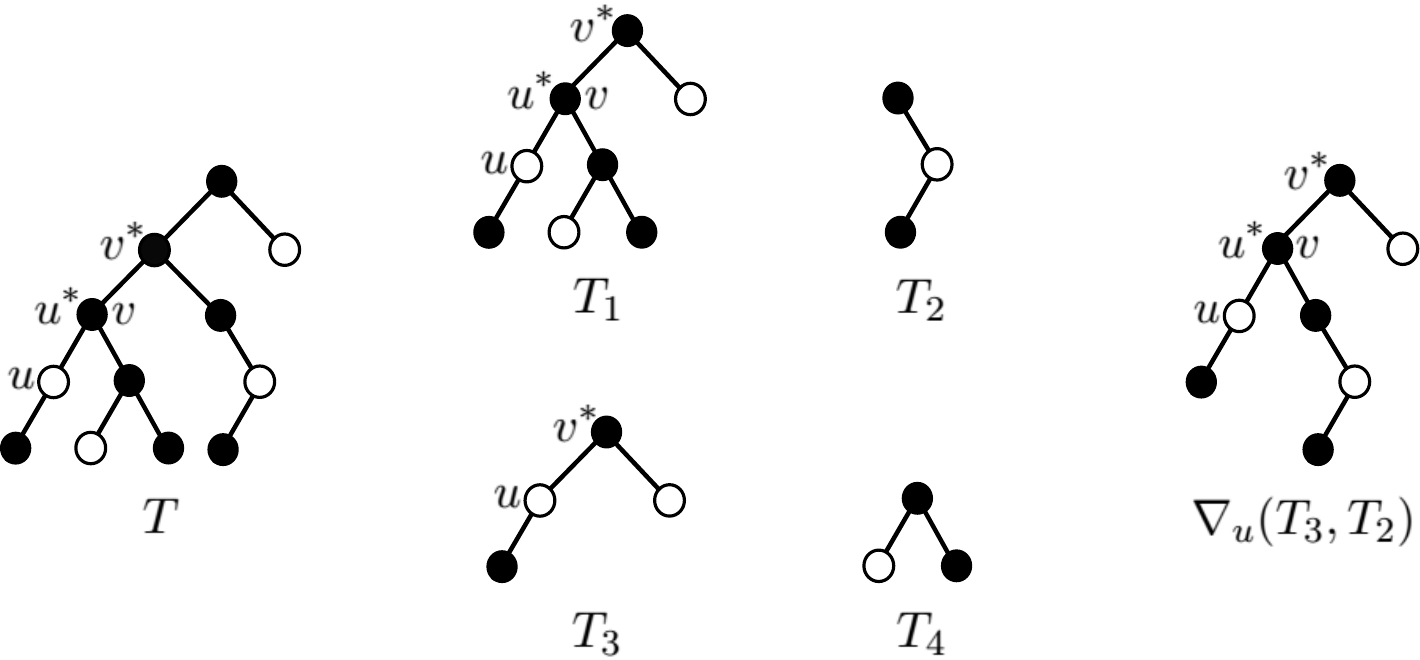}}
  \end{center}
  \caption{An illustration of trees appearing in Case 3 in the proof of Lemma~\ref{Lem3}.}\label{Fig12} 
\end{figure}

We say a colored binary plane tree $T$ is a \dfn{branch} if none of the vertices in $T$ have $2$ children. Let $\mathsf{Branch}$ denote the set of branches. 

\begin{theorem}\label{Thm19}
There is a bijective correspondence between the collection of all troupes and the collection of all sets of branches. Under this correspondence, a troupe ${\bf T}$ corresponds to ${\bf T}\cap\mathsf{Branch}$, and a set $B$ of branches corresponds to $\InsCl(B)$. 
\end{theorem}

\begin{proof}
Because branches do not have vertices with $2$ children, they cannot be decomposed. Therefore, every set of branches is vacuously decomposition-closed and black-peaked. Let ${\bf T}$ be a troupe. Since ${\bf T}\cap\mathsf{Branch}$ is decomposition-closed and black-peaked, it follows from Lemma~\ref{Lem3} that $\InsCl({\bf T}\cap\mathsf{Branch})$ is a troupe. Because ${\bf T}$ is insertion-closed and contains ${\bf T}\cap\mathsf{Branch}$, we must have $\InsCl({\bf T}\cap\mathsf{Branch})\subseteq{\bf T}$. We wish to prove the reverse containment. We will show that ${\bf T}_n\subseteq\InsCl({\bf T}\cap\mathsf{Branch})$ for every $n\geq 0$, where ${\bf T}_n$ is the set of trees in ${\bf T}$ with $n$ vertices. This is certainly true for $n\leq 2$ because every colored binary plane tree with at most $2$ vertices is a branch. Therefore, we may assume $n\geq 3$ and induct on $n$. Choose $T\in{\bf T}_n$. If $T$ is a branch, then $T\in\InsCl({\bf T}\cap\mathsf{Branch})$. Suppose $T$ is not a branch. This means there is a vertex $v^*$ of $T$ that has $2$ children. The vertex $v^*$ must be black because ${\bf T}$ is black-peaked. Let $(T_1,T_2)=\Delta_{v^*}(T)$, and let $v$ be the left child of $v^*$ in $T$. The trees $T_1$ and $T_2$ are in ${\bf T}$ because ${\bf T}$ is decomposition-closed. By induction on $n$, the trees $T_1$ and $T_2$ are in $\InsCl({\bf T}\cap\mathsf{Branch})$. Since this set is insertion-closed, it must contain the tree $\nabla_v(T_1,T_2)=T$. This completes the proof that ${\bf T}=\InsCl({\bf T}\cap\mathsf{Branch})$. 

Now let $B$ be a set of branches. It is clear that $B\subseteq\InsCl(B)\cap\mathsf{Branch}$. To prove the reverse containment, we use the fact that $\InsCl(B)$ is the set of trees obtained by starting with $B$ and performing all possible sequences of insertions. If a tree is in $\InsCl(B)\cap\mathsf{Branch}$, then it cannot be obtained from an insertion because it has no vertices with $2$ children. Therefore, every tree in $\InsCl(B)\cap\mathsf{Branch}$ must be in $B$. 
\end{proof}

In light of the preceding theorem, we define the \dfn{branch generators} of a troupe ${\bf T}$ to be the elements of ${\bf T}\cap\mathsf{Branch}$. One can think of the branch generators as the ``indecomposable'' elements of the troupe. Notice that Theorem~\ref{Thm19} implies that there are uncountably many troupes (there are even uncountably many troupes contained in $\mathsf{BPT}$). Our theorems in the rest of the paper will apply to all troupes, but our concrete examples will focus on the following four. 

\begin{example}[Binary Plane Trees]\label{ExamBinary}
The set $\mathsf{BPT}$ of all binary plane trees is certainly a troupe; its branch generators are the branches whose vertices are all black. It is well known that $|\mathsf{BPT}_n|=C_n$, where $C_n=\frac{1}{n+1}\binom{2n}{n}$ is the $n^\text{th}$ Catalan number. Therefore, 
\begin{equation}\label{Eq13}
\sum_{n\geq 0}|\mathsf{BPT}_n|z^n=\sum_{n\geq 0}C_nz^n=\frac{1-\sqrt{1-4z}}{2z}. \qedhere
\end{equation}
\end{example} 

\begin{example}[Full Binary Plane Trees]
We say a binary plane tree is \dfn{full} if every vertex has either $0$ or $2$ children. We also make the convention that the empty tree is not full. Let $\mathsf{FBPT}$ be the set of full binary plane trees, and let $\mathsf{FBPT}_n$ be the set of trees in $\mathsf{FBPT}$ with $n$ vertices. The trees in $\mathsf{FBPT}_7$ are shown in Figure~\ref{Fig3}. It is easy to see that $\mathsf{FBPT}$ is a troupe; its only branch generator is the tree consisting of a single black vertex. Every full binary plane tree has an odd number of vertices, and there is a natural bijection from $\mathsf{FBPT}_{2k+1}$ to $\mathsf{BPT}_k$ obtained by removing (also called \emph{pruning}) the leaves of the trees in $\mathsf{FBPT}_{2k+1}$. Therefore, $|\mathsf{FBPT}_{2k+1}|=C_k$. We have 
\begin{equation}\label{Eq14}
\sum_{n\geq 0}|\mathsf{FBPT}_n|z^n=\sum_{k\geq 0}|\mathsf{FBPT}_{2k+1}|z^{2k+1}=\sum_{k\geq 0}C_kz^{2k+1}=\frac{1-\sqrt{1-4z^2}}{2z}.
\end{equation} 

\begin{figure}[ht]
  \begin{center}{\includegraphics[height=1.9cm]{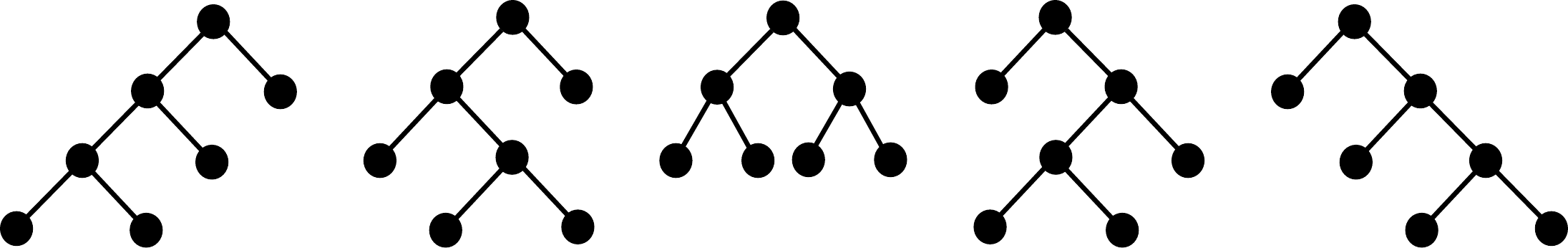}}
  \end{center}
  \caption{The $5$ full binary plane trees with $7$ vertices.}\label{Fig3} \qedhere
\end{figure}
\end{example} 

\begin{example}[Motzkin Trees]\label{ExamMotzkin2} 

A \dfn{Motzkin tree} (also called a \dfn{unary-binary tree}) is a nonempty binary plane tree in which every child that has no siblings is a left child. Let $\mathsf{Mot}$ be the set of Motzkin trees, and let $\mathsf{Mot}_n$ be the set of Motzkin trees with $n$ vertices. It is straightforward to check that $\mathsf{Mot}$ is a troupe; its branch generators are the nonempty branches that have only left edges and only black vertices.  

\begin{remark}\label{Rem8}
Motzkin trees are usually defined so that children without siblings are not designated as left or right children. Under this definition, Motzkin trees are not binary plane trees. However, there is an obvious bijection between Motzkin trees in the traditional sense and Motzkin trees as we have defined them above (just designate each child without siblings to be a left child). This allows us to apply our theorems to Motzkin trees as well.  
\end{remark}

The trees in $\mathsf{Mot}_4$ are depicted in Figure~\ref{Fig5}. It is well known that $|\mathsf{Mot}_n|=M_{n-1}$, where $M_n$ denotes the $n^\text{th}$ Motzkin number (with the convention $M_{-1}=0$). These numbers form the OEIS sequence A001006 \cite{OEIS}; they can be defined via their generating function $\displaystyle\sum_{n\geq 0}M_nz^n=\frac{1-z-\sqrt{1-2z-3z^2}}{2z^2}$. Hence, 
\begin{equation}\label{Eq15}
\sum_{n\geq 0}|\mathsf{Mot}_n|z^n=\sum_{n\geq 0}M_{n-1}z^n=\frac{1-z-\sqrt{1-2z-3z^2}}{2z}. 
\end{equation}

\begin{figure}[ht]
  \begin{center}{\includegraphics[height=1.9cm]{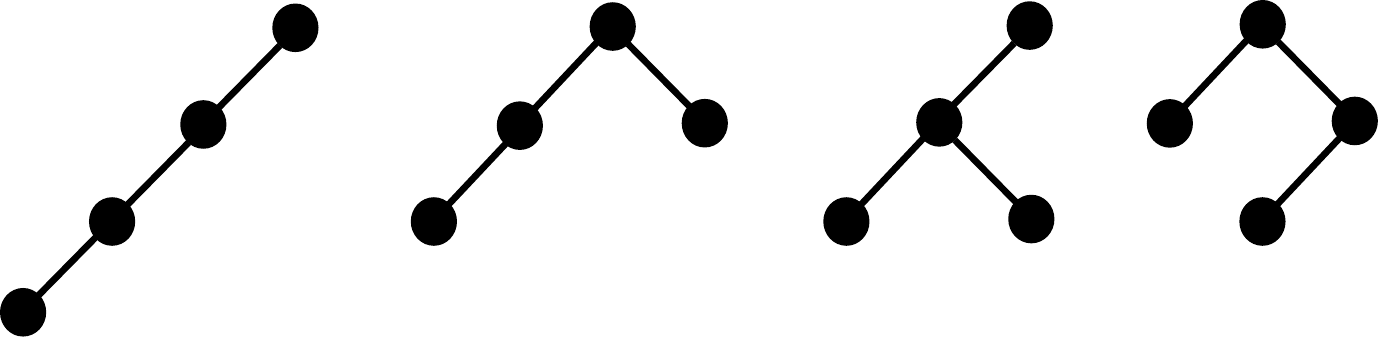}}
  \end{center}
  \caption{The $4$ Motzkin trees with $4$ vertices.}\label{Fig5} \qedhere
\end{figure}
\end{example} 

\begin{example}[Schr\"oder $2$-Colored Binary Trees]
A \dfn{$2$-colored binary tree} is a colored binary plane tree in which each vertex is colored either black or white. Of course, the number of such trees with $n$ vertices is $2^nC_n$. The focus of the article \cite{Gu} is the bijective enumeration of $2$-colored binary trees that satisfy certain constraints. For example, let $\mathsf{Sch}$ denote the set of $2$-colored binary trees in which no white vertex has a left child. Let $\mathsf{Sch}_n$ be the set of trees in $\mathsf{Sch}$ with $n$ vertices. The trees in $\mathsf{Sch}_2$ are shown in Figure~\ref{Fig6}. We call these trees \dfn{Schr\"oder $2$-colored binary trees} because (an equivalent reformulation of) the first part of Corollary 4.2 in \cite{Gu} states that $|\mathsf{Sch}_n|$ is the $n^\text{th}$ \dfn{large Schr\"oder number} $\mathscr S_n$. These numbers form OEIS sequence A006318 \cite{OEIS}. We have the generating function identity \[\sum_{n\geq 0}|\mathsf{Sch}_n|z^n=\sum_{n\geq 0}\mathscr S_nz^n=\dfrac{1-z-\sqrt{1-6z+z^2}}{2z}.\] The large Schr\"oder numbers also satisfy the identity $\mathscr S_n=\sum_{j=0}^n\binom{n+j}{n-j}C_j$. In fact, an equivalent reformulation of the second part of Corollary 4.2 in \cite{Gu} states that $\binom{n+j}{n-j}C_j$ is the number of trees in $\mathsf{Sch}_n$ with $j$ black vertices. 

\begin{figure}[ht]
  \begin{center}{\includegraphics[height=0.9cm]{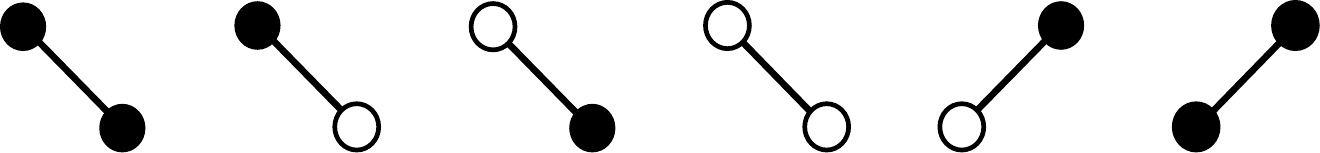}}
  \end{center}
  \caption{The $6$ Schr\"oder $2$-colored binary trees with $2$ vertices.}\label{Fig6} 
\end{figure}
The set $\mathsf{Sch}$ is a troupe; its set of branch generators is the set of branches that are $2$-colored binary trees in which no white vertices have left children. 
\end{example}

\subsection{Decreasing Colored Binary Plane Trees}

Let $X$ be a finite set of positive integers. A \dfn{labeled colored binary plane tree} on $X$ is a colored binary plane tree whose vertices are bijectively labeled with the elements of $X$. A \dfn{decreasing colored binary plane tree} is a labeled colored binary plane tree in which every nonroot vertex is given a label that is smaller than the label of its parent. A labeled colored binary plane tree with $n$ vertices is \dfn{standardized} if its set of labels is $[n]=\{1,\ldots,n\}$. 

\begin{definition}
The \dfn{skeleton} of a labeled colored binary plane tree $\mathcal T$ is the colored binary plane tree $\skel(\mathcal T)$ obtained by removing the labels from the vertices of $\mathcal T$. Given a set ${\bf T}$ of colored binary plane trees, let $\mathsf D{\bf T}$ denote the set of decreasing colored binary plane trees $\mathcal T$ such that $\skel(\mathcal T)\in{\bf T}$. Let $\mathsf{\overline D}{\bf T}$ be the set of standardized trees in $\mathsf{D}{\bf T}$. 
\end{definition}

The trees in $\mathsf{\overline{D}BPT}_3$ are depicted in Figure~\ref{Fig7}. Each colored binary plane tree represents a poset on its set of vertices in which $u<v$ if and only if $u$ is a descendant of $v$. From this point of view, a decreasing colored binary plane tree is a pair $(T,L)$, where $T$ is a colored binary plane tree and $L$ is a linear extension of the poset represented by $T$. 

\begin{figure}[ht]
  \begin{center}{\includegraphics[height=1.3cm]{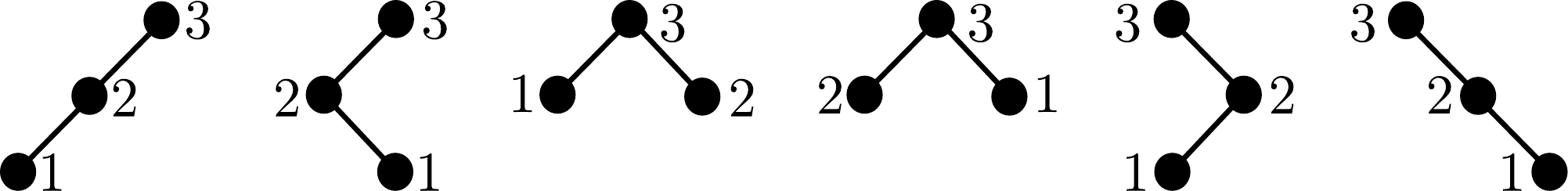}}
  \end{center}
  \caption{The $6$ standardized decreasing binary plane trees with $3$ vertices.}\label{Fig7}
\end{figure}

A \dfn{permutation} is an ordering of a finite set of positive integers; we write permutations as words in one-line notation. Let $S_n$ denote the set of permutations of $[n]$. Two natural tree traversals that produce a permutation of a set $X$ from a labeled colored binary plane tree on $X$ are the \dfn{in-order reading} $\mathcal I$ and the \dfn{postorder reading} $\mathcal P$. If $\mathcal T$ is the empty tree, then $\mathcal I(\mathcal T)$ and $\mathcal P(\mathcal T)$ are both just the empty permutation. Now suppose $\mathcal T$ is a nonempty labeled colored binary plane tree on $X$. Let $\mathcal T_L$ and $\mathcal T_R$ be the (possibly empty) left and right subtrees of the root of $\mathcal T$, respectively. Let $m\in X$ be the label of the root. The in-order reading and postorder reading are defined recursively by \[\mathcal I(\mathcal T)=\mathcal I(\mathcal T_L)\,m\,\mathcal I(\mathcal T_R)\quad\text{and}\quad\mathcal P(\mathcal T)=\mathcal P(\mathcal T_L)\,\mathcal P(\mathcal T_R)\,m.\]
For example, the in-order readings of the trees in Figure~\ref{Fig7} are, from left to right, $123$, $213$, $132$, $231$, $312$, $321$. The postorder readings of these trees are $123$, $123$, $123$, $213$, $123$, $123$. 

When dealing with in-order readings and postorder readings, we will focus exclusively on decreasing colored binary plane trees. It is well known that the in-order reading gives a bijection between decreasing binary plane trees on $X$ and permutations of $X$. Given a  permutation $\pi$, we let $\mathcal I^{-1}(\pi)$ denote the unique tree in $\mathsf{DBPT}$ whose in-order reading is $\pi$. On the other hand, it is clear from the trees in Figure~\ref{Fig7} that the postorder reading does not yield such a bijection. 

\subsection{Insertion-Additive Tree Statistics}
Our main results will allow us to work with enumerations of troupes that are refined according to tree statistics that interact nicely with insertion. 

\begin{definition}
A \dfn{tree statistic} is a function $f:\mathsf{CBPT}\to\mathbb C$. For every tree statistic $f$, we define a function $\ddot f:\mathsf D{\bf T}\to\mathbb C$ by \[\ddot f(\mathcal T)=f(\skel(\mathcal T)).\] We say a tree statistic $f$ is \dfn{insertion-additive} if for all nonempty colored binary plane trees $T_1$ and $T_2$ and all vertices $v$ in $T_1$, we have \[f(\nabla_v(T_1,T_2))=f(T_1)+f(T_2).\] 
\end{definition}

\begin{example}\label{Exam1}
The most obvious insertion-additive tree statistic is the statistic that maps each tree with $n$ vertices to the number $n+1$. We now describe some other natural statistics that are insertion-additive. One can show that if $\mathcal T$ is any decreasing colored binary plane tree, then the number of right edges in $\mathcal T$ is equal to the number of descents in the in-order reading $\mathcal I(\mathcal T)$. Furthermore, the number of vertices in $\mathcal T$ that have $2$ children is equal to the number of peaks of $\mathcal I(\mathcal T)$ (this explains the term ``black-peaked'' from Definition~\ref{Def2}). Therefore, it makes sense to define $\des(\mathcal T)$ and $\peak(\mathcal T)$ to be the number of right edges in $\mathcal T$ and the number of vertices with $2$ children in $\mathcal T$, respectively. Since $\des(\mathcal T)$ and $\peak(\mathcal T)$ only depend on the skeleton of $\mathcal T$, the maps $\des$ and $\peak$ descend to tree statistics on colored binary plane trees. For every colored binary plane tree $T$, we define $\des(T)$ and $\peak(T)$ to be the number of right edges in $T$ and the number of vertices with $2$ children in $T$, respectively. Thus, $\des(\mathcal T)=\des(\skel(\mathcal T))$ and $\peak(\mathcal T)=\peak(\skel(\mathcal T))$ for every decreasing colored binary plane tree $\mathcal T$. We make the convention that $\des(\varepsilon)=0$ and $\peak(\varepsilon)=-1$, where $\varepsilon$ is the empty tree. 

Suppose $T_1,T_2\in\mathsf{CBPT}$ are nonempty, and let $v$ be a vertex in $T_1$. The right edges in $\nabla_v(T_1,T_2)$ are the right edges in $T_1$, the right edges in $T_2$, and the new right edge connecting the new vertex $v^*$ to the root of $T_2$. Therefore, $\des(\nabla_v(T_1,T_2))=\des(T_1)+\des(T_2)+1$. This shows that the function $T\mapsto\des(T)+1$ is an insertion-additive tree statistic. The vertices with $2$ children in $\nabla_v(T_1,T_2)$ are the vertices with $2$ children in $T_1$, the vertices with $2$ children in $T_2$, and the new vertex $v^*$. Consequently, the function $T\mapsto\peak(T)+1$ is insertion-additive. Another natural insertion-additive tree statistic is the map $T\mapsto\black(T)+1$, where $\black(T)$ is the number of black vertices in $T$. Indeed, the black vertices in $\nabla_v(T_1,T_2)$ are the black vertices in $T_1$, the black vertices in $T_2$, and the new vertex $v^*$. If $c$ is any color other than black, then the statistic that maps $T$ to the number of vertices with color $c$ in $T$ is insertion-additive. 
\end{example}

\subsection{The Stack-Sorting Map}\label{Sec:Stack-Sorting}
In his book \emph{The Art of Computer Programming}, Knuth introduced a ``stack-sorting algorithm'' \cite{Knuth}; his analysis of this algorithm led to several advances in combinatorics, including the notion of a permutation pattern and the kernel method \cite{Banderier, Bona, Kitaev, Linton}. In his Ph.D. thesis, West defined a deterministic variant of Knuth's algorithm \cite{West}. This variant is a function that we denote by $s$ and call the \dfn{stack-sorting map}. Despite the fact that this function is so easy to define, it is remarkably difficult to analyze. Consequently, it has received a great amount of attention since its inception (see \cite{Bona, BonaSurvey, DefantCounting, DefantCatalan, DefantEngenMiller} and the references therein). 

To define $s$, assume we are given an input permutation $\pi=\pi_1\cdots\pi_n$. Throughout this procedure, if the next entry in the input permutation is smaller than the entry at the top of the stack or if the stack is empty, the next entry in the input permutation is placed at the top of the stack. Otherwise, the entry at the top of the stack is annexed to the end of the growing output permutation. This procedure stops when the output permutation has length $n$. We then define $s(\pi)$ to be this output permutation. Figure~\ref{Fig1} illustrates this procedure and shows that $s(4162)=1426$.  

\begin{figure}[ht]
\begin{center}
\includegraphics[width=1\linewidth]{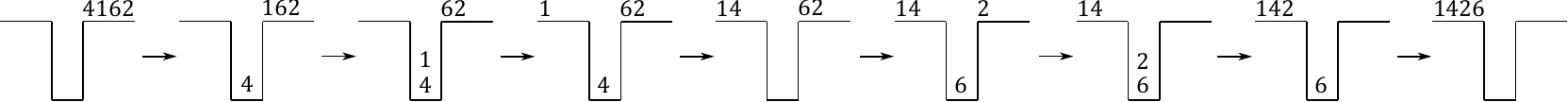}
\end{center}  
\caption{The stack-sorting map $s$ sends $4162$ to $1426$.}
\end{figure}\label{Fig1}

If $\pi$ is a permutation with largest entry $m$, then we can write $\pi=LmR$ for some permutations $L$ and $R$. A useful recursive description of the stack-sorting map states that 
\begin{equation}\label{Eq36}
s(\pi)=s(L)s(R)m.
\end{equation} For example, \[s(416352)=s(41)\,s(352)\,6=s(1)\,4\,s(3)\,s(2)\,56=143256.\]

There is another alternative definition of the stack-sorting map that makes use of in-order and postorder readings of binary plane trees. Namely, 
\begin{equation}\label{Eq11}
s=\mathcal P\circ \mathcal I^{-1}.
\end{equation} For example, we have \[246153\xrightarrow{\mathcal I^{-1}}\begin{array}{l}\includegraphics[height=1.3cm]{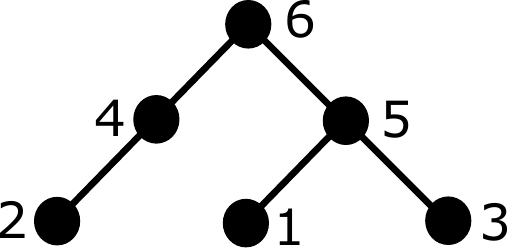}\end{array}\xrightarrow{\,\,\,\mathcal P\,\,\,}241356,\] and one can check that $s(246153)=241356$. It is this alternative definition that will allow us to apply free probability theory to answer some (otherwise very difficult) natural questions about $s$ in Section~\ref{Sec:Applications}. 

West \cite{West} defined the \dfn{fertility} of a permutation $\pi$ to be $|s^{-1}(\pi)|$. According to \eqref{Eq11}, 
\begin{equation}\label{Eq12}
|s^{-1}(\pi)|=|\mathcal P^{-1}(\pi)\cap\mathsf{DBPT}|.
\end{equation}
In other words, the fertility of $\pi$ is the number of decreasing binary plane trees with postorder $\pi$. \emph{A priori}, computing fertilities of permutations is a difficult task. Indeed, West devoted ten pages of his dissertation \cite{West} to the computation of the fertilities of permutations of the forms \[23\cdots k1(k+1)\cdots n,\quad 12\cdots(k-2)k(k-1)(k+1)\cdots n,\quad\text{and}\quad k12\cdots(k-1)(k+1)\cdots n.\] Bousquet-M\'elou \cite{Bousquet} defined a permutation to be \dfn{sorted} if its fertility is positive (i.e., it is in the image of $s$). She found an algorithm that determines whether or not a given permutation is sorted. She then asked for a general method for computing the fertility of any given permutation.

The current author has found two distinct yet related methods for computing fertilities of arbitrary permutations. The first method, which we call the Fertility Formula, was developed in \cite{DefantPostorder}. The Fertility Formula expresses $|s^{-1}(\pi)|$ as a sum of products of Catalan numbers, where the sum runs over combinatorial objects called \emph{valid hook configurations}. The second method is the Decomposition Lemma, a recursive formula that was first proven in \cite{DefantCounting}. We are going to give a new proof of the Decomposition Lemma; the new proof makes use of decreasing binary plane trees and, consequently, is more conceptual than the original purely permutation-based proof. 

The definition of a valid hook configuration appears complicated at first glance, and the original proof of the Fertility Formula in \cite{DefantPostorder} seems very ad hoc. Our new proof of the Decomposition Lemma yields a more general result that allows us to rederive the Fertility Formula. This new derivation is much cleaner than the original proof, and it explains why valid hook configurations are defined the way they are. This new proof of the Decomposition Lemma also generalizes to arbitrary troupes; hence, we will call this more general result the Tree Decomposition Lemma. From the Tree Decomposition Lemma, we will derive a new Tree Fertility Formula, which will be one of the two main tools allowing us to connect free and classical cumulants with families of colored binary plane trees and decreasing colored binary plane trees in Section~\ref{Sec:TroupsAndCumulants}. In fact, we can even refine the Tree Decomposition Lemma and Tree Fertility Formula by taking into account certain tree statistics; we call these more general results the Refined Tree Decomposition Lemma and the Refined Tree Fertility Formula.  

\section{The Refined Tree Decomposition Lemma} \label{SecRefinedDecomposition}

The \dfn{plot} of a permutation $\pi=\pi_1\cdots\pi_n$ is the diagram showing the points $(i,\pi_i)\in\mathbb R^2$ for all $1\leq i\leq n$. A \dfn{hook} of $\pi$ is a rotated L shape connecting two points $(i,\pi_i)$ and $(j,\pi_j)$ with $i<j$ and $\pi_i<\pi_j$, as in Figure~\ref{Fig4}. The point $(i,\pi_i)$ is the \dfn{southwest endpoint} of the hook, and $(j,\pi_j)$ is the \dfn{northeast endpoint} of the hook. Let $\SW_i(\pi)$ be the set of hooks of $\pi$ with southwest endpoint $(i,\pi_i)$. For example, Figure~\ref{Fig4} shows the plot of the permutation $\pi=426315789$. The hook shown in this figure is in $\SW_3(\pi)$ because its southwest endpoint is $(3,6)$. Its northeast endpoint is $(8,8)$. 

\begin{figure}[ht]
  \begin{center}{\includegraphics[width=0.22\textwidth]{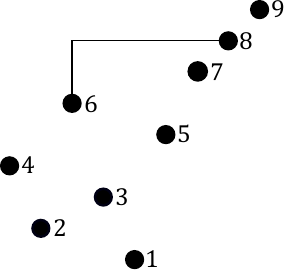}}
  \end{center}
  \caption{The plot of $426315789$ along with a single hook.}\label{Fig4}
\end{figure}

Suppose $\pi$ is not a monotone increasing permutation, and let $d_1<\cdots<d_k$ be its descents. The \dfn{rightmost ascending run} of the plot of $\pi$ is the sequence of points $(d_k+1,\pi_{d_k+1}),\ldots,(n,\pi_n)$. For example, the rightmost ascending run of $426315789$ is $(5,1), (6,5), (7,7), (8,8), (9,9)$. We say a descent $d$ of $\pi$ is \dfn{right-bound}\footnote{In the hypothesis of the Refined Decomposition Lemma in \cite{DefantCounting}, one chooses a ``tail-bound descent'' of a permutation. By contrast, our Refined Tree Decomposition Lemma will require the choice of a right-bound descent. Every tail-bound descent is right-bound, so this is one of the two ways in which the Refined Tree Decomposition Lemma generalizes the Refined Decomposition Lemma (the other is that it applies to troupes).} if every hook in $\SW_d(\pi)$ has its northeast endpoint in the rightmost ascending run of the plot of $\pi$. This is equivalent to saying that every entry in $\pi$ that is greater than $\pi_d$ and to the right of $\pi_d$ is also to the right of $\pi_{d_k}$. Of course, this means that $d_k$ is automatically a right-bound descent. The descents of $426315789$ are $1$, $3$, and $4$, but the only right-bound descents are $3$ and $4$. 

Let $H$ be a hook of $\pi$ with southwest endpoint $(i,\pi_i)$ and northeast endpoint $(j,\pi_j)$. The \dfn{$H$-unsheltered subpermutation} of $\pi$ is the permutation $\pi_U^H=\pi_1\cdots\pi_i\pi_{j+1}\cdots\pi_n$. Similarly, the \dfn{$H$-sheltered subpermutation} of $\pi$ is $\pi_S^H=\pi_{i+1}\cdots\pi_{j-1}$. For instance, if $\pi=426315789$ and $H$ is the hook shown in Figure~\ref{Fig4}, then $\pi_U^H=4269$ and $\pi_S^H=3157$. The terms ``sheltered'' and ``unsheltered'' come from the fact that, in applications, the plot of $\pi_S^H$ will lie entirely below the hook $H$. In particular, this will be the case if $i$ is a right-bound descent of $\pi$.  

\begin{theorem}[Refined Tree Decomposition Lemma]\label{Thm17}
Let ${\bf T}$ be a troupe. Let $f_1,\ldots,f_r$ be insertion-additive tree statistics, and let $x_1,\ldots,x_r$ be variables. If $d$ is a right-bound descent of a nonempty permutation $\pi$, then \[\sum_{\mathcal T\in \mathcal P^{-1}(\pi)\cap\mathsf D{\bf T}}x_1^{\ddot f_1(\mathcal T)}\cdots x_r^{\ddot f_r(\mathcal T)}\] \[=\sum_{H\in\SW_d(\pi)}\left(\sum_{\mathcal T_U\in \mathcal P^{-1}(\pi_U^H)\cap\mathsf D{\bf T}}x_1^{\ddot f_1(\mathcal T_U)}\cdots x_r^{\ddot f_r(\mathcal T_U)}\right)\left(\sum_{\mathcal T_S\in \mathcal P^{-1}(\pi_S^H)\cap\mathsf D{\bf T}}x_1^{\ddot f_1(\mathcal T_S)}\cdots x_r^{\ddot f_r(\mathcal T_S)}\right).\]
\end{theorem}

\begin{proof}
Let $\pi=\pi_1\cdots\pi_n$. Let $X$ be the set of entries in $\pi$, and let $m=\max(X)$. If $\SW_d(\pi)$ is empty, then $\pi_n\neq m$. In this case, $\mathcal P^{-1}(\pi)\cap\mathsf{D}{\bf T}$ is empty, so both sides of the desired equation are $0$. Thus, we may assume $\SW_d(\pi)$ is nonempty. This implies that $\pi_d<m$. 

Suppose $\mathcal T\in \mathcal P^{-1}(\pi)\cap\mathsf{D}{\bf T}$, and let $v$ be the vertex of $\mathcal T$ with the label $\pi_d$. Because $\pi_d<m$, the vertex $v$ is not the root of $\mathcal T$. Thus, $v$ has a parent $v^*$. Let $\pi_j$ be the label of $v^*$. Note that $\pi_d<\pi_j$ because $\mathcal T$ is decreasing. Because $d$ is a descent of $\pi$, we know that $j\neq d+1$. It follows from the fact that $\pi=\mathcal P(\mathcal T)$ that $d+1<j$, that $v$ is a left child of $v^*$, and that $v^*$ has a nonempty right subtree $\mathcal T_S$. Furthermore, $(j,\pi_j)$ is the northeast endpoint of a hook $H\in \SW_d(\pi)$. Let $\mathcal T_U^*$ be the tree obtained from $\mathcal T$ by deleting $\mathcal T_S$. There is a left edge in $\mathcal T_U^*$ connecting $v$ to its parent $v^*$; contract this edge into a single vertex $v$ (with the same color as the original $v$), and give this vertex the label $\pi_d$. This produces a decreasing colored binary plane tree $\mathcal T_U$. 

The procedure we just described for producing the decreasing colored binary plane trees $\mathcal T_U$ and $\mathcal T_S$ is exactly the same as the procedure used to decompose $\skel(\mathcal T)$ at the vertex $v^*$ (note that $v^*$ is black because it is a vertex with $2$ children in the tree $\skel(\mathcal T)$, which is in the troupe ${\bf T}$). Hence, \[\Delta_{v^*}(\skel(\mathcal T))=(\skel(\mathcal T_U),\skel(\mathcal T_S)).\] The assumption that $\mathcal T\in\mathsf D{\bf T}$ tells us that $\skel(\mathcal T)\in{\bf T}$. Because the troupe ${\bf T}$ is decomposition-closed, the trees $\skel(\mathcal T_U)$ and $\skel(\mathcal T_S)$ are also in ${\bf T}$. Thus, $\mathcal T_U$ and $\mathcal T_S$ are in $\mathsf D{\bf T}$. Comparing the above decomposition procedure with the definition of the postorder reading, we find that $\mathcal P(\mathcal T_U)=\pi_U^H$ and $\mathcal P(\mathcal T_S)=\pi_S^H$. Consequently, $\mathcal T_U\in \mathcal P^{-1}(\pi_U^H)\cap\mathsf D{\bf T}$ and $\mathcal T_S\in \mathcal P^{-1}(\pi_S^H)\cap\mathsf D{\bf T}$.

This process is reversible. Suppose we are given the hook $H\in\SW_d(\pi)$ with northeast endpoint $(j,\pi_j)$ along with the trees $\mathcal T_U\in \mathcal P^{-1}(\pi_U^H)\cap\mathsf D{\bf T}$ and $\mathcal T_S\in \mathcal P^{-1}(\pi_S^H)\cap\mathsf D{\bf T}$. Let $v$ be the vertex in $\mathcal T_U$ with label $\pi_d$. The trees $\skel(\mathcal T_U)$ and $\skel(\mathcal T_S)$ are in ${\bf T}$. Because ${\bf T}$ is a troupe, $\nabla_v(\skel(\mathcal T_U),\skel(\mathcal T_S))\in{\bf T}$. Let $v^*$ be the parent of $v$ in $\nabla_v(\skel(\mathcal T_U),\skel(\mathcal T_S))$. Every vertex in $\nabla_v(\skel(\mathcal T_U),\skel(\mathcal T_S))$ other than $v^*$ is a vertex in either $\mathcal T_U$ or $\mathcal T_S$. Let us give $v^*$ the label $\pi_j$ and give each of the other vertices in $\nabla_v(\skel(\mathcal T_U),\skel(\mathcal T_S))$ the same label that it has in either $\mathcal T_U$ or $\mathcal T_S$. This produces a new labeled colored binary plane tree $\mathcal T$ satisfying $\skel(\mathcal T)=\nabla_v(\skel(\mathcal T_U),\skel(\mathcal T_S))$. Note that $\mathcal T_S$ is the right subtree of $v^*$ in $\mathcal T$. Let $\mathcal T_U^*$ be the tree obtained by deleting $\mathcal T_S$ from $\mathcal T$. Our goal is to show that $\mathcal T$ is a decreasing colored binary plane tree. In order to do this, it suffices to show that $\mathcal T_S$ and $\mathcal T_U^*$ are decreasing and that the label of the root of $\mathcal T_S$ is less than the label $\pi_j$ of $v^*$. 

We are given that $\mathcal T_U$ and $\mathcal T_S$ are decreasing colored binary plane trees; we want to check that $\mathcal T_U^*$ is also. The only thing we need to verify is that the label $\pi_j$ of $v^*$ is smaller than the label of the parent of $v^*$ in $\mathcal T_U^*$ if $v^*$ has a parent. Suppose $v^*$ does have a parent $u$ in $\mathcal T_U^*$, and let $a$ be its label. Since $u$ is the parent of $v$ in $\mathcal T_U$, we know that $a$ is greater than $\pi_d$ and appears to the right of $\pi_d$ in the permutation $\mathcal P(\mathcal T_U)=\pi_U^H$. Every entry that appears to the right of $\pi_d$ in $\pi_U^H$ must appear to the right of $\pi_j$ in $\pi$. It is at this point that we use the crucial hypothesis that $d$ is a right-bound descent. Indeed, this hypothesis implies that $\pi_j$ is in the rightmost ascending run of $\pi$, so it is less than every entry that appears to its right in $\pi$. In particular, this means that $\pi_j<a$, as desired. 

It now remains to check that the label of the root of $\mathcal T_S$ is less than $\pi_j$. This label is an entry in the permutation $\pi_S^H$, whose plot lies below $H$ in the plot of $\pi$ (because $d$ is right-bound). Consequently, this label is less than $\pi_j$. It follows from our construction that $\mathcal P(\mathcal T)=\pi$, so we have shown that $\mathcal T\in \mathcal P^{-1}(\pi)\cap\mathsf D{\bf T}$. 

The above argument shows that there is a bijection \[\varphi:\bigcup_{H\in\SW_d(\pi)}(\mathcal P^{-1}(\pi_U^H)\cap\mathsf{D}{\bf T})\times(\mathcal P^{-1}(\pi_S^H)\cap\mathsf{D}{\bf T})\to \mathcal P^{-1}(\pi)\cap\mathsf{D}{\bf T}\] such that \[\skel(\varphi(\mathcal T_U,\mathcal T_S))=\nabla_v(\skel(\mathcal T_U),\skel(\mathcal T_S)),\] where $v$ is the vertex with label $\pi_d$. For each $i\in[r]$, we have \[\ddot f_i(\varphi(\mathcal T_U,\mathcal T_S))=f_i(\skel(\varphi(\mathcal T_U,\mathcal T_S)))=f_i(\nabla_v(\skel(\mathcal T_U),\skel(\mathcal T_S)))=f_i(\skel(\mathcal T_U))+f_i(\skel(\mathcal T_S))\] \[=\ddot f_i(\mathcal T_U)+\ddot f_i(\mathcal T_S),\] where we have used the definition of $\ddot f_i$ and the assumption that $f_i$ is insertion-additive. This completes the proof.  
\end{proof}

\begin{example}
We illustrate the proof of Theorem~\ref{Thm17} in the specific case in which ${\bf T}$ is the troupe $\mathsf{BPT}$ and $\pi=426315789$ is the permutation shown in Figure~\ref{Fig4}. Let $d=3$ (so $\pi_d=6$). One element of $\mathcal P^{-1}(\pi)\cap\mathsf{DBPT}$ is the tree
\[\mathcal T=\begin{array}{l}\includegraphics[height=2.3cm]{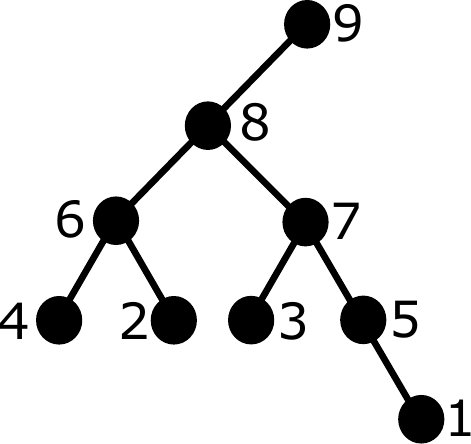}\end{array}.\] Here, $v$ is the vertex with label $6$. The parent $v^*$ of $v$ has label $\pi_j=8$, so $H$ is the hook shown in Figure~\ref{Fig4}. The corresponding unsheltered and sheltered subpermutations are $\pi_U^H=4269$ and $\pi_S^H=3157$. Now observe that \[\mathcal T_U=\begin{array}{l}\includegraphics[height=1.3cm]{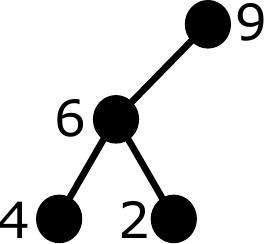}\end{array}\in\mathcal P^{-1}(\pi_U^H)\cap\mathsf{DBPT}\quad\text{and}\quad \mathcal T_S=\begin{array}{l}\includegraphics[height=1.35cm]{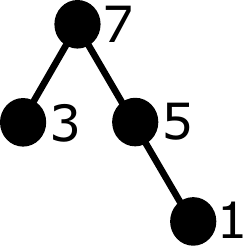}\end{array}\in\mathcal P^{-1}(\pi_S^H)\cap\mathsf{DBPT}. \qedhere\]   
\end{example}

\begin{corollary}[Tree Decomposition Lemma]\label{Cor1}
Let ${\bf T}$ be a troupe. If $d$ is a right-bound descent of a nonempty permutation $\pi$, then \[|\mathcal P^{-1}(\pi)\cap\mathsf D{\bf T}|=\sum_{H\in\SW_d(\pi)}|\mathcal P^{-1}(\pi_U^H)\cap\mathsf D{\bf T}|\cdot|\mathcal P^{-1}(\pi_S^H)\cap\mathsf D{\bf T}|.\]
\end{corollary}

\begin{proof}
Set $x_1=\cdots=x_r=1$ in Theorem~\ref{Thm17}. 
\end{proof}

The following corollary is a mild generalization  of the Refined Decomposition Lemma from \cite{DefantCounting} (the original formulation used ``tail-bound descents,'' which are less general than right-bound descents). Recall that $s$ denotes the stack-sorting map. 

\begin{corollary}[Refined Decomposition Lemma]\label{Cor2}
If $d$ is a right-bound descent of a nonempty permutation $\pi$, then \[\sum_{\sigma\in s^{-1}(\pi)}x_1^{\des(\sigma)+1}x_2^{\peak(\sigma)+1}\] \[=\sum_{H\in\SW_d(\pi)}\left(\sum_{\mu\in s^{-1}(\pi_U^H)}x_1^{\des(\mu)+1}x_2^{\peak(\mu)+1}\right)\left(\sum_{\lambda\in s^{-1}(\pi_S^H)}x_1^{\des(\lambda)+1}x_2^{\peak(\lambda)+1}\right).\]
\end{corollary}

\begin{proof}
We can apply Theorem~\ref{Thm17} with ${\bf T}=\mathsf{BPT}$. It follows from \eqref{Eq11} that for every permutation $\tau$, the in-order reading $\mathcal I:\mathcal P^{-1}(\tau)\cap\mathsf{DBPT}\to s^{-1}(\tau)$ is a bijection. Let $f_1$ and $f_2$ be the tree statistics given by $f_1(T)=\des(T)+1$ and $f_2(T)=\peak(T)+1$, as defined in Example~\ref{Exam1}. As mentioned in that example, we have $\ddot f_1(\mathcal T)=\des(\mathcal I(\mathcal T))+1$ and $\ddot f_2(\mathcal T)=\peak(\mathcal I(\mathcal T))+1$ for every $\mathcal T\in\mathsf{DBPT}$. The proof now follows from Theorem~\ref{Thm17}. 
\end{proof}

\begin{corollary}[Decomposition Lemma]\label{Cor3}
If $d$ is a right-bound descent of a nonempty permutation $\pi$, then \[|s^{-1}(\pi)|=\sum_{H\in\SW_d(\pi)}|s^{-1}(\pi_U^H)|\cdot|s^{-1}(\pi_S^H)|.\] 
\end{corollary}

\begin{proof}
Set $x_1=x_2=1$ in Corollary~\ref{Cor2}. 
\end{proof}

\begin{remark}
One of the useful aspects of the Refined Tree Decomposition Lemma and its corollaries is that they give us the freedom to choose \emph{any} right-bound descent $d$ of the permutation $\pi$. In the applications of the Decomposition Lemma in \cite{DefantCounting, DefantEnumeration}, it is most convenient to choose a right-bound descent $d$ such that $(d,\pi_d)$ is the highest point in the plot of $\pi$ that does not lie in the ``tail'' of $\pi$. By contrast, we will find it useful in the next section to take $d$ to be the largest descent of $\pi$. 
\end{remark}

\section{The Refined Tree Fertility Formula}\label{Sec:TreeFertility}
Let us fix a troupe ${\bf T}$, and let ${\bf T}_n$ denote the set of trees in ${\bf T}$ with $n$ vertices. Let us also fix insertion-additive tree statistics $f_1,\ldots,f_r$ and variables $x_1,\ldots,x_r$. Let \[{\bf G}_n(x_1,\ldots,x_r)=\sum_{T\in{\bf T}_n}x_1^{f_1(T)}\cdots x_r^{f_r(T)}.\] 
In many applications, the polynomials ${\bf G}_n(x_1,\ldots,x_r)$ can be computed using standard combinatorial methods. The basic idea of this section is to iteratively apply the Refined Tree Decomposition Lemma in order to express \[\sum_{\mathcal T\in \mathcal P^{-1}(\pi)\cap\mathsf D{\bf T}}x_1^{\ddot f_1(\mathcal T)}\cdots x_r^{\ddot f_r(\mathcal T)}\] as a sum of products of polynomials ${\bf G}_{q_t}(x_1,\ldots,x_r)$, where the sum ranges over valid hook configurations. A different sum over valid hook configurations will appear in the VHC Cumulant Formula, allowing us to connect cumulants with trees. We begin by handling the case in which $\pi$ is an increasing permutation. 

\begin{lemma}\label{Lem2}
If $\pi$ is an increasing permutation of length $n$, then \[\sum_{\mathcal T\in \mathcal P^{-1}(\pi)\cap\mathsf D{\bf T}}x_1^{\ddot f_1(\mathcal T)}\cdots x_r^{\ddot f_r(\mathcal T)}={\bf G}_n(x_1,\ldots,x_r).\] 
\end{lemma} 

\begin{proof}
If $\pi$ is any permutation (not necessarily increasing), then the map $\skel:\mathcal P^{-1}(\pi)\cap\mathsf D{\bf T}\to {\bf T}_n$ is an injection satisfying $\ddot f_i(\mathcal T)=f_i(\skel(\mathcal T))$ for all $i\in[r]$ and $\mathcal T\in\mathcal P^{-1}(\pi)\cap\mathsf D{\bf T}$. The lemma follows from the observation that this map is bijective when $\pi$ is increasing. Indeed, suppose $T\in{\bf T}_n$. There is a unique way to label the vertices of $T$ to obtain a labeled colored binary plane tree $\mathcal T$ with postorder reading $\pi$. Using the fact that $\pi$ is increasing, one can straightforwardly show that $\mathcal T$ is in fact decreasing. 
\end{proof}

\noindent {\bf Iterative Decomposition Procedure.} 
Suppose $\pi$ is a permutation with descents $d_1<\cdots< d_k$, where $k=\des(\pi)\geq 1$. The descent $d_k$ is necessarily right-bound. Therefore, the Tree Decomposition Lemma (Corollary~\ref{Cor1}) tells us that choosing a tree in $\mathcal P^{-1}(\pi)\cap\mathsf D{\bf T}$ is equivalent to choosing a hook $H_k\in\SW_{d_k}(\pi)$ and then choosing trees $\mathcal T_U^{(k)}\in \mathcal P^{-1}(\pi_U^{H_k})\cap\mathsf D{\bf T}$ and $\mathcal T_S^{(k)}\in \mathcal P^{-1}(\pi_S^{H_k})\cap\mathsf D{\bf T}$. For notational convenience, let $\pi=\pi^{(k)}$ and $\pi^{(k-1)}=\pi_U^{H_k}=(\pi^{(k)})_U^{H_k}$. Notice that if $k\geq 2$, then $d_{k-1}$ is the largest descent of $\pi^{(k-1)}$ (in particular, it is right-bound). This means that we can invoke the Tree Decomposition Lemma once again to see that choosing $\mathcal T_U^{(k)}$ is equivalent to choosing a hook $H_{k-1}\in \SW_{d_{k-1}}(\pi^{(k-1)})$ and then choosing trees $\mathcal T_U^{(k-1)}\in \mathcal P^{-1}((\pi^{(k-1)})_U^{H_{k-1}})\cap\mathsf D{\bf T}$ and $\mathcal T_S^{(k-1)}\in \mathcal P^{-1}((\pi^{(k-1)})_S^{H_{k-1}})\cap\mathsf D{\bf T}$. Let $\pi^{(k-2)}=(\pi^{(k-1)})_U^{H_{k-1}}$. We can repeat this process by always choosing $H_i\in\SW_{d_i}(\pi^{(i)})$ and setting $\pi^{(i-1)}=(\pi^{(i)})_U^{H_i}$. We do this either until it is impossible to choose $H_i$ because $\SW_{d_i}(\pi^{(i)})=\emptyset$ or until we obtain a permutation $\pi^{(0)}$. Note that if we do reach a point where $\SW_{d_i}(\pi^{(i)})=\emptyset$ for some $i\in [k]$, then $\mathcal P^{-1}(\pi^{(i)})\cap\mathsf D{\bf T}=\emptyset$. 
\hspace*{\fill}$\lozenge$

We now see that the number of ways to choose a tree in $\mathcal P^{-1}(\pi)\cap\mathsf D{\bf T}$ is equal to the number of ways to do the following two tasks: 
\begin{itemize}
\item[($\dagger$)]\label{ItemI} Perform the Iterative Decomposition Procedure until producing a permutation $\pi^{(0)}$. 
\item[($\dagger\dagger$)]\label{ItemII} Choose trees $\mathcal T_U^{(1)}\in \mathcal P^{-1}(\pi^{(0)})\cap\mathsf D{\bf T}$ and $\mathcal T_S^{(i)}\in \mathcal P^{-1}((\pi^{(i)})_S^{H_i})\cap\mathsf D{\bf T}$ for all $i\in [k]$. 
\end{itemize} 

Suppose we have already performed Task ($\dagger$). We can naturally view the chosen hooks $H_1,\ldots,H_k$ as hooks of the original permutation $\pi$. By construction, the permutations $\pi^{(0)}$ and $(\pi^{(i)})_S^{H_i}$, which we can see as subpermutations of $\pi$, are each increasing. Let $q_0$ denote the length of $\pi^{(0)}$. For each $i\in[k]$, let $q_i$ denote the length of $(\pi^{(i)})_S^{H_i}$. According to Lemma~\ref{Lem2}, the number of ways to perform Task ($\dagger\dagger$) is given by the product $\prod_{t=0}^k{\bf G}_{q_t}(1,\ldots,1)$. In fact, we can say more by taking into account the statistics $f_1,\ldots,f_r$. It follows from the Refined Tree Decomposition Lemma (more precisely, the bijection used in its proof) and Lemma~\ref{Lem2} that $\sum x_1^{\ddot f_1(\mathcal T)}\cdots x_r^{\ddot f_r(\mathcal T)}=\prod_{t=0}^k{\bf G}_{q_t}(x_1,\ldots,x_r)$, where the sum ranges over all trees in $\mathcal P^{-1}(\pi)\cap\mathsf{D}{\bf T}$ that can be formed by performing Task ($\dagger\dagger$) in various ways (we still assume that we have already performed Task ($\dagger$)).   

\begin{example}\label{Exam3}
Consider the permutation $\pi=2\,7\,3\,5\,9\,10\,11\,4\,8\,1\,6\,12\,13\,14\,15\,16$, whose plot is shown in Figure~\ref{Fig8}. We have $k=\des(\pi)=3$; the descents of $\pi$ are $d_1=2$, $d_2=7$, and $d_3=9$. Let us begin the Iterative Decomposition Procedure by setting $\pi^{(3)}=\pi$ and choosing $H_3$ to be the hook in $\SW_9(\pi^{(3)})$ whose northeast endpoint has height $13$. We obtain the subpermutations $(\pi^{(3)})_S^{H_3}=1\,6\,12$ and $\pi^{(2)}=(\pi^{(3)})_U^{H_3}=2\,7\,3\,5\,9\,10\,11\,4\,8\,14\,15\,16$. We next choose $H_2$ to be the hook in $\SW_7(\pi^{(2)})$ whose northeast endpoint has height $15$; we can naturally identify this hook with the hook of $\pi$ labeled $H_2$ in Figure~\ref{Fig8}. We obtain the subpermutations $(\pi^{(2)})_S^{H_2}=4\,8\,14$ and $\pi^{(1)}=(\pi^{(2)})_U^{H_2}=2\,7\,3\,5\,9\,10\,11\,16$. Finally, choose $H_1$ to be the hook in $\SW_2(\pi^{(1)})$ whose northeast endpoint has height $11$; as before we can see $H_1$ as a hook of $\pi$. We obtain the subpermutations $(\pi^{(1)})_S^{H_1}=3\,5\,9\,10$ and $\pi^{(0)}=(\pi^{(1)})_U^{H_1}=2\,7\,16$. 

\begin{figure}[ht]
\begin{center}
\includegraphics[width=.5\linewidth]{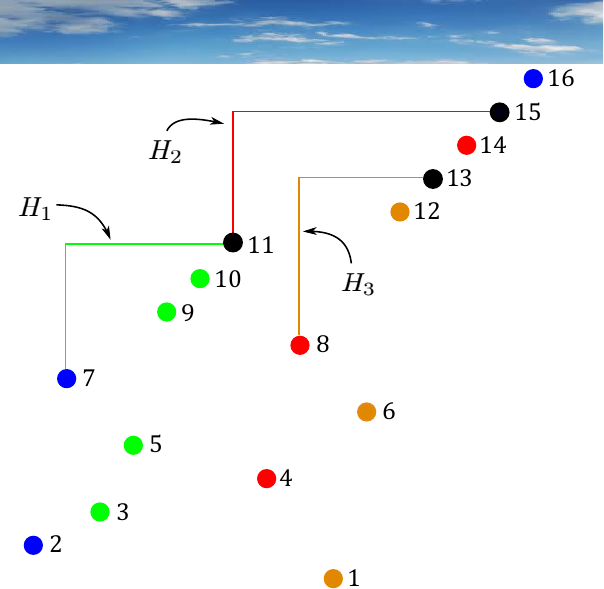}
\caption{Hooks that divide the plot of a permutation $\pi$ into smaller pieces.} 
\label{Fig8}
\end{center}  
\end{figure} 

In Figure~\ref{Fig8}, we have assigned different colors to the hooks $H_i$ and have colored the points of $(\pi^{(i)})_S^{H_i}$ with the same color as $H_i$. We have also drawn a sky above the entire diagram. For now, the sky is mostly decorative, but we will see later that it plays an important role. We have colored the points of $\pi^{(0)}$ blue to match the color of the sky. We have $q_0=3$, $q_1=4$, $q_2=3$, and $q_3=3$, so this particular choice of hooks accounts for a total of ${\bf G}_3(1,\ldots,1)^3{\bf G}_4(1,\ldots,1)$ of the trees in $\mathcal P^{-1}(\pi)\cap\mathsf D{\bf T}$. The sum of $x_1^{\ddot f_1(\mathcal T)}\cdots x_r^{\ddot f_r(\mathcal T)}$ over all such trees $\mathcal T$ is ${\bf G}_3(x_1,\ldots,x_r)^3{\bf G}_4(x_1,\ldots,x_r)$.   
\end{example} 

Suppose, as above, that we have a permutation $\pi$ with descents $d_1<\cdots<d_k$ and that we have performed Task ($\dagger$). For every $i\in[k]$, the southwest endpoint of the hook $H_i$ is the point $(d_i,\pi_{d_i})$. Observe that none of these hooks pass directly underneath any points in the plot of $\pi$. Furthermore, no two of these hooks intersect each other unless the southwest endpoint of one hook is the northeast endpoint of another. This leads us naturally to the definition of a valid hook configuration. 

\begin{definition}\label{Def5}
Let $\pi$ be a permutation with descents $d_1<\cdots <d_k$, where $k=\des(\pi)$. A \dfn{valid hook configuration} of $\pi$ is a tuple $\mathcal H=(H_1,\ldots,H_k)$ of hooks of $\pi$ that satisfy the following properties: 
\begin{enumerate}
\item \label{Item1} For each $i\in[k]$, the southwest endpoint of $H_i$ is $(d_i,\pi_{d_i})$. 
\item \label{Item2} No point in the plot of $\pi$ lies directly above a hook in $\mathcal H$. 
\item \label{Item3} No two hooks intersect or overlap each other unless the northeast endpoint of one is the southwest endpoint of the other. 
\end{enumerate}

Let $\VHC(\pi)$\label{sym:VHC} denote the set of valid hook configurations of $\pi$. More generally, for each set $S$ of permutations, let $\VHC(S)=\bigcup_{\pi\in S}\VHC(\pi)$. We make the convention that a valid hook configuration includes its underlying permutation as part of its identity. In other words, $\VHC(\pi)$ and $\VHC(\pi')$ are disjoint whenever $\pi\neq\pi'$. We also make the convention that if $\pi$ is monotonically increasing, then $\VHC(\pi)$ contains a single element: the empty valid hook configuration of $\pi$, which has no hooks. 
\end{definition} 

To build further intuition for this definition, consider Figure~\ref{Fig9}, which shows all six elements of $\VHC(3142567)$. 

\begin{figure}[ht]
\begin{center} 
\includegraphics[height=6.5cm]{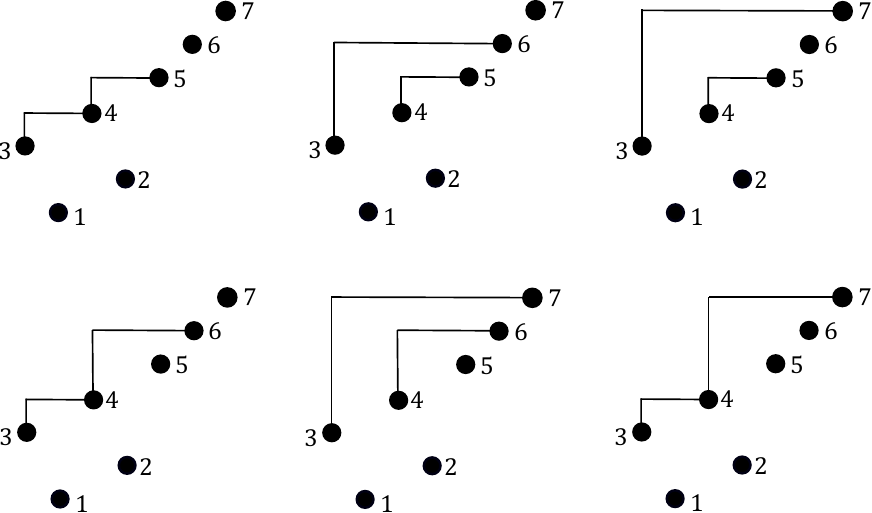}
\end{center}
\caption{The valid hook configurations of $3142567$.}\label{Fig9}
\end{figure}

We have observed that every configuration of hooks that arises by completing Task~($\dagger$) is a valid hook configuration. It is straightforward to check (for example, by induction), that every valid hook configuration of $\pi$ arises in this way. In other words, completing Task~($\dagger$) is equivalent to choosing an element of $\VHC(\pi)$. Once we have done this, the number of ways to complete Task~($\dagger\dagger$) is given by a product of numbers ${\bf G}_{q_t}(1,\ldots,1)$, as described above. We can obtain the specific numbers $q_t$ immediately from the given valid hook configuration as follows.  

Fix $\pi=\pi_1\cdots\pi_n$ with $\des(\pi)=k$. Each valid hook configuration $\mathcal H=(H_1,\ldots,H_k)\in\VHC(\pi)$ induces a coloring of the plot of $\pi$. To begin this coloring, draw a sky over the entire diagram and assign a color to the sky; in this article, we will always color the sky blue. Assign arbitrary distinct colors other than blue to the hooks $H_1,\ldots,H_k$. There are $k$ northeast endpoints of hooks, and these points remain uncolored. However, all of the other $n-k$ points will be colored. In order to decide how to color a point $(i,\pi_i)$ that is not a northeast endpoint, imagine that this point looks directly upward. If it sees a hook when looking upward, it receives the same color as the hook that it sees. If it does not see a hook, it must see the sky, so it receives the color blue. However, if $(i,\pi_i)$ is the southwest endpoint of a hook, then it must look around (on the left side of) the vertical part of that hook. 

Figure~\ref{Fig8} shows the coloring of the plot of a permutation induced by a valid hook configuration. Indeed, the points given the same color as the hook $H_i$ in this coloring correspond to the entries of the subpermutation $(\pi^{(i)})_S^{H_i}$ that appears in the Iterative Decomposition Procedure. Similarly, the points colored blue correspond to the entries of $\pi^{(0)}$. Keeping with the notation introduced earlier, we let $q_i$ denote the number of points given the same color as the hook $H_i$. Let $q_0$ be the number of blue points. Let ${\bf q}^{\mathcal H}$ denote the tuple $(q_0,\ldots,q_k)$. Observe that the point $(d_i+1,\pi_{d_i+1})$ is necessarily the same color as $H_i$, while $(1,\pi_1)$ is necessarily blue. This means that ${\bf q}^{\mathcal H}$ is a composition of $n-k$ into $k+1$ parts. We will find it convenient to write \[{\bf G}_{{\bf q}^{\mathcal H}}(x_1,\ldots,x_r)=\prod_{t=0}^k{\bf G}_{q_t}(x_1,\ldots,x_r).\]  

\begin{remark}\label{Rem1}
Let $\Comp_{k+1}(n-k)$ denote the set of compositions of $n-k$ into $k+1$ parts. Fix a permutation $\pi$ with $\des(\pi)=k$. Suppose $\mathcal H=(H_1,\ldots,H_k)\in\VHC(\pi)$ is such that ${\bf q}^{\mathcal H}=(q_0,\ldots,q_k)$. The hook $H_k$ is completely determined by the number $q_k$ and the permutation $\pi$. It then follows by induction on $\ell$ that the hooks $H_{k-\ell}$ for $1\leq \ell\leq k-1$ are also determined by the entries in $(q_0,\ldots,q_k)$. Thus, the map $\VHC(\pi)\to\Comp_{k+1}(n-k)$ given by $\mathcal H\mapsto{\bf q}^{\mathcal H}$ is injective. 
\end{remark} 

The previous discussion yields a formula for the number of trees in $\mathsf D{\bf T}$ with a prescribed postorder reading, counted according to the statistics $\ddot f_1,\ldots,\ddot f_r$. 

\begin{theorem}[Refined Tree Fertility Formula]\label{Thm22}
Let ${\bf T}$ be a troupe, and let $f_1,\ldots,f_r$ be insertion-additive tree statistics. For every permutation $\pi$, we have \[\sum_{\mathcal T\in\mathcal P^{-1}(\pi)\cap\mathsf D{\bf T}}x_1^{\ddot f_1(\mathcal T)}\cdots x_r^{\ddot f_r(\mathcal T)}=\sum_{\mathcal H\in\VHC(\pi)}{\bf G}_{{\bf q}^{\mathcal H}}(x_1,\ldots,x_r),\] where $\displaystyle {\bf G}_{{\bf q}^{\mathcal H}}(x_1,\ldots,x_r)=\prod_{t=0}^k{\bf G}_{q_t}(x_1,\ldots,x_r)=\prod_{t=0}^k\sum_{T\in{\bf T}_{q_t}}x_1^{f_1(T)}\cdots x_r^{f_r(T)}$ when ${\bf q}^{\mathcal H}=(q_0,\ldots,q_k)$.    
\end{theorem}

We now give some concrete examples to illustrate the Refined Tree Fertility Formula. Recall that if $(u_n)_{n\geq 1}$ is a sequence of elements of a field and ${\bf q}=(q_0,\ldots,q_k)$ is a composition, then we let $u_{\bf q}=\prod_{t=0}^ku_{q_t}$. 

\begin{example}[Binary Plane Trees]\label{ExamBinary1}
Let ${\bf T}=\mathsf{BPT}$, and let $f_1(T)=\des(T)+1$ and $f_2(T)=\peak(T)+1$. We know by Example~\ref{Exam1} that $f_1$ and $f_2$ are insertion-additive. Therefore, we may apply the Refined Tree Fertility Formula to find that 
\begin{equation}\label{Eq17}
\sum_{\mathcal T\in\mathcal P^{-1}(\pi)\cap\mathsf {DBPT}}x_1^{\des(\mathcal T)+1}x_2^{\peak(\mathcal T)+1}=\sum_{\mathcal H\in\VHC(\pi)}{\bf G}_{{\bf q}^{\mathcal H}}(x_1,x_2)
\end{equation} for every permutation $\pi$. In order to perform calculations with this formula, we would like to know how to explicitly compute the polynomials \[{\bf G}_n(x_1,x_2)=\sum_{T\in\mathsf{BPT}_n}x_1^{\des(T)+1}x_2^{\peak(T)+1}.\] Recall that we made the conventions $\des(\varepsilon)=0$ and $\peak(\varepsilon)=-1$, where $\varepsilon$ is the empty tree. In other words, ${\bf G}_0(x_1,x_2)=x_1$. Let ${\bf G}^{(x_1,x_2)}(z)=\sum_{n\geq 0}{\bf G}_n(x_1,x_2)z^n$. Each binary plane tree must be empty, a single root vertex, a root vertex with a nonempty left subtree and an empty right subtree, or a root vertex with a nonempty right subtree and a left subtree that may or may not be empty.  This observation translates into the equation \[{\bf G}^{(x_1,x_2)}(z)=x_1 + x_1x_2z + z({\bf G}^{(x_1,x_2)}(z)-x_1)+z({\bf G}^{(x_1,x_2)}(z)-x_1){\bf G}^{(x_1,x_2)}(z),\] which implies that 
\begin{equation}\label{Eq16}
{\bf G}^{(x_1,x_2)}(z)=\frac{1-z+x_1z-\sqrt{(1-z+x_1z)^2-4x_1z(1-z+x_2z)}}{2z}.
\end{equation} Although we will not need it, we remark that one can derive from \eqref{Eq16} the formula \[[x_1^ix_2^jz^n]{\bf G}^{(x_1,x_2)}(z)=\frac{1}{n+1-j}\binom{n-1}{n-j}\binom{n+1-j}{j}\binom{n+1-2j}{i-j}\] for the coefficient of $x_1^ix_2^jz^n$ in ${\bf G}^{(x_1,x_2)}(z)$.  

We now consider the \dfn{Narayana numbers} $N(n,i)=\frac{1}{n}\binom{n}{i}\binom{n}{i-1}$. These numbers constitute one of the most common refinements of the sequence of Catalan numbers. In particular, $N(n,i)$ is the number of binary plane trees with $n$ vertices and $i-1$ right edges. Therefore, we have \[{\bf G}_n(x_1,1)=\sum_{T\in\mathsf{BPT}_n}x_1^{\des(T)+1}=N_n(x_1),\] where $N_n(x)$ is the \dfn{Narayana polynomial} defined by $\displaystyle N_n(x)=\sum_{i=1}^nN(n,i)x^i$. Narayana polynomials are also the $h$-polynomials of associahedra \cite{Postnikov}. Specializing $x_2=1$ in \eqref{Eq17} yields 
\begin{equation}\label{Eq18}
\sum_{\mathcal T\in\mathcal P^{-1}(\pi)\cap\mathsf {DBPT}}x_1^{\des(\mathcal T)+1}=\sum_{\mathcal H\in\VHC(\pi)}N_{{\bf q}^{\mathcal H}}(x_1)
\end{equation} for every permutation $\pi$. Specializing further, we have $N_n(1)=C_n$, so 
\begin{equation}\label{Eq19}
|\mathcal P^{-1}(\pi)\cap\mathsf{DBPT}|=\sum_{\mathcal H\in\VHC(\pi)}C_{{\bf q}^{\mathcal H}}.
\end{equation} 

These results concerning binary plane trees translate immediately into the language of stack-sorting. It follows from \eqref{Eq11} that for every permutation $\tau$, the in-order reading $\mathcal I:\mathcal P^{-1}(\tau)\cap\mathsf{DBPT}\to s^{-1}(\tau)$ is a bijection. We have $\des(\mathcal T)=\des(\mathcal I(\mathcal T))$ and $\peak(\mathcal T)=\peak(\mathcal I(\mathcal T))$ for every $\mathcal T\in\mathsf{DBPT}$. Therefore, \eqref{Eq17}, \eqref{Eq18}, and \eqref{Eq19} are equivalent to the equations 
\begin{equation}\label{Eq20}
\sum_{\sigma\in s^{-1}(\pi)}x_1^{\des(\sigma)+1}x_2^{\peak(\sigma)+1}=\sum_{\mathcal H\in\VHC(\pi)}{\bf G}_{{\bf q}^{\mathcal H}}(x_1,x_2),
\end{equation}
\begin{equation}\label{Eq21}
\sum_{\sigma\in s^{-1}(\pi)}x_1^{\des(\sigma)+1}=\sum_{\mathcal H\in\VHC(\pi)}N_{{\bf q}^{\mathcal H}}(x_1),
\end{equation}
and 
\begin{equation}\label{Eq22}
|s^{-1}(\pi)|=\sum_{\mathcal H\in\VHC(\pi)}C_{{\bf q}^{\mathcal H}}.
\end{equation}

We call equations \eqref{Eq20} and \eqref{Eq22} the Refined Fertility Formula and the Fertility Formula, respectively. The current author has used these formulas to generalize many known results and to prove several new results concerning the stack-sorting map \cite{DefantCatalan, DefantClass, DefantFertility, DefantFertilityWilf, DefantEngenMiller, DefantPreimages}. 
\end{example}

The results in the following three examples are all new. This includes the specific cases in which we set $x_1=\cdots=x_r=1$, which provide formulas for the number of trees of each specified type that have a prescribed postorder reading. 

\begin{example}[Full Binary Plane Trees]\label{ExamFull1}
Let ${\bf T}$ be the troupe $\mathsf{FBPT}$. Every tree in $\mathsf{FBPT}$ has an odd number of vertices, and it is straightforward to check that $\des(T)=\peak(T)=k$ for every $T\in\mathsf{FBPT}_{2k+1}$. Therefore, counting trees in $\mathsf{FBPT}$ according to the statistics $\des$ and $\peak$ is not interesting (it suffices to count according the number of vertices). We have ${\bf G}_n=|\mathsf{FBPT}_n|=C_{(n-1)/2}$, where $C_{(n-1)/2}=0$ when $n$ is even. The sequence $(C_{(n-1)/2})_{n\geq 1}$ is the sequence of \dfn{aerated Catalan numbers} (OEIS sequence A126120 \cite{OEIS}). By the Tree Fertility Formula, we have 
\begin{equation}\label{Eq25}
|\mathcal P^{-1}(\pi)\cap\mathsf {DFBPT}|=\sum_{\mathcal H\in\VHC(\pi)}C_{({\bf q}^{\mathcal H}-1)/2}
\end{equation} for every permutation $\pi$. 

As in the previous example, we can reformulate this result in terms of the stack-sorting map. We say a permutation $\pi=\pi_1\cdots\pi_n$ is \dfn{alternating} if its set of descents is precisely the set of even elements of $[n-1]$. Let $\ALT$ be the set of alternating permutations. Alternating permutations have been studied extensively (see Stanley's survey \cite{StanleyAlternating} and the references therein); the number of alternating permutations in $S_n$ is the \dfn{Euler number} $E_n$. These numbers can be defined via the generating function equation $\displaystyle\sum_{n\geq 0}E_n\frac{z^n}{n!}=\sec(z)+\tan(z)$. The numbers $E_{n}$ with $n$ even are called \dfn{secant numbers}, and the numbers $E_n$ with $n$ odd are called \dfn{tangent numbers}. When $n$ is odd, the in-order reading gives a bijection from $\mathsf{DFBPT}_n$ to the set of alternating permutations of length $n$. Therefore, we can use \eqref{Eq11} and \eqref{Eq25} to obtain the following new theorem concerning the stack-sorting map. 

\begin{theorem}\label{Thm6}
For every permutation $\pi$ of odd length, the number of alternating permutations in $s^{-1}(\pi)$ is \[\sum_{\mathcal H\in\VHC(\pi)}C_{({\bf q}^{\mathcal H}-1)/2}.\]
\end{theorem}

\begin{problem}\label{Prob1} 
Find an analogue of Theorem~\ref{Thm6} for permutations of even length. \qedhere
\end{problem}   
\end{example} 

\begin{example}[Motzkin Trees]\label{ExamMotzkin1} 
Let ${\bf T}$ be the troupe $\mathsf{Mot}$, and let $f_1$ be the insertion-additive tree statistic given by $f_1(T)=\des(T)+1$. We could also consider the statistic given by $T\mapsto \peak(T)+1$, but that would be redundant because $\des(T)=\peak(T)$ for every Motzkin tree $T$. Consider the \dfn{Motzkin polynomials} $\displaystyle M_{n-1}(x)=\sum_{i=1}^n\frac{(n-1)!}{(n-2i+1)!\,i!\,(i-1)!}x^i$, whose coefficients form the OEIS sequence A055151 \cite{OEIS}. These polynomials refine the sequence of Motzkin numbers in the same way that Narayana polynomials refine the sequence of Catalan numbers. Indeed, the coefficient of $x^i$ in $M_{n-1}(x)$ is the number of Motzkin trees with $n$ vertices and $i-1$ right edges. It is also known \cite{Postnikov} that Motzkin polynomials are the $\gamma$-polynomials of associahedra. Let \[{\bf G}_n(x_1)=\sum_{T\in\mathsf{Mot}_n}x_1^{\des(T)+1}=M_{n-1}(x_1),\] 
and let ${\bf G}^{(x_1)}(z)=\sum_{n\geq 0}{\bf G}_n(x_1)z^n$. Each Motzkin tree must be a single root vertex, a root vertex with a nonempty left subtree and an empty right subtree, or a root vertex with two nonempty subtrees. This yields the equation \[{\bf G}^{(x_1)}(z)=x_1z + z{\bf G}^{(x_1)}(z)+z{\bf G}^{(x_1)}(z)^2,\] which implies that 
\begin{equation}\label{Eq23}
{\bf G}^{(x_1)}(z)=\frac{1-z-\sqrt{1-2z+z^2-4x_1z^2}}{2z}.
\end{equation} By the Refined Tree Fertility Formula, we have 
\begin{equation}\label{Eq24}
\sum_{\mathcal T\in\mathcal P^{-1}(\pi)\cap\mathsf {DMot}}x_1^{\des(\mathcal T)+1}=\sum_{\mathcal H\in\VHC(\pi)}{\bf G}_{{\bf q}^{\mathcal H}}(x_1)=\sum_{\mathcal H\in\VHC(\pi)}M_{{\bf q}^{\mathcal H}-1}(x_1)
\end{equation} for every permutation $\pi$. 
Since $M_{n-1}(1)=M_{n-1}=|\mathsf{Mot}_n|$, we can specialize $x_1=1$ to obtain the formula 
\begin{equation}\label{Eq26}
|\mathcal P^{-1}(\pi)\cap\mathsf{DMot}|=\sum_{\mathcal H\in\VHC(\pi)}M_{{\bf q}^{\mathcal H}-1}.
\end{equation}  

We can translate these results into the language of stack-sorting. Let \[\EDP=\{\pi:\des(\pi)=\peak(\pi)\}\] denote the set of permutations in which every descent is a peak. Alternatively, $\EDP$ is the set of permutations that have no double descents (i.e., consecutive descents) and in which $1$ is not a descent. The standardized permutations in $\EDP$ are counted by the OEIS sequence A080635 \cite{OEIS}, which has the exponential generating function \[\sum_{n\geq 0}|\EDP\cap S_n|\frac{z^n}{n!}=\frac{1}{2}+\frac{\sqrt 3}{2}\tan\left(\frac{\sqrt 3}{2}z+\frac{\pi}{6}\right).\] The in-order reading gives a bijection $\mathcal I:\mathsf{DMot}\to\EDP$. Therefore, we can use \eqref{Eq11} and \eqref{Eq24} to obtain the following new theorem concerning the stack-sorting map. 

\begin{theorem}\label{Thm16}
For every permutation $\pi$, we have \[\sum_{\sigma\in s^{-1}(\pi)\cap\EDP}x_1^{\des(\sigma)+1}=\sum_{\mathcal H\in\VHC(\pi)}M_{{\bf q^{\mathcal H}}-1}(x_1).\] In particular, the number of permutations in $s^{-1}(\pi)$ whose descents are all peaks is \[\sum_{\mathcal H\in\VHC(\pi)}M_{{\bf q}^{\mathcal H}-1}. \qedhere\]
\end{theorem}
\end{example} 

\begin{example}[Schr\"oder $2$-Colored Binary Trees]\label{ExamSchroder1}
Let ${\bf T}$ be the troupe $\mathsf{Sch}$, and let $f_1(T)=\des(T)+1$, $f_2(T)=\peak(T)+1$, and $f_3(T)=\black(T)+1$. We saw in Example~\ref{Exam1} that $f_1,f_2,f_3$ are insertion-additive. Let ${\bf G}^{(x_1,x_2,x_3)}(z)=\sum_{n\geq 0}{\bf G}_n(x_1,x_2,x_3)z^n$, where \[{\bf G}_n(x_1,x_2,x_3)=\sum_{T\in\mathsf{Sch}_n}x_1^{\des(T)+1}x_2^{\peak(T)+1}x_3^{\black(T)+1}.\] Although we will not explicitly need this formula, one can show that \[{\bf G}^{(x_1,x_2,x_3)}(z)=\frac{1-x_1z-x_3z+x_1x_3z-\sqrt{Q^{(x_1,x_2,x_3)}(z)}}{2z},\] where \[Q^{(x_1,x_2,x_3)}(z)=(1-x_1z-x_3z+x_1x_3z)^2-4z(x_1x_3-x_1^2x_3z+x_1x_2x_3z-x_1x_3^2z+x_1x_2x_3^2z).\] The Refined Tree Fertility Formula tells us that 
\[
\sum_{\mathcal T\in\mathcal P^{-1}(\pi)\cap\mathsf {DSch}}x_1^{\des(\mathcal T)+1}x_2^{\peak(\mathcal T)+1}x_3^{\black(\mathcal T)+1}=\sum_{\mathcal H\in\VHC(\pi)}{\bf G}_{{\bf q}^{\mathcal H}}(x_1,x_2,x_3)
\] for every permutation $\pi$.

Let us specialize to the case in which $x_2=x_3=1$. Note that ${\bf G}_n(x_1,1,1)$ counts Schr\"oder $2$-colored binary trees with $n$ vertices according to their number of right edges. Every Schr\"oder $2$-colored binary tree with $n$ vertices and $j$ right edges can be constructed by choosing a binary plane tree with $n$ vertices and $j$ right edges, coloring black the $n-1-j$ vertices that have left children (since white vertices cannot have left children), and then coloring each of the remaining $j+1$ vertices either black or white. There are $2^{j+1}N(n,j+1)$ ways to make these choices, so \[{\bf G}_n(x_1,1,1)=\sum_{j=0}^{n-1}2^{j+1}N(n,j+1)x_1^{j+1}=N_{n}(2x_1),\] where $N(n,j+1)$ and $N_n(2x)$ denote Narayana numbers and Narayana polynomials. Thus, \[
\sum_{\mathcal T\in\mathcal P^{-1}(\pi)\cap\mathsf {DSch}}x_1^{\des(\mathcal T)+1}=\sum_{\mathcal H\in\VHC(\pi)}N_{{\bf q}^{\mathcal H}}(2x_1).
\]

Another interesting specialization comes from setting $x_1=x_2=1$. It follows from Corollary 4.2 in \cite{Gu} that $\binom{n+j}{n-j}C_j$ is the number of trees in $\mathsf{Sch}_n$ with $j$ black vertices. Therefore, \[\sum_{\mathcal T\in\mathcal P^{-1}(\pi)\cap\mathsf {DSch}}x_3^{\black(\mathcal T)+1}=\sum_{\mathcal H\in\VHC(\pi)}{\bf G}_{{\bf q}^{\mathcal H}}(1,1,x_3),\] where $\displaystyle{\bf G}_{{\bf q}^{\mathcal H}}(1,1,x_3)=\prod_{t=0}^k\sum_{j=0}^{q_t}\binom{q_t+j}{q_t-j}C_jx_3^{j+1}$ when ${\bf q}^{\mathcal H}=(q_0,\ldots,q_k)$.

Finally, ${\bf G}_n(1,1,1)$ is the $n^\text{th}$ large Schr\"oder number $\mathscr S_n$, so 
\begin{equation}\label{Eq50}
|\mathcal P^{-1}(\pi)\cap\mathsf {DSch}|=\sum_{\mathcal H\in\VHC(\pi)}\mathscr S_{{\bf q}^{\mathcal H}}. \qedhere
\end{equation} 
\end{example}

In summary, this section shows that if one knows the set of valid hook configurations of a permutation $\pi$, then one can count the trees in $\mathcal P^{-1}(\pi)\cap\mathsf D{\bf T}$ for a large variety of sets ${\bf T}$. One can even count these trees according to some natural statistics. 

\section{Free Probability Theory and the VHC Cumulant Formula}\label{Sec:VHCCumulant}

\subsection{Background}
Let $\mathbb K$ be a field. Let $\Pi(X)$ denote the collection of all set partitions of a totally ordered finite set $X$. We let $\Pi(n)=\Pi([n])$. Given a sequence $(u_n)_{n\geq 1}$ of elements of $\mathbb K$ and a set partition $\rho$, we let \[(u_\bullet)_\rho=\prod_{B\in\rho}u_{|B|}.\] We say two distinct blocks $B,B'$ of a set partition $\rho\in\Pi(X)$ \dfn{form a crossing} if there exist $i,j\in B$ and $i',j'\in B'$ such that either $i<i'<j<j'$ or $i>i'>j>j'$. A partition is \dfn{noncrossing} if no two of its blocks form a crossing. Let $\NC(X)$ be the set of noncrossing partitions in $\Pi(X)$, and let $\NC(n)=\NC([n])$. The sets $\Pi(n)$ and $\NC(n)$ are both lattices under the reverse refinement ordering \cite[Lecture 9]{Nica}. 

A \dfn{noncommutative probability space} over $\mathbb K$ is a pair $(\mathcal A,\varphi)$, where $\mathcal A$ is a unital associative algebra and $\varphi:\mathcal A\to\mathbb K$ is a unital linear functional (meaning $\varphi(1_{\mathcal A})=1_{\mathbb K}$). Given $a_1,\ldots,a_n\in\mathcal A$ and $B=\{b_1<\cdots<b_r\}\subseteq [n]$, let $a_B=(a_{b_1},\ldots,a_{b_r})$. One of the goals of noncommutative probability theory is to understand the \dfn{joint moments} \[m_n(a_1,\ldots,a_n):=\varphi(a_1\cdots a_n).\] The \dfn{classical cumulants} are the elements $c_n(a_1,\ldots,a_n)$ of $\mathbb K$ that satisfy the formula 
\begin{equation}\label{Eq1}
m_n(a_1,\ldots,a_n)=\sum_{\rho\in\Pi(n)}c_\rho(a_1,\ldots,a_n),
\end{equation} where $\displaystyle c_\rho(a_1,\ldots,a_n)=\prod_{B\in\rho}c_{|B|}(a_B)$. This formula immediately implies that the joint moments are determined by the classical cumulants. On the other hand, one can apply M\"obius inversion to \eqref{Eq1} in order to deduce that the classical cumulants are determined by the joint moments \cite[Lecture~11]{Nica}. 

The \dfn{free cumulants}, originally introduced by Speicher in \cite{Speicher}, are the elements $\kappa_n(a_1,\ldots,a_n)$ of $\mathbb K$ that satisfy the formula 
\begin{equation}\label{Eq2}
m_n(a_1,\ldots,a_n)=\sum_{\eta\in\NC(n)}\kappa_\eta(a_1,\ldots,a_n),
\end{equation} where $\displaystyle \kappa_\eta(a_1,\ldots,a_n)=\prod_{B\in\eta}\kappa_{|B|}(a_B)$. This shows that the free cumulants determine the moments. One can use M\"obius inversion, this time on the noncrossing partition lattice, to rearrange \eqref{Eq2}, expressing the free cumulants in terms of the moments \cite[Lecture~11]{Nica}. 

The preceding paragraphs describe moments, classical cumulants, and free cumulants that are \dfn{multivariate} in the sense that they involve several (possibly) distinct elements $a_1,a_2,\ldots$ of $\mathcal A$. In many applications, it will suffice to consider the \dfn{univariate} case in which the elements $a_1,a_2,\ldots$ are all equal. In this case, we drop the notation expressing the dependence on $a_1,a_2,\ldots$ and simply write $m_n$, $c_n$, and $\kappa_n$. In fact, we will rarely need to refer to the noncommutative probability space $(\mathcal A,\varphi)$. For the sake of notational convenience and clarity of exposition, we will phrase all of the results concerning cumulants in the univariate setting; we will then explicitly point out which results generalize straightforwardly to the multivariate setting. Because we will only perform formal combinatorial and algebraic manipulations, the sequences $(m_n)_{n\geq 1}$, $(c_n)_{n\geq 1}$, $(\kappa_n)_{n\geq 1}$ can be \emph{any} sequences of elements of $\mathbb K$, so long as they satisfy the defining equations 
\begin{equation}\label{Eq3}
m_n=\sum_{\rho\in\Pi(n)}(c_\bullet)_\rho\qquad\text{and}\qquad m_n=\sum_{\eta\in\NC(n)}(\kappa_\bullet)_\eta\qquad\text{for all }n\geq 1.
\end{equation}

As mentioned above, each one of the sequences $(m_n)_{n\geq 1}$, $(c_n)_{n\geq 1}$, $(\kappa_n)_{n\geq 1}$ determines the other two. For example, if we are given a sequence of classical cumulants $(c_n)_{n\geq 1}$, then the corresponding free cumulants $\kappa_n$ are given by \[\kappa_n=\sum_{\eta\in\NC(n)}\mu^{\NC}(\eta,\widehat 1_n)(m_\bullet)_\eta=\sum_{\eta\in\NC(n)}\mu^{\NC}(\eta,\widehat 1_n)\prod_{B\in\eta}\sum_{\rho\in\Pi(B)}(c_\bullet)_\rho,\] where $\mu^{\NC}$ and $\widehat 1_n$ denote the M\"obius function of $\NC(n)$ and the maximal element of $\NC(n)$, respectively. This last expression is somewhat unsatisfying because it does not give a clear combinatorial picture of what is happening. The following result due to Lehner gives a much simpler combinatorial explanation of how to convert from classical to free cumulants. The \dfn{crossing graph} $G(\rho)$ of a set partition $\rho\in\Pi(X)$ is the graph whose vertices are the blocks of $\rho$ in which two blocks are adjacent if and only if they form a crossing. In particular, a partition is noncrossing if and only if its crossing graph has no edges. We say a set partition is \dfn{connected} if its crossing graph is connected. Let $\Pi^{\con}(X)$ denote the set of connected set partitions in $\Pi(X)$, and let $\Pi^{\con}(n)=\Pi^{\con}([n])$. 

\begin{theorem}[\!\!\cite{Lehner}]\label{Thm1}
If $(c_n)_{n\geq 1}$ is a sequence of classical cumulants, then the corresponding free cumulants are given by \[\kappa_n=\sum_{\rho\in\Pi^{\con}(n)}(c_\bullet)_\rho.\] 
\end{theorem} 

More recently, Josuat-Verg\`es found a simple combinatorial formula that inverts Theorem \ref{Thm1}. Let $T_G(x,y)$ be the \emph{Tutte polynomial} of a finite graph $G$. We refer the reader to \cite{Bollobas} and the references therein for more information about this important graph invariant and its generalizations. 

\begin{theorem}[\!\!\cite{Josuat}]\label{Thm2}
If $(\kappa_n)_{n\geq 1}$ is a sequence of free cumulants, then the corresponding classical cumulants are given by \[-c_n=\sum_{\rho\in\Pi^{\con}(n)}T_{G(\rho)}(1,0)(-\kappa_\bullet)_\rho.\] 
\end{theorem}

The obvious generalizations of Lehner's theorem and Josuat-Verg\`es' theorem to the multivariate setting hold as well. 

The reason why it is useful to have combinatorial formulas for converting between cumulants, especially in our applications to the stack-sorting map in Section~\ref{Sec:Applications}, is that they correspond to transformations of generating functions. Whenever we have a sequence $(u_n)_{n\geq 1}$ of elements of $\mathbb K$, we can consider the ordinary generating function $\displaystyle F(z)=\sum_{n\geq 1}u_nz^n$; we then let $\displaystyle \widehat F(z)=\sum_{n\geq 1}u_n\dfrac{z^n}{n!}$ denote the corresponding exponential generating function. It turns out that 
\begin{equation}\label{Eq6}
\widehat F(z)=\mathcal L^{-1}\{F(1/t)/t\}(z),
\end{equation} where $\mathcal L^{-1}$ denotes the inverse Laplace transform. Indeed, this follows from the linearity of the inverse Laplace transform and the fact that $\mathcal L^{-1}\{t^{-n-1}\}(z)=\dfrac{z^n}{n!}$. If $u_1\neq 0$, the series $F(z)$ has a unique formal compositional inverse, which is another power series that we denote by $F^{\langle -1\rangle}(z)$; it satisfies $F(F^{\langle -1\rangle}(z))=F^{\langle -1\rangle}(F(z))=z$. If $u_1=0$, there can be multiple compositional inverses of $F(z)$. When this arises in applications, we can determine the correct series by analyzing the initial terms of the candidate compositional inverses. 

Let $(m_n)_{n\geq 1}$ be a moment sequence, and let $(c_n)_{n\geq 1}$ and $(\kappa_n)_{n\geq 1}$ be the corresponding sequences of classical and free cumulants. The \dfn{moment series} $M(z)$ is simply the ordinary generating function \[M(z)=\sum_{n\geq 1}m_nz^n.\] Then $\displaystyle \widehat M(z)=\sum_{n\geq 1}m_n\dfrac{z^n}{n!}$. Because the classical cumulants $c_n$ satisfy the formula on the left in \eqref{Eq3}, it follows from the Exponential Formula \cite[Chapter 5]{Stanley2} that 
\begin{equation}\label{Eq4}
\sum_{n\geq 1}c_n\frac{z^n}{n!}=\log(1+\widehat M(z)).
\end{equation}
The $R$-transform $R(z)$ of the moment series $M(z)$, which was originally defined by Voiculescu \cite{Voiculescu, Voiculescu2} in his foundational work on free probability, is the ordinary generating function of the free cumulants: 
\[R(z)=\sum_{n\geq 1}\kappa_n z^n.\] The $R$-transform and the moment series are related by the equation  
\begin{equation}\label{Eq5}
\frac{R^{\langle -1\rangle}(z)}{1+z}=M^{\langle -1\rangle}(z),
\end{equation}
where we must choose the appropriate branches of the compositional inverses when $m_1$ and $\kappa_1$ are $0$ \cite[Lecture 16]{Nica}. On the level of generating functions, we can convert from free to classical cumulants (and vice versa) by combining \eqref{Eq6}, \eqref{Eq4}, and \eqref{Eq5}.

\subsection{Cumulants and Valid Hook Configurations}\label{Subsec:CumulantsVHCs} 

We now state and prove the central theorem connecting free and classical cumulants with valid hook configurations. This theorem will represent the second half of the bridge connecting the free probability world with the rooted plane tree (and stack-sorting) world (the first half of this bridge is the Refined Tree Fertility Formula). In fact, much of the heavy lifting needed to prove this theorem was done in \cite{DefantEngenMiller}; we just need to recall the results from that paper. 

Our first order of business is to slightly modify the colorings of valid hook configurations that we introduced in Section~\ref{Sec:TreeFertility}. Recall that, originally, we did not color the northeast endpoints of hooks. Here, it will be convenient to color these points as well. We simply color the northeast endpoint of a hook $H$ the same color as $H$. The top right panel in Figure~\ref{Fig10} shows an example. 

\begin{figure}[ht]
  \begin{center}{\includegraphics[height=8.3cm]{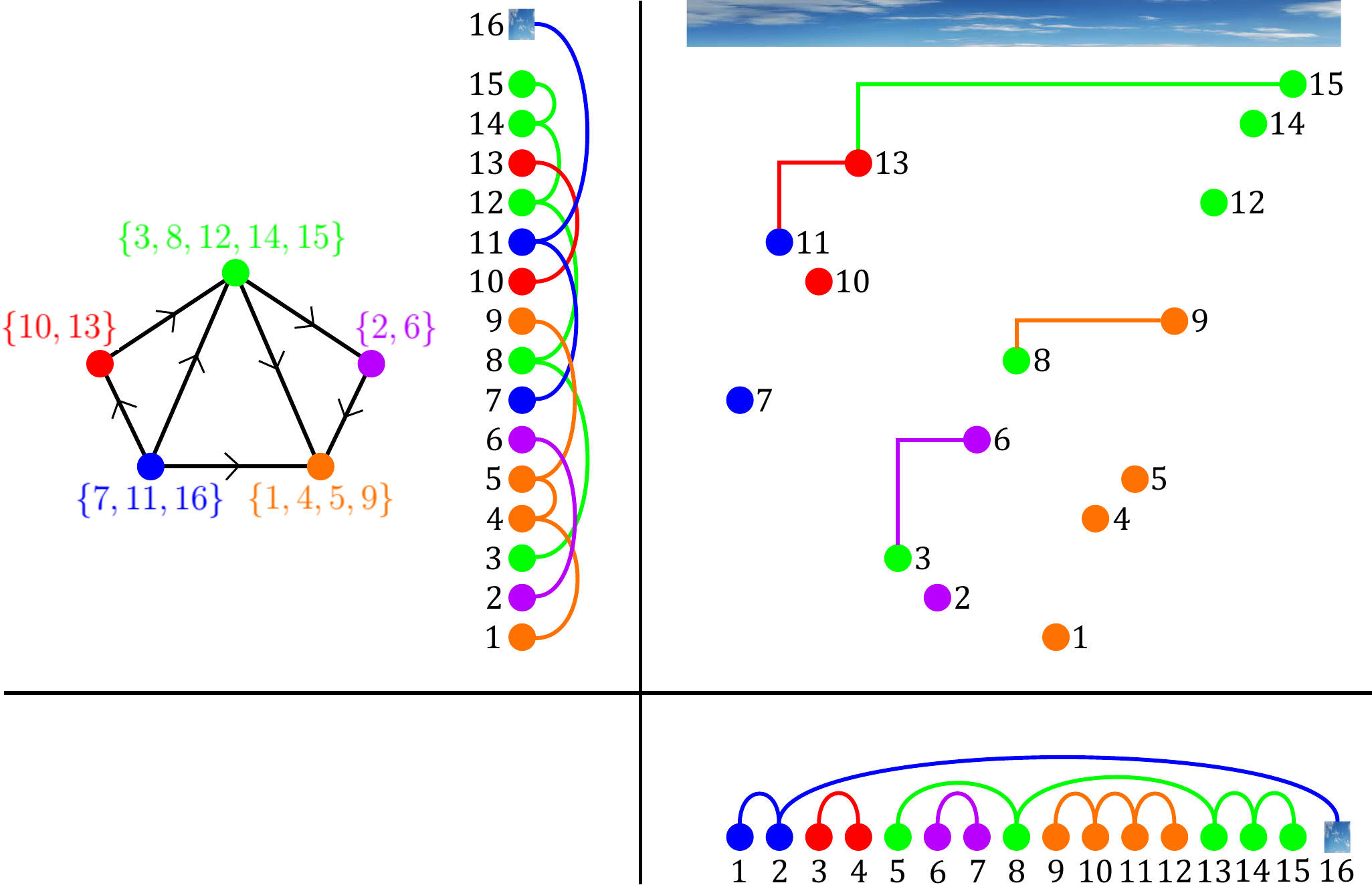}}
  \end{center}
  \caption{The top right panel shows the modified diagram of a valid hook configuration $\mathcal H\in\VHC(\pi)$, where $\pi=7\,11\,10\,13\,3\,2\,6\,8\,1\,4\,5\,9\,12\,14\,15$. The top left panel shows the connected set partition $\vert\mathcal H$ and acyclic orientation of the crossing graph $G(\vert\mathcal H)$ that correspond to $\mathcal H$ under the bijection $\Phi$. The bottom right panel shows the noncrossing partition $\underline{\mathcal H}$}\label{Fig10}
\end{figure}

Let $\pi\in S_{n-1}$ be a permutation, and let $\mathcal H\in\VHC(\pi)$ be a valid hook configuration of $\pi$. Now imagine projecting the colored points and the sky in the modified coloring of $\mathcal H$ onto a vertical wall on the left side of the diagram. This produces a set partition $\vert\mathcal H\in\Pi(n)$. More precisely, we define $\vert\mathcal H$ by saying that two elements $a,a'\in[n]$ are in the same block of $\vert\mathcal H$ if and only if the points with heights $a$ and $a'$ in the modified diagram of $\mathcal H$ have the same color, where we think of the sky as a blue point with height $n$. We color each block $B\in\vert\mathcal H$ the same color as the points whose heights are in $B$. For example, if $\mathcal H$ is as shown in the top right panel in Figure~\ref{Fig10}, then \[\vert\mathcal H=\{{\color{MyOrange}\{1,4,5,9\}},{\color{MyIndigo}\{2,6\}},{\color{MyGreen}\{3,8,12,14,15\}},{\color{MyBlue}\{7,11,16\}},{\color{MyRed}\{10,13\}}\}.\]

\begin{remark}\label{Rem2}
Given a valid hook configuration $\mathcal H$ with $k$ hooks, we can consider the composition ${\bf q}^{\mathcal H}=(q_0,\ldots,q_k)$ defined in Section~\ref{Sec:TreeFertility}. The sizes of the blocks of $\vert\mathcal H$ are $q_0+1,\ldots,q_k+1$. 
\end{remark}

There is another natural partition $\underline{\mathcal H}\in\Pi(n)$ associated to $\mathcal H$, which we get by projecting the points in the modified diagram of $\mathcal H$ downward onto a floor. More precisely, we first declare that two elements $i,i'\in[n-1]$ are in the same block of $\underline{\mathcal H}$ if and only if $(i,\pi_i)$ and $(i',\pi_{i'})$ have the same color in the modified coloring of $\mathcal H$. We then add the number $n$ to the block containing the numbers $i$ such that $(i,\pi_i)$ is blue. We color each block $B\in \underline{\mathcal H}$ the same color as the points whose positions (i.e., $x$-coordinates) are in $B$. For example, if $\mathcal H$ is as shown in the top right panel in Figure~\ref{Fig10}, then \[\underline{\mathcal H}=\{{\color{MyBlue}\{1,2,16\}},{\color{MyRed}\{3,4\}},{\color{MyGreen}\{5,8,13,14,15\}},{\color{MyIndigo}\{6,7\}},{\color{MyOrange}\{9,10,11,12\}}\}.\]

It follows easily from the definition of a valid hook configuration (Definition~\ref{Def5}) that the partition $\underline{\mathcal H}$ is noncrossing. On the other hand, it follows from Theorem~\ref{Thm18} below that $\vert\mathcal H$ is a connected set partition. For each block $B\in\vert\mathcal H$, let $\wideparen B$ be the block in $\underline{\mathcal H}$ with the same color as $B$. The map $B\mapsto\wideparen B$ is clearly a bijection from $\vert\mathcal H$ to $\underline{\mathcal H}$ that preserves sizes of blocks. 

An \dfn{acyclic orientation} of a graph $G$ is an assignment of a direction to each of the edges of $G$ so that there are no directed cycles in the resulting directed graph. A \dfn{source} of a directed graph is a vertex with in-degree $0$. When we speak of a source of an acyclic orientation, we mean a source in the corresponding directed graph. Let $\widetilde \Pi^{\con}(n)$ denote the set of pairs $(\rho,\alpha)$ such that $\rho\in\Pi^{\con}(n)$ and $\alpha$ is an acyclic orientation of the crossing graph $G(\rho)$ whose unique source is the block of $\rho$ containing the number $n$. Greene and Zaslavsky \cite{Greene} proved that if $v$ is a vertex in a finite simple graph $G$, then the number of acyclic orientations of $G$ in which $v$ is the unique source is the value $T_G(1,0)$ of the Tutte polynomial of $G$. It follows that \[T_{G(\rho)}(1,0)=|\{\alpha:(\rho,\alpha)\in\widetilde\Pi^{\con}(n)\}|\] for each $\rho\in\Pi^{\con}(n)$.

We can now state the main bijection from \cite{DefantEngenMiller}. Let $\VHC(S_{n-1})$ denote the set of all valid hook configurations of permutations in $S_{n-1}$. Choose $\mathcal H\in\VHC(S_{n-1})$, and let $\vert\mathcal H$ and $\underline{\mathcal H}$ be its associated connected set partition and its associated noncrossing partition, respectively. Suppose $B$ and $B'$ are two blocks of $\vert\mathcal H$ that are adjacent in $G(\vert\mathcal H)$ (i.e., they form a crossing). Let $\wideparen B$ and $\wideparen B'$ be the corresponding blocks in $\underline{\mathcal H}$. If $\min\wideparen B<\min\wideparen B'$, orient the edge connecting $B$ and $B'$ in $G(\rho)$ from $B$ to $B'$. If $\min\wideparen B'<\min\wideparen B$, orient the edge connecting $B$ and $B'$ in $G(\rho)$ from $B'$ to $B$. After orienting all of the edges of $G(\vert\mathcal H)$ in this way, we obtain an acyclic orientation $\alpha_{\mathcal H}$ of $G(\vert\mathcal H)$. Let \[\Phi(\mathcal H)=(\vert\mathcal H,\alpha_{\mathcal H}).\] For example, if $\mathcal H$ is the valid hook configuration whose modified coloring is shown in the top right panel of Figure~\ref{Fig10}, then $\vert\mathcal H$ and $\alpha_{\mathcal H}$ are shown in the top left panel of the same figure. 

\begin{theorem}[\!\!\cite{DefantEngenMiller}]\label{Thm18}
If $\mathcal H\in\VHC(S_{n-1})$, then $\Phi(\mathcal H)\in\widetilde\Pi^{\con}(n)$. Furthermore, the map \[\Phi:\VHC(S_{n-1})\to\widetilde\Pi^{\con}(n)\] is a bijection. 
\end{theorem}

The following corollary now follows immediately from Josuat-Verg\`es' formula (Theorem~\ref{Thm2}) and the preceding theorem. 

\begin{corollary}[VHC Cumulant Formula]\label{Cor4}
If $(\kappa_n)_{n\geq 1}$ is a sequence of free cumulants, then the corresponding classical cumulants are given by \[-c_n=\sum_{\mathcal H\in\VHC(S_{n-1})}(-\kappa_\bullet)_{\vert\mathcal H}.\] 
\end{corollary}

Because Josuat-Verg\`es' formula extends to the multivariate setting, so does Corollary~\ref{Cor4}. More precisely, this means that if $(\mathcal A,\varphi)$ is a noncommutative probability space and $a_1,\ldots,a_n\in\mathcal A$, then 
\begin{equation}\label{Eq7}
-c_n(a_1,\ldots,a_n)=\sum_{\mathcal H\in\VHC(S_{n-1})}\prod_{B\in\vert\mathcal H}(-\kappa_{|B|}(a_B)).
\end{equation} In Section~\ref{Sec:OtherFormulas}, we will use Corollary~\ref{Cor4} and the combinatorics of valid hook configurations to give new formulas that convert from free to classical cumulants; these new formulas will \emph{not} extend to the multivariate setting.   

\begin{example}\label{Exam2}
Figure~\ref{Fig11} shows the modified colorings of the two valid hook configurations in $\VHC(S_3)$. The associated connected set partitions are $\{{\color{MyBlue}\{1,2,3,4\}}\}$ and $\{{\color{MyRed}\{1,3\}},{\color{MyBlue}\{2,4\}}\}$. If $(\mathcal A,\varphi)$ is a noncommutative probability space and $a_1,a_2,a_3,a_4\in\mathcal A$, then it follows from \eqref{Eq7} that 
\[-c_4(a_1,a_2,a_3,a_4)=-\kappa_4(a_1,a_2,a_3,a_4)+(-\kappa_2(a_1,a_3))(-\kappa_2(a_2,a_4))\] \[=-\kappa_4(a_1,a_2,a_3,a_4)+\kappa_2(a_1,a_3)\kappa_2(a_2,a_4).\] Specializing to the univariate setting, this says that \[-c_4=-\kappa_4+\kappa_2^2.\] 

\begin{figure}[ht]
  \begin{center}{\includegraphics[height=1.4cm]{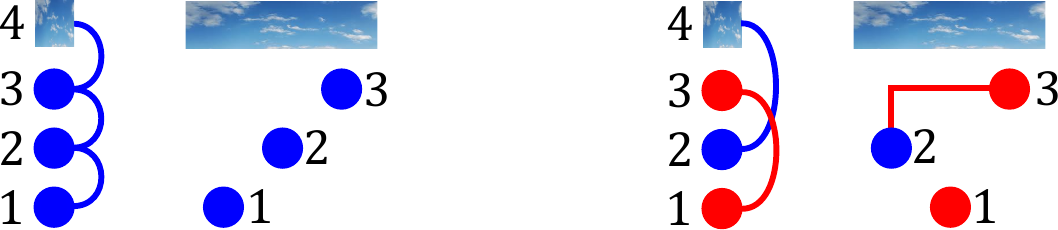}}
  \end{center}
  \caption{The modified colorings of the valid hook configurations of permutations in $S_3$ and their associated connected set partitions.}\label{Fig11}\qedhere
\end{figure}
\end{example}

\section{Troupes and Cumulants}\label{Sec:TroupsAndCumulants}
We can now state and prove one of our main results. We will then illustrate this theorem with several examples. In what follows, we work over the field $\mathbb K=\mathbb C(x_1,\ldots,x_r)$. Given a set ${\bf T}$ of colored binary plane trees, let $\overline{\mathsf D}{\bf T}$ denote the set of standardized trees in $\mathsf D{\bf T}$.

\begin{theorem}\label{Thm5}
Let ${\bf T}$ be a troupe. Let $f_1,\ldots,f_r$ be insertion-additive tree statistics, and let $x_1,\ldots,x_r$ be variables. If $(\kappa_n)_{n\geq 1}$ is the sequence of free cumulants defined by \[\kappa_n=-\sum_{T\in{\bf T}_{n-1}}x_1^{f_1(T)}\cdots x_r^{f_r(T)},\] then the corresponding sequence $(c_n)_{n\geq 1}$ of classical cumulants is given by \[c_n=-\sum_{\mathcal T\in\overline{\mathsf D}{\bf T}_{n-1}}x_1^{\ddot f_1(\mathcal T)}\cdots x_r^{\ddot f_r(\mathcal T)}.\] 
\end{theorem}

\begin{proof}
Preserving the notation from the Refined Tree Fertility Formula (Theorem~\ref{Thm22}), we let ${\bf G}_m(x_1,\ldots,x_r)=-\kappa_{m+1}$. If $\mathcal H$ is a valid hook configuration with ${\bf q}^{\mathcal H}=(q_0,\ldots,q_k)$, then the blocks of $\vert\mathcal H$ have sizes $q_0+1,\ldots,q_k+1$. Thus, \[{\bf G}_{{\bf q}^{\mathcal H}}(x_1,\ldots,x_r)=\prod_{t=0}^k{\bf G}_{q_t}(x_1,\ldots,x_r)=\prod_{t=0}^k(-\kappa_{q_t+1})=(-\kappa_\bullet)_{\vert\mathcal H}.\] Combining the Refined Tree Fertility Formula (Theorem~\ref{Thm22}) with the VHC Cumulant Formula (Corollary~\ref{Cor4}), we find that \[-c_n=\sum_{\mathcal H\in\VHC(S_{n-1})}(-\kappa_\bullet)_{\vert\mathcal H}=\sum_{\pi\in S_{n-1}}\sum_{\mathcal H\in\VHC(\pi)}{\bf G}_{\bf q^{\mathcal H}}(x_1,\ldots,x_r)\] \[=\sum_{\pi\in S_{n-1}}\sum_{\mathcal T\in\mathcal P^{-1}(\pi)\cap\mathsf D{\bf T}}x_1^{\ddot f_1(\mathcal T)}\cdots x_r^{\ddot f_r(\mathcal T)}=\sum_{\mathcal T\in\overline{\mathsf D}{\bf T}_{n-1}}x_1^{\ddot f_1(\mathcal T)}\cdots x_r^{\ddot f_r(\mathcal T)}.\qedhere\] 
\end{proof}

\begin{remark}
Despite the very simple-looking relationship between the free and classical cumulants in Theorem~\ref{Thm5}, we do not see any way to prove the result without the use of the Refined Tree Fertility Formula (which, in turn, relies on the Refined Tree Decomposition Lemma) and the VHC Cumulant Formula. It is likely that these tools are truly necessary for the proof because the hypothesis that ${\bf T}$ is a troupe is needed. Indeed, suppose we were to take ${\bf T}$ to be the set of all binary plane trees that are branches (meaning no vertices have $2$ children). Set $r=0$. In this case, the free cumulants $\kappa_n$ in Theorem~\ref{Thm5} satisfy $\kappa_1=-|{\bf T}_0|=-1$ and $\kappa_n=-|{\bf T}_{n-1}|=-2^{n-2}$ for all $n\geq 2$. If Theorem~\ref{Thm5} applied in this setting, it would tell us that $c_4=-|\overline{\mathsf D}{\bf T}_3|=-4$. However, we can use the computation in Example~\ref{Exam2} to see that $-c_4=-\kappa_4+\kappa_2^2=4+(-1)^2=5$.  
\end{remark} 

\begin{example}[Binary Plane Trees]\label{ExamBinary2}
Let ${\bf T}=\mathsf{BPT}$, and let $f_1(T)=\des(T)+1$ and $f_2(T)=\peak(T)+1$. Let \[{\bf G}_n(x_1,x_2)=\sum_{T\in\mathsf{BPT}_n}x_1^{\des(T)+1}x_2^{\peak(T)+1}\] be as in Example~\ref{ExamBinary1}. The generating function for these polynomials is given explicitly in \eqref{Eq16}. Theorem~\ref{Thm5} tells us that if we define free cumulants by $\kappa_n=-{\bf G}_{n-1}(x_1,x_2)$, then the corresponding classical cumulants are \[c_n=-\sum_{\mathcal T\in\mathsf{\overline{D}BPT}_{n-1}}x_1^{\des(\mathcal T)+1}x_2^{\peak(\mathcal T)+1}.\] 

Using the in-order reading $\mathcal I$, we can rephrase this result purely in terms of permutations. We say a permutation $\pi=\pi_1\cdots\pi_n$ is \dfn{$231$-avoiding} if there do not exist indices $i_1<i_2<i_3$ such that $\pi_{i_3}<\pi_{i_1}<\pi_{i_2}$. Given a binary plane tree $T$ with $n$ vertices, there is a unique decreasing binary plane tree $\ddot T$ with skeleton $T$ and postorder reading $123\cdots n$. It is well known that the map $T\mapsto \mathcal I(\ddot T)$ is a bijection from $\mathsf{BPT}_n$ to the set $\Av_n(231)$ of $231$-avoiding permutations in $S_n$. Furthermore, $\des(T)=\des(\mathcal I(\ddot T))$ and $\peak(T)=\peak(\mathcal I(\ddot T))$. On the other hand, the map $\mathcal I:\mathsf{\overline{D}BPT}_n\to S_n$ is a bijection satisfying $\des(\mathcal I(\mathcal T))=\des(\mathcal T)$ and $\peak(\mathcal I(\mathcal T))=\peak(\mathcal T)$. It follows that the above free and classical cumulants are \[\kappa_n=-\sum_{\pi\in\Av_{n-1}(231)}x_1^{\des(\pi)+1}x_2^{\peak(\pi)+1}\quad\text{and}\quad c_n=-\sum_{\pi\in S_{n-1}}x_1^{\des(\pi)+1}x_2^{\peak(\pi)+1}.\] 

Let us now specialize by setting $x_2=1$. In this case, the free cumulants are given by Narayana polynomials: 
\begin{equation}\label{Eq8}
\kappa_n=-{\bf G}_{n-1}(x_1,1)=-N_{n-1}(x_1).
\end{equation} Using the above expression for the classical cumulants in terms of permutations, we find that 
\begin{equation}\label{Eq9}
c_n=-\sum_{\pi\in S_{n-1}}x_1^{\des(\pi)+1}=-x_1A_{n-1}(x_1),
\end{equation} where $A_{n-1}(x_1)$ is an \dfn{Eulerian polynomial} (see OEIS sequence A008292). The Eulerian polynomials are the $h$-polynomials of permutohedra \cite{Postnikov}. Thus, we have shown, in a combinatorial fashion, that the above sequence of free cumulants given by Narayana polynomials, which are the $h$-polynomials of associahedra, corresponds to a sequence of classical cumulants given by Eulerian polynomials, which are the $h$-polynomials of permutohedra. 

If we specialize to the case in which $x_1=1$, then the free cumulants count $231$-avoiding standardized permutations according to their number of peaks (see OEIS sequence A091894) and the classical cumulants count arbitrary standardized permutations according to their number of peaks (see OEIS sequence A008303). 

Finally, we can specialize to the case $x_1=x_2=1$. Here, the free and classical cumulants are \[\kappa_n=-C_{n-1}\quad\text{and}\quad c_n=-(n-1)!.\] The corresponding moment sequence $(m_n)_{n\geq 1}$ is simply $-1,0,0,0,\ldots$. The fact that the free cumulants $-C_{n-1}$ correspond to the classical cumulants $-(n-1)!$ is well known; it follows from the fact that the numbers $(-1)^{n-1}C_{n-1}$ and $(-1)^{n-1}(n-1)!$ are the M\"obius invariants of noncrossing partition lattices and partition lattices, respectively. What is nontrivial is our combinatorial explanation of this correspondence, which relies on valid hook configurations and binary plane trees.  
\end{example}

\begin{example}[Full Binary Plane Trees]
Let ${\bf T}=\mathsf{FBPT}$. Theorem~\ref{Thm5} tells us that if we define free cumulants by $\kappa_n=-|\mathsf{FBPT}_{n-1}|=-C_{(n-2)/2}$ (where $C_{(n-2)/2}=0$ when $n$ is odd), then the corresponding classical cumulants are $c_n=-|\mathsf{\overline{D}FBPT}_{n-1}|$. Note that $c_n=0$ when $n$ is odd. Suppose $n$ is even. As mentioned in Example~\ref{ExamFull1}, the in-order reading gives a bijection from $\mathsf{\overline{D}FBPT}_{n-1}$ to the set $S_{n-1}\cap\ALT$ of alternating permutations in $S_{n-1}$. Thus, $c_n=-E_{n-1}$, where $E_{n-1}$ denotes an Euler number (also called a tangent number since $n-1$ is odd). 
\end{example}

\begin{example}[Motzkin Trees]
Let ${\bf T}=\mathsf{Mot}$, and let $f_1(T)=\des(T)+1$. Let \[{\bf G}_n(x_1)=\sum_{T\in\mathsf{Mot}_n}x_1^{\des(T)+1}=M_{n-1}(x_1)\] be the Motzkin polynomials. The generating function for these polynomials and an explicit formula for their coefficients are given in Example~\ref{ExamMotzkin1}. Theorem~\ref{Thm5} tells us that if we define free cumulants by $\kappa_n=-M_{n-2}(x_1)$, then the corresponding classical cumulants are \[c_n=-\sum_{\mathcal T\in\mathsf{\overline{D}Mot}_{n-1}}x_1^{\des(\mathcal T)+1}.\] Referring back to the bijections $\mathsf{DMot}\to\EDP$ and $\mathsf{BPT}_{n-1}\to\Av_{n-1}(231)$ mentioned in Example~\ref{ExamMotzkin1} and Example~\ref{ExamBinary2}, respectively, we find that we can write \[\kappa_n=-\sum_{\pi\in\Av_{n-1}(231)\cap\EDP}x_1^{\des(\pi)+1}\quad\text{and}\quad c_n=-\sum_{\pi\in S_{n-1}\cap\EDP}x_1^{\des(\pi)+1}.\] The coefficients of the polynomials $-c_n$ form the $\gamma$-vectors of permutohedra \cite{Shapiro, Postnikov}; they appear in the OEIS sequence A101280 \cite{OEIS}. Thus, we have shown, in a combinatorial fashion, that free cumulants given by Motzkin polynomials, which are the $\gamma$-polynomials of associahedra, correspond to classical cumulants given by the $\gamma$-polynomials of permutohdera. 

If we specialize to the case $x_1=1$ and define the sequence of free cumulants by $\kappa_n=-M_{n-2}$, then the corresponding classical cumulants are $c_n=-|\mathsf{\overline{D}Mot}_{n-1}|=-|S_{n-1}\cap\EDP|$. The numbers $-c_n$ form the OEIS sequence A080635 \cite{OEIS}. 
\end{example}

\begin{example}[Schr\"oder $2$-Colored Binary Trees]
Let ${\bf T}=\mathsf{Sch}$, and let $f_1(T)=\des(T)+1$, $f_2(T)=\peak(T)+1$, and $f_3(T)=\black(T)+1$. Let \[{\bf G}_n(x_1,x_2,x_3)=\sum_{T\in\mathsf{Sch}_n}x_1^{\des(T)+1}x_2^{\peak(T)+1}x_3^{\black(T)+1}\] be as in Example~\ref{ExamSchroder1}. Theorem~\ref{Thm5} tells us that the free cumulants $\kappa_n=-{\bf G}_{n-1}(x_1,x_2,x_3)$ correspond to the classical cumulants \[c_n=-\sum_{\mathcal T\in\mathsf{\overline{D}Sch}_{n-1}}x_1^{\des(\mathcal T)+1}x_2^{\peak(\mathcal T)+1}x_3^{\black(T)+1}.\] 

Let us consider the specialization $x_2=x_3=1$. In this case, $-\kappa_n$ counts Schr\"oder $2$-colored binary trees with $n-1$ vertices according to their number of right edges. As mentioned in Example~\ref{ExamSchroder1}, \[\kappa_n=-{\bf G}_{n-1}(x_1,1,1)=-N_{n-1}(2x_1),\] where $N_{n-1}(x)$ is a Narayana polynomial. The corresponding classical cumulants are given by \[c_n=-\sum_{\mathcal T\in\mathsf{\overline{D}Sch}_{n-1}}x_1^{\des(\mathcal T)+1}.\] Notice that we can appeal directly to \eqref{Eq8} and \eqref{Eq9} (replacing $x_1$ with $2x_1$) to see that we also have $c_n=-2x_1A_{n-1}(2x_1)$, where $A_{n-1}(x)$ is an Eulerian polynomial. This yields the following enumerative corollary, which appears to be new.  

\begin{corollary}\label{Cor13}
For every $n\geq 1$, we have \[\sum_{T\in \mathsf{\overline{D}Sch_{n-1}}}x^{\des(T)+1}=2xA_{n-1}(2x).\]
\end{corollary} 

We can also consider the specialization $x_1=x_2=1$. By \cite[Corollary 4.2]{Gu}, we have \[\kappa_n=-\sum_{T\in\mathsf{Sch}_{n-1}}x_3^{\black(T)+1}=-\sum_{j=0}^n\binom{n+j}{n-j}C_jx_3^{j+1},\] from which one can compute the $R$-transform \[R(z)=\sum_{n\geq 1}\kappa_nz^n=
-\frac{1-z-\sqrt{(1-z)^2-4x_3z}}{2}.\] Invoking \eqref{Eq5}, one can show that the corresponding moments are given by $m_n=-x_3$ for all $n\geq 1$. According to \eqref{Eq4}, we have \[\sum_{n\geq 1}c_n\frac{z^n}{n!}=\log(1-x_3(e^{z}-1)).\] On the other hand, \[c_n=-\sum_{\mathcal T\in\mathsf{\overline{D}Sch}_{n-1}}x_3^{\black(\mathcal T)+1}.\] This proves the following enumerative corollary. 

\begin{corollary}\label{Cor12}
We have \[\sum_{n\geq 1}\sum_{T\in \mathsf{\overline{D}Sch}_{n-1}}x_3^{\black(T)+1}\frac{z^n}{n!}=-\log(1-x_3(e^z-1)).\] In particular, \[\sum_{n\geq 1}|\mathsf{\overline{D}Sch}_{n-1}|\frac{z^n}{n!}=-\log(2-e^z).\]
\end{corollary} 

The series $-\log(1-x_3(e^z-1))$ is the exponential generating function of the triangle of numbers appearing as OEIS sequence A028246 \cite{OEIS}. These numbers have numerous known properties, including close connections with Bernoulli numbers, but none of them seem to have much to do with trees. Hence, Corollary~\ref{Cor12} appears to give a new combinatorial interpretation of these numbers. Namely, it tells us that the rows in the triangle count (standardized) decreasing Schr\"oder $2$-colored binary trees according to the number of black vertices. Setting $x_3=1$, we obtain the second part of Corollary~\ref{Cor12}, which tells us that (standardized) decreasing Schr\"oder $2$-colored binary trees are counted by the numbers appearing in OEIS sequence A000629. These numbers count many objects, including cyclically ordered set partitions, but this combinatorial interpretation in terms of decreasing Schr\"oder $2$-colored binary trees seems to be new. 
\end{example}

\section{Further Applications}\label{Sec:Applications} 
We now apply the machinery developed in the preceding sections to answer several other natural questions about valid hook configurations, rooted plane trees, and the stack-sorting map. The basic idea is to start by choosing an appropriate sequence of free cumulants $\kappa_n$ from an appropriate field $\mathbb K$. For our purposes, it will suffice to take $\mathbb K=\mathbb C(x)$. We then combine the VHC Cumulant Formula (Corollary~\ref{Cor4}) with the Refined Tree Fertility Formula (Theorem~\ref{Thm22}) or one of its corollaries to obtain combinatorial information about the corresponding classical cumulants $c_n$. Then, by combining \eqref{Eq6}, \eqref{Eq4}, and \eqref{Eq5}, we can gain information about the exponential generating function $\sum_{n\geq 1}c_n\dfrac{z^n}{n!}$. 

When computing asymptotic formulas for the sequences arising in the this section, it will be useful to keep in mind the following lemma, which is a standard result in singularity analysis (see \cite[Chapter IV]{Flajolet}). 
\begin{lemma}\label{Lem1}
Let $f:\mathbb C\to\mathbb C$ be a meromorphic function with a simple pole at the complex number $c$. Suppose that $f$ is holomorphic at every complex number $z$ such that $|z|\leq |c|$ and $z\neq c$. Letting $[z^n]f(z)$ denote the coefficient of $z^n$ in the power series expansion of $f(z)$ about the origin, we have $\displaystyle [z^n]f(z)\sim- c^{-n-1}\Res_{z=c}f(z)$. 
\end{lemma} 

Let $\#\rho$ denote the number of blocks in a set partition $\rho$. Let $\hook(\mathcal H)$ denote the number of hooks in a valid hook configuration $\mathcal H$. Throughout our applications, it will be useful to keep in mind the fact that if $\mathcal H$ is a valid hook configuration of a permutation $\pi$, then \[\#\vert\mathcal H=\hook(\mathcal H)+1=\des(\pi)+1.\]

\subsection{Counting Valid Hook Configurations}
Our first application in this section will be the only one that does not make use of the Refined Tree Fertility Formula or one of its corollaries. We are going to compute a generating function that counts valid hook configurations according to their number of hooks. The following theorem is implicit in \cite{DefantEngenMiller}.  

\begin{theorem}\label{Thm3}
We have \[\sum_{n\geq 1}\sum_{\mathcal H\in\VHC(S_{n-1})}x^{\hook(\mathcal H)+1}\frac{z^n}{n!}=-\log\left(1-x\int \frac{e^{(1-x)z}J_1(2z\sqrt{x})}{z\sqrt{x}}\,dz\right),\] where $J_1$ is a Bessel function of the first kind and the indefinite integral is taken so that it approaches $0$ as $z\to 0$.   
\end{theorem} 

\begin{proof}
Define a sequence $(\kappa_n)_{n\geq 1}$ of free cumulants by letting $\kappa_n=-x$ for all $n\geq 1$. Let $(m_n)_{n\geq 1}$ and $(c_n)_{n\geq 1}$ be the corresponding sequences of moments and classical cumulants, which satisfy \eqref{Eq3}. It is well known that the Narayana number $N(n,r)$ is equal to the number of partitions in $\NC(n)$ with $r$ blocks \cite[Lecture 9]{Nica}. It follows from \eqref{Eq3} that \[m_n=\sum_{\eta\in\NC(n)}(-x)^{\#\eta}=\sum_{r=1}^n N(n,r)(-x)^r,\] so the moment series is $\displaystyle M(z)=\sum_{n\geq 1}\sum_{r=1}^n N(n,r)(-x)^rz^n$. It is known \cite{OEIS} that \[\sum_{n\geq 1}\sum_{r=1}^nN(n,r)y^r\frac{z^n}{n!}=y\int \frac{e^{(1+y)z}J_1(2z\sqrt{-y})}{z\sqrt{-y}}\,dz,\] so \[\widehat M(z)=\sum_{n\geq 1}\sum_{r=1}^n N(n,r)(-x)^r\frac{z^n}{n!}=-x\int \frac{e^{(1-x)z}J_1(2z\sqrt{x})}{z\sqrt{x}}\,dz.\] Now, equation \eqref{Eq4} tells us that \[\sum_{n\geq 1}c_n\frac{z^n}{n!}=\log(1+\widehat M(z))=\log\left(1-x\int \frac{e^{(1-x)z}J_1(2z\sqrt{x})}{z\sqrt{x}}\,dz\right).\]
To finish the proof, we invoke Corollary~\ref{Cor4}, which tells us that \[-c_n=\sum_{\mathcal H\in\VHC(S_{n-1})}(-\kappa_\bullet)_{\vert\mathcal H}=\sum_{\mathcal H\in\VHC(S_{n-1})}x^{\#\vert\mathcal H}=\sum_{\mathcal H\in\VHC(S_{n-1})}x^{\hook(\mathcal H)+1}. \qedhere\] 
\end{proof}

\begin{corollary}\label{Cor10}
We have \[\sum_{n\geq 1}|\VHC(S_{n-1})|\frac{z^n}{n!}=-\log\left(1-z\,{}_1\hspace{-.03cm}F_2\left(\frac{1}{2};\frac{3}{2},2;-z^2\right)\right),\] where ${}_1\hspace{-.03cm}F_2$ denotes a generalized hypergeometric function. 
\end{corollary} 

\begin{proof}
Put $x=1$ in Theorem~\ref{Thm3}, and use the fact that \[\int \frac{J_1(2z)}{z}\,dz=z\,{}_1\hspace{-.03cm}F_2\left(\frac{1}{2};\frac{3}{2},2;-z^2\right). \qedhere\]
\end{proof} 

We can derive from Corollary~\ref{Cor10} the following asymptotic formula for $|\VHC(S_n)|$. 

\begin{corollary}\label{Cor14}
As $n\to\infty$, we have $|\VHC(S_n)|\sim n!/c^{n+1}$, where $c\approx 1.32874$ is the smallest positive real root of $\displaystyle 1-z\,{}_1\hspace{-.03cm}F_2\left(\frac{1}{2};\frac{3}{2},2;-z^2\right)$.  
\end{corollary}

\begin{proof}
Let $\displaystyle Q(z)=1-z\,{}_1\hspace{-.03cm}F_2\left(\frac{1}{2};\frac{3}{2},2;-z^2\right)$, and let $a=\displaystyle\frac{1}{2\pi i}\int_{|z|=1.4}\frac{Q'(z)}{Q(z)}\,dz$. The function $Q(z)$ is entire, so it follows from the argument principle from complex analysis that $a$ is the number of zeros of $Q(z)$ with absolute value less than $1.4$, counted with multiplicity. One can show\footnote{To prove this rigorously, one can first expand $Q'(z)$ and $Q(z)$ in series in order to estimate $\dfrac{Q'(z)}{Q(z)}$ with a sufficiently small explicit error that holds uniformly for all $z$ with $|z|=1.4$. One can then numerically estimate the integral to show that $|a-1|<1$. Since $a$ must be an integer, it must be $1$. We omit the details of this computation.} that $a=1$. We can compute that the unique root of $Q(z)$ with absolute value less than $1.4$ is $c\approx 1.32874$. This means that $Q'(z)/Q(z)$ is a meromorphic function whose only pole in the disc $\{z\in\mathbb C:|z|<1.4\}$ is a simple pole at $z=c$. It follows from Corollary~\ref{Cor10} and Lemma~\ref{Lem1} that \[|\VHC(S_n)|=n!\,[z^n]\left(\frac{\partial}{\partial z}\sum_{m\geq 1}|\VHC(S_{m-1})|\frac{z^m}{m!}\right)=n!\,[z^n]\left(\frac{\partial}{\partial z}(-\log(Q(z)))\right)\] \[=-n!\,[z^n]\left(\frac{Q'(z)}{Q(z)}\right)\sim n!c^{-n-1}\Res_{z=c}\left(\frac{Q'(z)}{Q(z)}\right)=n!c^{-n-1}. \qedhere\]
\end{proof}

Using Corollary~\ref{Cor10}, one can compute that \[(|\VHC(S_{n-1})|)_{n\geq 1}=1, 1, 1, 2, 6, 22, 99, 520, 3126, 21164, 159226, 1318000, 11902268, 116444668, \ldots.\] Upon inspection of the initial values of this sequence, we arrive at the following conjecture. 

\begin{conjecture}\label{Conj1}
If $n\geq 3$, then $|\VHC(S_{n-1})|$ is odd if and only if $n+1$ is a power of $2$. 
\end{conjecture}

The enumeration of valid hook configurations avoiding certain patterns has been initiated in \cite{DefantMotzkin} and extended in \cite{Maya}. 

\subsection{Uniquely Sorted Permutations}\label{Subsec:UniquelySorted}
In this section, we will see that the Fertility Formula and the VHC Cumulant Formula allow us to understand \dfn{uniquely sorted} permutations, which are permutations that have exactly one preimage under the stack-sorting map $s$. We make the convention that the empty permutation is not uniquely sorted. The results in this section are taken from \cite{DefantEngenMiller}. Our goal is simply to illustrate the use of the Fertility Formula and the VHC Cumulant Formula, so we will sketch the main ideas and refer the reader to \cite{DefantEngenMiller} for a more detailed treatment. 

Recall that a permutation is called \emph{sorted} if it is in the image of the stack-sorting map $s$. It is known \cite[Chapter 8, Exercise 18]{Bona} that if $\pi=\pi_1\cdots\pi_n$ is nonempty and sorted, then $\des(\pi)\leq\dfrac{n-1}{2}$. It is natural to ask what can be said about the permutations that achieve this upper bound. The following theorem from \cite{DefantEngenMiller} answers this question. 

\begin{theorem}[\!\!\cite{DefantEngenMiller}]\label{Thm4}
A permutation $\pi=\pi_1\cdots\pi_n$ is uniquely sorted if and only if it is sorted and has exactly $\dfrac{n-1}{2}$ descents. In particular, every uniquely sorted permutation has odd length. 
\end{theorem}

\begin{proof}
Let $k=\des(\pi)$. Suppose $\pi$ is uniquely sorted. Certainly $\pi$ is sorted, so we must show that $n=2k+1$. It follows from the Fertility Formula \eqref{Eq22} that $\pi$ has a unique valid hook configuration $\mathcal H$ and that ${\bf q}^{\mathcal H}=(1,\ldots,1)$ consists of only $1$'s. When we defined ${\bf q}^{\mathcal H}$ in Section~\ref{Sec:TreeFertility}, we remarked that it is a composition of $n-k$ into $k+1$ parts. It follows that $n=2k+1$. 

Conversely, assume $\pi$ is sorted and $n=2k+1$. Because $\pi$ is sorted, it follows from the Fertility Formula that $\pi$ has a valid hook configuration $\mathcal H$. The only composition of $n-k$ into $k+1$ parts is the tuple $(1,\ldots,1)$ with $k+1$ $1$'s. Using Remark~\ref{Rem1}, we find that $\mathcal H$ is the unique valid hook configuration of $\pi$ and that ${\bf q}^{\mathcal H}=(1,\ldots,1)$. It now follows from \eqref{Eq22} that $|s^{-1}(\pi)|=1$, so $\pi$ is uniquely sorted. 
\end{proof}

Let $\mathcal U_{n}$ denote the set of uniquely sorted permutations in $S_n$, and let $\VHC(\mathcal U_n)=\bigcup_{\pi\in\mathcal U_n}\VHC(\pi)$. We have seen that each uniquely sorted permutation has a unique valid hook configuration, so we obtain a bijection $\mathcal U_n\to\VHC(\mathcal U_n)$ by sending each uniquely sorted permutation to its valid hook configuration. A \dfn{matching} is a set partition in which each block has size $2$. Referring to Remark~\ref{Rem2} and the above proof of Theorem~\ref{Thm4}, we see that $\VHC(\mathcal U_n)$ is precisely the set of valid hook configurations in $\VHC(S_n)$ such that $\vert\mathcal H$ is a matching. Therefore, if we define a sequence $(\kappa_n)_{n\geq 1}$ of free cumulants by $\kappa_2=-1$ and $\kappa_n=0$ for all $n\neq 2$, then we can use the VHC Cumulant Formula (Corollary~\ref{Cor4}) to see that \[-c_n=\sum_{\mathcal H\in\VHC(S_{n-1})}(-\kappa_\bullet)_{\vert\mathcal H}=\sum_{\mathcal H\in\VHC(\mathcal U_{n-1})}1=|\VHC(\mathcal U_{n-1})|=|\mathcal U_{n-1}|.\]
 
Let $\NC^{(2)}(n)$ denote the set of noncrossing matchings of $\{1,\ldots,n\}$. It is clear that $|\NC^{(2)}(n)|=0$ when $n$ is odd, and it is known \cite[Lecture 8]{Nica} that $|\NC^{(2)}(2k)|=C_k$ for all integers $k\geq 1$. Referring to \eqref{Eq3} again, we find that $m_n=0$ when $n$ is odd and that \[m_{2k}=\sum_{\eta\in\NC(2k)}(\kappa_\bullet)_\eta=\sum_{\eta\in\NC^{(2)}(2k)}(-1)^k=(-1)^kC_k\quad\text{for all }k\geq 1.\] This means that \[\widehat M(z)=\sum_{n\geq 1}m_n\frac{z^n}{n!}=\sum_{k\geq 1}(-1)^kC_k\frac{z^{2k}}{(2k)!}=\frac{J_1(2z)}{z}-1,\] where $J_1$ is a Bessel function of the first kind. Finally, we can use \eqref{Eq4} to see that \[\sum_{n\geq 1}|\mathcal U_{n-1}|\frac{z^n}{n!}=-\sum_{n\geq 1}c_n\frac{z^n}{n!}=-\log\left(\frac{J_1(2z)}{z}\right).\]
It is known \cite{OEIS} that $\displaystyle-\log\left(\dfrac{J_1(2z)}{z}\right)=\sum_{k\geq 1}\mathscr{A}_k\dfrac{z^{2k}}{(2k)!}$, where $(\mathscr{A}_k)_{k\geq 1}$ is known as \emph{Lassalle's sequence}. This fascinating sequence first emerged in \cite{Lassalle}, where Lassalle proved that its terms are positive and increasing, settling a conjecture of Zeilberger. The first few terms of Lassalle's sequence are \[
	1, 1, 5, 56, 1092, 32670, 1387815, 79389310, 5882844968, 548129834616, 62720089624920.
\]
Thus, we have the following theorem. 
\begin{theorem}[\!\!\cite{DefantEngenMiller}]\label{Thm7}
For every $k\geq 0$, we have $|\mathcal U_{2k+1}|=\mathscr{A}_{k+1}$. 
\end{theorem}

We refer the reader to \cite{Josuat, Lassalle, Tevlin} for more information about Lassalle's sequence and to \cite{DefantEngenMiller} for more information about uniquely sorted permutations. The investigation of pattern-avoiding uniquely sorted permutations was initiated in \cite{DefantCatalan} and extended in \cite{Hanna}. 

\subsection{Descents in Sorted Permutations}\label{Subsec:DescentsSorted}
Several papers have investigated the relationship between the permutation statistic $\des$, which counts the descents of a permutation, and the stack-sorting map $s$. For example, it is known \cite[Chapter 8, Exercise 18]{Bona} that $0\leq\des(s(\sigma))\leq\dfrac{n-1}{2}$ for every $\sigma\in S_n$. Of course, the lower bound of $0$ is tight since $s(\sigma)$ could be the identity permutation $123\cdots n$. Knuth's \cite{Knuth} characterization and enumeration of the permutations $\sigma\in S_n$ such that $\des(s(\sigma))=0$ was the first result about stack-sorting; it also initiated the study of permutation patterns and the kernel method \cite{Banderier, Bona, Kitaev, Linton}. As mentioned in Section~\ref{Subsec:UniquelySorted}, the upper bound is tight (for $n$ odd) and is attained when $s(\sigma)$ is uniquely sorted. It is natural to ask for the expected value of $\des(s(\sigma))$ when $\sigma\in S_n$ is chosen uniformly at random. To simplify some of the formulas, we will actually consider the problem of computing the expected value $\mathbb E(D_n)$, where $D_n=\des(s(\sigma))+1$ and $\sigma$ is chosen uniformly at random from $S_{n-1}$. 

In what follows, recall the notation from \eqref{Eq6}, \eqref{Eq4}, and \eqref{Eq5}. Let us define 
\begin{equation}\label{Eq27}
F_x(z)=\frac{1}{2}\left(-x-x^2z+x\sqrt{1-4z+2xz+x^2z^2}\right).
\end{equation} 
We choose the branch of the square root that evaluates to $1$ when $z\to 0$. 
We view $F_x(z)$ as a power series in the variable $z$ with coefficients in $\mathbb C(x)$. In particular, copying \eqref{Eq6}, we have $\widehat F_x(z)=\mathcal L^{-1}\{F_x(1/t)/t\}(z)$, where the inverse Laplace transform is taken with respect to the variable $t$. 

\begin{theorem}\label{Thm8}
We have \[\sum_{n\geq 1}\left(\sum_{\sigma\in S_{n-1}}x^{\des(s(\sigma))+1}\right)\frac{z^n}{n!}=-\log(1+\widehat F_x(z)).\]
\end{theorem}

\begin{proof}
Define a sequence $(\kappa_n)_{n\geq 1}$ of free cumulants by $\kappa_n=-xC_{n-1}$. Let $(m_n)_{n\geq 1}$ and $(c_n)_{n\geq 1}$ be the corresponding sequences of moments and classical cumulants, respectively. According to the VHC Cumulant Formula (Corollary~\ref{Cor4}), we have \[-c_n=\sum_{\mathcal H\in\VHC(S_{n-1})}(-\kappa_\bullet)_{\vert\mathcal H}=\sum_{\mathcal H\in\VHC(S_{n-1})}x^{\#\vert\mathcal H}(C_{\bullet-1})_{\vert\mathcal H}=\sum_{\pi\in S_{n-1}}x^{\des(\pi)+1}\sum_{\mathcal H\in\VHC(\pi)}(C_{\bullet-1})_{\vert\mathcal H}.\] We can now use the Fertility Formula \eqref{Eq22}, along with Remark~\ref{Rem2}, to see that \[\sum_{\mathcal H\in\VHC(\pi)}(C_{\bullet-1})_{\vert\mathcal H}=\sum_{\mathcal H\in\VHC(\pi)}C_{\bf q^{\mathcal H}}=|s^{-1}(\pi)|\] for every $\pi\in S_{n-1}$. Consequently, 
\[
-c_n=\sum_{\pi\in S_{n-1}}x^{\des(\pi)+1}|s^{-1}(\pi)|=\sum_{\sigma\in S_{n-1}}x^{\des(s(\sigma))+1}.
\] 
According to \eqref{Eq4}, we have \[\sum_{n\geq 1}\left(-\sum_{\sigma\in S_{n-1}}x^{\des(s(\sigma))+1}\right)\frac{z^n}{n!}=\log(1+\widehat M(z)),\] where $\displaystyle M(z)=\sum_{n\geq 1}m_nz^n$ is the moment series. Thus, it suffices to show that $M(z)=F_x(z)$. 

The $R$-transform is given by \[R(z)=\sum_{n\geq 1}\kappa_nz^n=-x\sum_{n\geq 1}C_{n-1}z^n=-x\frac{1-\sqrt{1-4z}}{2}.\] It is straightforward to check that $R^{\langle -1\rangle}(z)=\dfrac{-xz-z^2}{x^2}$, so it follows from \eqref{Eq5} that $M^{\langle -1\rangle}(z)=\dfrac{-xz-z^2}{x^2(1+z)}$. From this, one can show that $M(z)=F_x(z)$, as desired. 
\end{proof}

Recall that the \dfn{variance} of a random variable $Y$ is $\Var(Y)=\mathbb E((Y-\mathbb E(Y))^2)=\mathbb E(Y^2)-\mathbb E(Y)^2$. The $m^\text{th}$ \dfn{moment} of a random variable $Y$ is defined to be $\mathbb E(Y^m)$. The identity in the previous theorem encodes all of the information about the random variables $D_n=\des(s(\sigma))+1$. Indeed, we will describe an algorithm for computing, for each fixed $m$, a generating function that encodes the $m^\text{th}$ moments of the variables $D_n$. We begin by illustrating how this algorithm allows us to compute the means and variances of these variables. Let us stress that there is no apparent way to use standard methods to show that the limit $\lim\limits_{n\to\infty}\dfrac{\mathbb E(D_n)}{n}$ even exists. This makes the incredible simplicity of the exact formula for $\mathbb E(D_n)$ in the next theorem all the more surprising. 

\begin{theorem}\label{Thm9}
For every $n\geq 1$, we have \[\mathbb E(D_n)=\left(3-\sum_{j=0}^n\frac{1}{j!}\right)n\sim(3-e)n.\]
\end{theorem}

\begin{proof}
We use the notation $[f(x)]_{x=1}$ to denote the evaluation of $f(x)$ at $x=1$. Notice that \[\sum_{n\geq 1}\frac{\mathbb E(D_n)}{n}z^n=\sum_{n\geq 1}\sum_{\sigma\in S_{n-1}}(\des(s(\sigma))+1)\frac{z^n}{n!}=\left[\frac{\partial}{\partial x}\sum_{n\geq 1}\left(\sum_{\sigma\in S_{n-1}}x^{\des(s(\sigma))+1}\right)\frac{z^n}{n!}\right]_{x=1}.\] By Theorem~\ref{Thm8}, we have 
\begin{equation}\label{Eq28}
\sum_{n\geq 1}\frac{\mathbb E(D_n)}{n}z^n=\left[\frac{\partial}{\partial x}(-\log(1+\widehat F_x(z)))\right]_{x=1}=\left[-\frac{\frac{\partial}{\partial x}\widehat F_x(z)}{1+\widehat F_x(z)}\right]_{x=1}=-\frac{\left[\frac{\partial}{\partial x}\widehat F_x(z)\right]_{x=1}}{1+\widehat F_1(z)}.
\end{equation} Referring to \eqref{Eq27}, we see that $F_1(z)=-z$, so 
\begin{equation}\label{Eq31}
\widehat F_1(z)=\mathcal L^{-1}\{-(1/t)/t\}(z)=-z. 
\end{equation}

To compute $\left[\frac{\partial}{\partial x}\widehat F_x(z)\right]_{x=1}$, notice that there are two operations being performed to the series $F_x(z)$. One is the transformation from an ordinary generating function in $z$ to the corresponding exponential generating function, which is described in \eqref{Eq6} via an inverse Laplace transform. The other is the operation that differentiates with respect to $x$ and then sets $x=1$. These two operations commute, so we have \[\left[\frac{\partial}{\partial x}\widehat F_x(z)\right]_{x=1}=\left[\frac{\partial}{\partial x}\left(\mathcal L^{-1}\left\{F_x(1/t)/t\right\}(z)\right)\right]_{x=1}=\mathcal L^{-1}\left\{\frac{\left[\frac{\partial}{\partial x}F_x(1/t)\right]_{x=1}}{t}\right\}(z).\] One can now compute $\displaystyle\frac{\left[\frac{\partial}{\partial x}F_x(1/t)\right]_{x=1}}{t}=\frac{t-2}{t^2(1-t)}$ so that 
\begin{equation}\label{Eq32}
\left[\frac{\partial}{\partial x}\widehat F_x(z)\right]_{x=1}=\mathcal L^{-1}\left\{\frac{t-2}{t^2(1-t)}\right\}(z)=-1-2z+e^z.
\end{equation} Combining this with \eqref{Eq28} and \eqref{Eq31}, we obtain 
\begin{equation}\label{Eq33}
\sum_{n\geq 1}\frac{\mathbb E(D_n)}{n}z^n=\frac{1+2z-e^z}{1-z}.
\end{equation} The desired result is now immediate if we extract the coefficient of $z^n$ in $\dfrac{1+2z-e^z}{1-z}$. 
\end{proof}

\begin{theorem}\label{Thm10}
We have \[\sum_{n\geq 1}\frac{\mathbb E(D_n^2)}{n}z^n=\frac{2+7z-(3+5z-3z^2+z^3)e^z+e^{2z}}{(1-z)^2}.\] The variances of the random variables $D_n$ satisfy $\Var(D_n)\sim(2+2e-e^2)n$. 
\end{theorem}
\begin{proof}
The proof is similar to that of Theorem~\ref{Thm9}. First, notice that \[\sum_{n\geq 1}\frac{\mathbb E(D_n(D_n-1))}{n}z^n=\sum_{n\geq 1}\sum_{\sigma\in S_{n-1}}(\des(s(\sigma))+1)\des(s(\sigma))\frac{z^n}{n!}\] \[=\left[\frac{\partial^2}{\partial x^2}\sum_{n\geq 1}\left(\sum_{\sigma\in S_{n-1}}x^{\des(s(\sigma))+1}\right)\frac{z^n}{n!}\right]_{x=1}.\] By Theorem~\ref{Thm8}, we have 
\[\sum_{n\geq 1}\frac{\mathbb E(D_n(D_n-1))}{n}z^n=\left[\frac{\partial^2}{\partial x^2}(-\log(1+\widehat F_x(z)))\right]_{x=1}\] 
\begin{equation}\label{Eq29}
=\left[-\frac{\frac{\partial^2}{\partial x^2}\widehat F_x(z)}{1+\widehat F_x(z)}+\left(\frac{\frac{\partial}{\partial x}\widehat F_x(z)}{1+\widehat F_x(z)}\right)^2\right]_{x=1}=-\frac{\left[\frac{\partial^2}{\partial x^2}\widehat F_x(z)\right]_{x=1}}{1-z}+\left(\frac{-1-2z+e^z}{1-z}\right)^2,
\end{equation} where the last equality follows from the identities \eqref{Eq31} and \eqref{Eq32} that we derived during the proof of Theorem~\ref{Thm9}. The same argument used to derive \eqref{Eq32} allows us to compute \[\left[\frac{\partial^2}{\partial x^2}\widehat F_x(z)\right]_{x=1}=\mathcal L^{-1}\left\{\frac{\left[\frac{\partial^2}{\partial x^2}F_x(1/t)\right]_{x=1}}{t}\right\}(z)=\mathcal L^{-1}\left\{-2\frac{1-3t+t^2}{t^2(1-t)^3}\right\}(z)=-z(2-2e^z+ze^z).\] Substituting this into \eqref{Eq29} yields 
\[\sum_{n\geq 1}\frac{\mathbb E(D_n(D_n-1))}{n}z^n=\frac{z(2-2e^z+ze^z)}{1-z}+\left(\frac{-1-2z+e^z}{1-z}\right)^2.\]
Consequently, \[\sum_{n\geq 1}\frac{\mathbb E(D_n^2)}{n}z^n=\sum_{n\geq 1}\frac{\mathbb E(D_n(D_n-1))}{n}z^n+\sum_{n\geq 1}\frac{\mathbb E(D_n)}{n}z^n\] \[=\frac{z(2-2e^z+ze^z)}{1-z}+\left(\frac{-1-2z+e^z}{1-z}\right)^2+\frac{1+2z-e^z}{1-z}=\frac{2+7z-(3+5z-3z^2+z^3)e^z+e^{2z}}{(1-z)^2},\]
where we have used \eqref{Eq33}. Theorem~\ref{Thm9} tells us that $\dfrac{\mathbb E(D_n)^2}{n}=(3-e)^2n+O(1/n!)$, so \[\sum_{n\geq 1}\frac{\Var(D_n)}{n}z^n=\sum_{n\geq 1}\frac{\mathbb E(D_n^2)}{n}z^n-\sum_{n\geq 1}\frac{\mathbb E(D_n)^2}{n}z^n\] \[=\frac{2+7z-(3+5z-3z^2+z^3)e^z+e^{2z}}{(1-z)^2}-\sum_{n\geq 1}(3-e)^2nz^n-\sum_{n\geq 1}O(1/n!)z^n\] \[=\frac{2+7z-(3+5z-3z^2+z^3)e^z+e^{2z}-(3-e)^2z}{(1-z)^2}-\sum_{n\geq 1}O(1/n!)z^n.\] The only singularity of $\dfrac{2+7z-(3+5z-3z^2+z^3)e^z+e^{2z}-(3-e)^2z}{(1-z)^2}$ is a simple pole at $z=1$, so it follows from Lemma~\ref{Lem1} that \[\lim_{n\to\infty}\frac{\Var(D_n)}{n}=-\Res_{z=1}\frac{2+7z-(3+5z-3z^2+z^3)e^z+e^{2z}-(3-e)^2z}{(1-z)^2}=2+2e-e^2.\qedhere\]
\end{proof}

In theory, one can repeat the main steps used in the above proof of Theorem~\ref{Thm10} to compute the generating functions $\displaystyle\sum_{n\geq 1}\frac{\mathbb E(D_n^m)}{n}z^n$ for each fixed $m\geq 2$. Note that \[\sum_{n\geq 1}\frac{\mathbb E(D_n(D_n-1)\cdots(D_n-m+1))}{n}z^n=\left[\frac{\partial^m}{\partial x^m}\sum_{n\geq 1}\left(\sum_{\sigma\in S_{n-1}}x^{\des(s(\sigma))+1}\right)\frac{z^n}{n!}\right]_{x=1}\] \[=\left[\frac{\partial^m}{\partial x^m}(-\log(1+\widehat F_x(z)))\right]_{x=1}\] by Theorem~\ref{Thm8}. After expanding this last expression, we find that its computation requires us to know $\widehat F_1(z)$ and $\displaystyle\left[\frac{\partial^p}{\partial x^p}\widehat F_x(z)\right]_{x=1}$ for all $1\leq p\leq m$. We have seen that $\widehat F_1(z)=-z$, and we can compute 
$\displaystyle\left[\frac{\partial^p}{\partial x^p}\widehat F_x(z)\right]_{x=1}$ using the fact that \[\left[\frac{\partial^p}{\partial x^p}\widehat F_x(z)\right]_{x=1}=\mathcal L^{-1}\left\{\frac{\left[\frac{\partial^p}{\partial x^p}F_x(1/t)\right]_{x=1}}{t}\right\}(z).\] This shows how to compute $\displaystyle\sum_{n\geq 1}\frac{\mathbb E(D_n(D_n-1)\cdots(D_n-m+1))}{n}z^n$. Using linearity of expectation, we can express $\displaystyle\sum_{n\geq 1}\frac{\mathbb E(D_n^m)}{n}z^n$ as a linear combination of the generating functions \[\displaystyle\sum_{n\geq 1}\frac{\mathbb E(D_n(D_n-1)\cdots(D_n-m+1))}{n}z^n\quad\text{and}\quad\sum_{n\geq 1}\frac{\mathbb E(D_n^p)}{n}z^n\quad\text{for }1\leq p\leq m-1.\] If we assume inductively that we have already computed the latter generating functions, then we can compute $\displaystyle\sum_{n\geq 1}\frac{\mathbb E(D_n^m)}{n}z^n$. 

Recall that the $m^\text{th}$ \dfn{central moment} of a random variable $Y$ is $\mathbb E((Y-\mathbb E(Y))^m)$. Using the procedure just described, we have computed $\displaystyle\sum_{n\geq 1}\frac{\mathbb E(D_n^m)}{n}z^n$ for $2\leq m\leq 6$. Using Theorem~\ref{Thm9}, which gives a simple explicit formula for $\mathbb E(D_n)$ for every $n\geq 1$, we have derived the asymptotics for the $m^\text{th}$ central moments of $D_n$ for $m\leq 6$ (we omit the details of these computations). The results are 
\begin{itemize}
\item $\mathbb E((D_n-\mathbb E(D_n))^2)\sim (2+2e-e^2)n$;

\item $\mathbb E((D_n-\mathbb E(D_n))^3)\sim (6-\frac{61}{2}e+24e^2-5e^3)n$; 

\item $\mathbb E((D_n-\mathbb E(D_n))^4)\sim 3(2+2e-e^2)^2n^2$;

\item $\mathbb E((D_n-\mathbb E(D_n))^5)\sim 5(24 - 98 e - 38 e^2 + 137 e^3 - 68 e^4 + 10 e^5)n^2$; 

\item $\mathbb E((D_n-\mathbb E(D_n))^6)\sim 15(2+2e-e^2)^3n^3$. 
\end{itemize}
These results are highly suggestive of an asymptotic normal distribution. Indeed, it is well known that if $Y$ is a normally-distributed random variable with mean $\mu$ and variance $\sigma^2$, which we write as $Y\sim N(\mu,\sigma^2)$, then the central moments of $Y$ are given by $\mathbb E((Y-\mu)^m)=0$ for $m$ odd and $\mathbb E((Y-\mu)^m)=\sigma^m(m-1)!!$ for $m$ even. Thus, we have the following conjecture. 

\begin{conjecture}\label{Conj2}
The sequence $(Y_n)_{n\geq 1}$ of random variables defined by \[Y_n=\frac{D_n-(3-e)n}{\sqrt{n}}\] converges in distribution to a random variable $Y$ such that $Y\sim N(0,2+2e-e^2)$. 
\end{conjecture}

An alternative approach one might take to proving Conjecture~\ref{Conj2} is as follows. For each $i\in[n-2]$ and $\pi\in S_{n-1}$, let $\des_i(\pi)=1$ if $i$ is a descent of $\pi$, and let $\des_i(\pi)=0$ otherwise. Define a random variable $D_{n,i}=\des_i(s(\sigma))$, where $\sigma$ is chosen uniformly at random from $S_{n-1}$. Then $D_n=1+\displaystyle\sum_{i=1}^{n-2}D_{n,i}$. One might hope to understand the distribution of $D_n$ by first understanding the distributions of the variables $D_{n,i}$ and their dependencies.

Suppose we wish to use this approach to prove, without free probability theory, that $\lim\limits_{n\to\infty}\dfrac{\mathbb E(D_n)}{n}$ $=3-e$. It would suffice to show that $\mathbb E(D_{n,i})\to 3-e$ as $n\to\infty$ for all $1\leq i\leq (1-o(1))n$. 
This approach seems promising at first because, as we will prove below, $\lim\limits_{n\to\infty}\mathbb E(D_{n,1})=3-e$. This says that if $\sigma\in S_{n-1}$ is chosen uniformly at random, then the probability that $1$ is a descent of $s(\sigma)$ is asymptotically (as $n\to\infty$) the same as the probability that a randomly-chosen index $i\in[n-2]$ is a descent of $s(\sigma)$.\footnote{This is also reminiscent of Theorem 5.7 in \cite{DefantEngenMiller}, which implies that the expected value of the first entry of a random uniquely sorted permutation in $S_{2k+1}$ is $k+1$. Of course, the expected value of the entry in a randomly-chosen position in such a permutation is also $k+1$. Perhaps there is some deeper connection between the behavior of the first entry of a sorted permutation and the behavior of a random entry.} However, it appears that $\lim\limits_{n\to\infty}\mathbb E(D_{n,2})\neq 3-e$. Thus, for right now, the identity 
\begin{equation}\label{Eq34}
\lim\limits_{n\to\infty}\mathbb E(D_{n,1})=\lim\limits_{n\to\infty}\dfrac{\mathbb E(D_n)}{n}
\end{equation} appears to be a mysterious coincidence. It would be interesting to have an explanation for why \eqref{Eq34} should hold, besides the fact that we can compute the limits separately and see that they are equal.     

\begin{theorem}\label{Thm11}
With the notation from above, we have $\lim\limits_{n\to\infty}\mathbb E(D_{n,1})=3-e$. 
\end{theorem}

\begin{proof}
The \dfn{standardization} of a permutation $\pi$ is the permutation obtained by replacing the $i^\text{th}$-smallest entry in $\pi$ with $i$ for all $i$. For example, the standardization of $3856$ is $1423$. Consider the definition of the stack-sorting map from Section~\ref{Sec:Stack-Sorting} that makes use of the stack. For entries $a$ and $b$ of a permutation $\sigma$, we say \dfn{$b$ forces $a$ out of the stack in $\sigma$} if $b$ is the leftmost entry that is greater than $a$ and to the right of $a$ in $\sigma$. Let $s(\sigma)_i$ denote the $i^\text{th}$ entry in $s(\sigma)$. Let $V_m$ be the set of permutations $\sigma=\sigma_1\cdots \sigma_m\in S_m$ such that $1$ is a descent of $s(\sigma)$ and no entry forces $s(\sigma)_2$ out of the stack in $\sigma$ (i.e., there are no entries greater than $s(\sigma)_2$ to the right of $s(\sigma)_2$ in $\sigma$). Let $V_{m,j}=\{\sigma\in V_m:s(\sigma)_1=j\}$. Notice that $V_{m,1}$ is empty since $1$ cannot be a descent of $s(\sigma)$ if $s(\sigma)_1=1$. Upon inspection of the definition of $s$, we find that for $2\leq j\leq m$, the set $V_{m,j}$ consists of permutations of the form \[m(m-1)\cdots (k+1)jk(k-1)\cdots (j+1)(j-1)(j-2)\cdots 1\] for some $k\in\{j+1,\ldots,m\}$. This shows that $|V_{m,j}|=m-j$, so $|V_m|=\displaystyle\sum_{j=2}^m(m-j)=\dfrac{(m-1)(m-2)}{2}$. 

Now let $V_{m+1}'$ be the set of permutations $\sigma=\sigma_1\cdots\sigma_{m+1}\in S_{m+1}$ such that $1$ is a descent of $s(\sigma)$ and such that $\sigma_{m+1}$ forces $s(\sigma)_2$ out of the stack in $\sigma$. One can check that $\sigma\in V_{m+1}'$ if and only if $\sigma_{m+1}\in\{2,\ldots,m+1\}$ and the standardization of $\sigma_1\cdots \sigma_m$ is in $V_m$. This implies that $|V_{m+1}'|=m|V_m|=\dfrac{m(m-1)(m-2)}{2}$. 

Finally, let $W_{m+1}^n$ be the set of permutations $\sigma=\sigma_1\cdots\sigma_n\in S_n$ such that $1$ is a descent of $s(\sigma)$ and $\sigma_{m+1}$ forces $s(\sigma)_2$ out of the stack in $\sigma$. It is straightforward to see that $\sigma\in W_{m+1}^n$ if and only if the standardization of $\sigma_1\cdots\sigma_{m+1}$ is in $V_{m+1}'$. Thus, the probability that a randomly-chosen element of $S_n$ is in $W_{m+1}^n$ is $\dfrac{|V_{m+1}'|}{(m+1)!}$. Notice that $1$ is a descent of $s(\sigma)$ if and only if $\sigma\in\left(\bigcup_{m=1}^{n-1}W_{m+1}^n\right)\cup V_n$. The sets $W_2^n,\ldots,W_n^n,V_n$ are disjoint. Therefore, if $\sigma$ is chosen uniformly at random from $S_n$, then the probability that $1$ is a descent of $s(\sigma)$ is \[\sum_{m=1}^{n-1}\frac{|V_{m+1}'|}{(m+1)!}+\frac{|V_n|}{n!}=\sum_{m=1}^{n-1}\frac{m(m-1)(m-2)/2}{(m+1)!}+\frac{(n-1)(n-2)/2}{n!}.\] As $n\to\infty$, this approaches $\displaystyle\sum_{m=1}^\infty\dfrac{m(m-1)(m-2)/2}{(m+1)!}=3-e$. 
\end{proof}

We end this section with another conjecture concerning the polynomials $\displaystyle\sum_{\sigma\in S_{n-1}}x^{\des(s(\sigma))+1}$, whose exponential generating function is given in Theorem~\ref{Thm8}. A sequence $a_1,\ldots,a_m$ is called \dfn{unimodal} if there exists an index $j$ such that $a_1\leq \cdots\leq a_{j-1}\leq a_j\geq a_{j+1}\geq \cdots \geq a_m$. Let $a_k(n)=|\{\sigma\in S_{n-1}:\des(s(\sigma))+1=k\}|$ be the coefficient of $x^k$ in $\displaystyle\sum_{\sigma\in S_{n-1}}x^{\des(s(\sigma))+1}$. We conjecture that the sequences $a_1(n),\ldots,a_{n-1}(n)$ are unimodal. It is known \cite{Branden2} that a polynomial with nonnegative real coefficients that has only real roots must have unimodal coefficients. Thus, our unimodality conjecture would follow from the following much stronger conjecture. 

\begin{conjecture}\label{Conj3}
For every $n\geq 1$, the polynomial $\displaystyle\sum_{\sigma\in S_{n-1}}x^{\des(s(\sigma))+1}$ has only real roots. 
\end{conjecture}

We have checked Conjecture~\ref{Conj3} for all $n\leq 33$. It would not be computationally feasible to check this many cases of the conjecture without the help of Theorem~\ref{Thm8}. 

\subsection{Descents in Postorder Readings of Trees}\label{Subsec:DescentsOthers}
Because the in-order reading $\mathcal I$ is a bijection from the set $\mathsf{\overline{D}BPT}_{n-1}$ of standardized decreasing binary plane trees with $n-1$ vertices to the set $S_{n-1}$, we can use \eqref{Eq11} to give an equivalent description of the random variable $D_n$. Namely, $D_n=\des(\mathcal P(\mathcal T))+1$, where $\mathcal T$ is chosen uniformly at random from $\mathsf{\overline{D}BPT}_{n-1}$. We can derive analogues of our results concerning $D_n$ for other troupes as well. The same approach used in Section~\ref{Subsec:DescentsSorted} provides an algorithm for computing the moments of the random variable $\des(\mathcal P(\mathcal T))+1$, where $\mathcal T$ is chosen uniformly at random from $\mathsf{\overline{D}{\bf T}}_{n-1}$. To illustrate this, we will focus specifically on full binary plane trees, Motzkin trees, and Schr\"oder $2$-colored binary trees. Furthermore, we will content ourselves with discussing only the expected values of the associated random variables. In each case, we will see that $\mathbb E(\des(\mathcal P(\mathcal T))+1)\sim \gamma n$ for some explicit constant $\gamma$ that we will compute. It is not clear how one could use standard methods to show that the constant $\gamma$ even exists, let alone compute its exact value. 

\subsubsection{Descents in Postorder Readings of Full Binary Plane Trees}\label{Subsubsec:DescentsFull} 

Recall that $|\mathsf{FBPT}_n|=C_{(n-1)/2}$, where $C_{(n-1)/2}=0$ when $n$ is even. As mentioned in Example~\ref{ExamFull1}, the number of standardized decreasing full binary plane trees with $n$ vertices when $n$ is odd is $|\mathsf{\overline{D}FBPT}_n|=E_n$, where the Euler numbers $E_n$ are defined via their generating function $\displaystyle\sum_{n\geq 0}E_n\frac{z^n}{n!}=\sec(z)+\tan(z)$. 

Let 
\begin{equation}\label{Eq39}
F^{\mathsf{FBPT}}_x(z)=-x\frac{1+2xz^2-\sqrt{1-4(1-x)z^2}}{2(1+x^2z^2)}.
\end{equation} 
We choose the branch of the square root that evaluates to $1$ when $z\to 0$. 
We view $F^{\mathsf{FBPT}}_x(z)$ as a power series in the variable $z$ with coefficients in $\mathbb C(x)$ so that $\widehat F^{\:\mathsf{FBPT}}_x(z)=\mathcal L^{-1}\{F^{\:\mathsf{FBPT}}_x(1/t)/t\}(z)$, where the inverse Laplace transform is taken with respect to the variable $t$. 

\begin{theorem}\label{Thm12}
We have \[\sum_{n\geq 1}\left(\sum_{\mathcal T\in \mathsf{\overline{D}FBPT}_{n-1}}x^{\des(\mathcal P(\mathcal T))+1}\right)\frac{z^n}{n!}=-\log(1+\widehat F^{\:\mathsf{FBPT}}_x(z)).\] 
\end{theorem}

\begin{proof}
Define a sequence $(\kappa_n)_{n\geq 1}$ of free cumulants by $\kappa_n=-xC_{(n-2)/2}$. Let $(m_n)_{n\geq 1}$ and $(c_n)_{n\geq 1}$ be the corresponding sequences of moments and classical cumulants, respectively. According to the VHC Cumulant Formula (Corollary~\ref{Cor4}), we have \[-c_n=\sum_{\mathcal H\in\VHC(S_{n-1})}(-\kappa_\bullet)_{\vert\mathcal H}=\sum_{\mathcal H\in\VHC(S_{n-1})}x^{\#\vert\mathcal H}(C_{(\bullet-2)/2})_{\vert\mathcal H}\] \[=\sum_{\pi\in S_{n-1}}x^{\des(\pi)+1}\sum_{\mathcal H\in\VHC(\pi)}(C_{(\bullet-2)/2})_{\vert\mathcal H}.\] Combining \eqref{Eq25} with Remark~\ref{Rem2}, we find that \[\sum_{\mathcal H\in\VHC(\pi)}(C_{(\bullet-2)/2})_{\vert\mathcal H}=\sum_{\mathcal H\in\VHC(\pi)}C_{({\bf q}^{\mathcal H}-1)/2}=|\mathcal P^{-1}(\pi)\cap\mathsf{\overline{D}FBPT}|\] for all $\pi\in S_{n-1}$. Therefore, 
\[
-c_n=\sum_{\pi\in S_{n-1}}x^{\des(\pi)+1}|\mathcal P^{-1}(\pi)\cap\mathsf{\overline{D}FBPT}|=\sum_{\mathcal T\in\mathsf{\overline{D}FBPT}_{n-1}}x^{\des(\mathcal P(\mathcal T))+1}.\] 
The equation \eqref{Eq4} tells us that \[\sum_{n\geq 1}\left(-\sum_{\mathcal T\in\mathsf{\overline{D}FBPT}_{n-1}}x^{\des(\mathcal P(\mathcal T))+1}\right)\frac{z^n}{n!}=\log(1+\widehat M(z)),\] where $\displaystyle M(z)=\sum_{n\geq 1}m_nz^n$ is the moment series. It now suffices to show that $M(z)=F^{\mathsf{FBPT}}_x(z)$. 

The $R$-transform is given by \[R(z)=\sum_{n\geq 1}\kappa_nz^n=-x\sum_{n\geq 1}C_{(n-2)/2}z^n=-x\frac{1-\sqrt{1-4z^2}}{2}.\] We now easily compute $R^{\langle -1\rangle}(z)=\pm\dfrac{\sqrt{-xz-z^2}}{x}$, so it follows from \eqref{Eq5} that \[M^{\langle -1\rangle}(z)=\pm\dfrac{\sqrt{-xz-z^2}}{x(1+z)}.\] From this, we find that \[M(z)=-x\frac{1+2xz^2\pm\sqrt{1-4(1-x)z^2}}{2(1+x^2z^2)}.\]
We must choose the minus sign in the $\pm$ because $M(0)=0$. Thus, $M(z)=F^{\mathsf{FBPT}}_x(z)$.  
\end{proof}

\begin{theorem}\label{Thm13}
If $n\geq 2$ is even and $\mathcal T$ is chosen uniformly at random from $\mathsf{\overline{D}FBPT}_{n-1}$, then \[\mathbb E(\des(\mathcal P(\mathcal T))+1)=\left(1-\frac{E_n}{nE_{n-1}}\right)n\sim\left(1-\frac{2}{\pi}\right)n.\]  
\end{theorem}

\begin{proof}
The proof is similar to that of Theorem~\ref{Thm9}. By Theorem~\ref{Thm12}, we have \[\sum_{n\geq 1}\sum_{\mathcal T\in\mathsf{\overline{D}FBPT}_{n-1}}(\des(\mathcal P(\mathcal T))+1)\frac{z^n}{n!}=\left[\frac{\partial}{\partial x}\sum_{n\geq 1}\left(\sum_{\mathcal T\in \mathsf{\overline{D}FBPT}_{n-1}}x^{\des(\mathcal P(\mathcal T))+1}\right)\frac{z^n}{n!}\right]_{x=1}\] 
\begin{equation}\label{Eq42}
=\left[\frac{\partial}{\partial x}(-\log(1+\widehat F^{\:\mathsf{FBPT}}_x(z)))\right]_{x=1}=\left[-\frac{\frac{\partial}{\partial x}\widehat F^{\:\mathsf{FBPT}}_x(z)}{1+\widehat F^{\:\mathsf{FBPT}}_x(z)}\right]_{x=1}=-\frac{\left[\frac{\partial}{\partial x}\widehat F^{\:\mathsf{FBPT}}_x(z)\right]_{x=1}}{1+\widehat F^{\:\mathsf{FBPT}}_1(z)}.
\end{equation} By \eqref{Eq39}, we have $F^{\mathsf{FBPT}}_1(z)=-\dfrac{z^2}{1+z^2}$. Hence, 
\begin{equation}\label{Eq40}
\widehat F^{\:\mathsf{FBPT}}_1(z)=\mathcal L^{-1}\left\{-\frac{(1/t)^2}{1+(1/t)^2}\frac{1}{t}\right\}(z)=\cos(z)-1.
\end{equation}
Also, \[\left[\frac{\partial}{\partial x}\widehat F^{\:\mathsf{FBPT}}_x(z)\right]_{x=1}=\left[\frac{\partial}{\partial x}\left(\mathcal L^{-1}\left\{F^{\:\mathsf{FBPT}}_x(1/t)/t\right\}(z)\right)\right]_{x=1}=\mathcal L^{-1}\left\{\frac{\left[\frac{\partial}{\partial x}F^{\:\mathsf{FBPT}}_x(1/t)\right]_{x=1}}{t}\right\}(z).\] One can now compute $\displaystyle\frac{\left[\frac{\partial}{\partial x}F^{\mathsf{FBPT}}_x(1/t)\right]_{x=1}}{t}=\frac{1-t^2}{t(1+t^2)^2}$ so that 
\begin{equation}\label{Eq41}
\left[\frac{\partial}{\partial x}\widehat F^{\:\mathsf{FBPT}}_x(z)\right]_{x=1}=\mathcal L^{-1}\left\{\frac{1-t^2}{t(1+t^2)^2}\right\}(z)=1-\cos(z)-z\sin(z).
\end{equation} Let us combine this with \eqref{Eq42} and \eqref{Eq40} to see that 
\begin{equation}\label{Eq43}
\sum_{n\geq 1}\sum_{\mathcal T\in\mathsf{\overline{D}FBPT}_{n-1}}(\des(\mathcal P(\mathcal T))+1)\frac{z^n}{n!}=-\frac{1-\cos(z)-z\sin(z)}{\cos(z)}=1-\sec(z)+z\tan(z).
\end{equation} 
Now, $\displaystyle\sec(z)=\sum_{\substack{n\geq 0\\ n\text{ even}}}E_n\dfrac{z^n}{n!}$, and $\displaystyle\tan(z)=\sum_{\substack{n\geq 0\\ n\text{ odd}}}E_n\dfrac{z^n}{n!}$. It follows that if $n\geq 2$ is even, then \[\frac{1}{|\mathsf{\overline{D}FBPT}_{n-1}|}\sum_{\mathcal T\in\mathsf{\overline{D}FBPT}_{n-1}}(\des(\mathcal P(\mathcal T))+1)=\frac{n!}{E_{n-1}}[z^n](1-\sec(z)+z\tan(z))\] \[=\frac{n!}{E_{n-1}}\left(-\frac{E_n}{n!}+\frac{E_{n-1}}{(n-1)!}\right)=\left(1-\frac{E_n}{nE_{n-1}}\right)n,\] as desired. It is known that $E_n\sim 2(2/\pi)^{n+1}n!$, so $\displaystyle\left(1-\frac{E_n}{nE_{n-1}}\right)n\sim\left(1-\frac{2}{\pi}\right)n$.   
\end{proof}

Because $\mathcal I:\mathsf{\overline{D}FBPT}_{n-1}\to S_{n-1}\cap\ALT$ is a bijection when $n$ is even, we can translate Theorem~\ref{Thm13} into the language of stack-sorting. 

\begin{corollary}\label{Cor5}
Suppose $n\geq 2$ is even. If $\sigma$ is chosen uniformly at random from the set of alternating permutations in $S_{n-1}$, then \[\mathbb E(\des(s(\sigma))+1)=\left(1-\frac{E_n}{nE_{n-1}}\right)n\sim\left(1-\frac{2}{\pi}\right)n.\] 
\end{corollary}

\begin{remark}\label{Rem4}
One might ask if an analogue of \eqref{Eq34} holds for the random variables associated to full binary plane trees (or alternatively, alternating permutations) considered in this section. It turns out that this is not the case. One can show that if $\sigma$ is chosen uniformly at random from the set of alternating permutations in $S_{2k-1}$, then the probability that $1$ is a descent of $s(\sigma)$ approaches $\dfrac{\pi}{2}-1$ as $k\to\infty$.  
\end{remark}

We also have the following analogue of Conjecture~\ref{Conj3}. Using Theorem~\ref{Thm12} and Mathematica, we have checked this conjecture for all $n\leq 90$. 

\begin{conjecture}\label{Conj4}
For every even $n\geq 2$, the polynomial $\displaystyle\sum_{\mathcal T\in\mathsf{\overline{D}FBPT}_{n-1}}x^{\des(\mathcal P(\mathcal T))+1}$ has only real roots. 
\end{conjecture} 

\subsubsection{Descents in Postorder Readings of Motzkin Trees} 

Let 
\begin{equation}\label{Eq44}
F^{\mathsf{Mot}}_x(z)=-x\frac{1-z+2xz^2-\sqrt{1-2z-3z^2+4xz^2}}{2(1-xz+x^2z^2)}.
\end{equation} 
We choose the branch of the square root that evaluates to $1$ when $z\to 0$. 
We view $F^{\mathsf{Mot}}_x(z)$ as a power series in the variable $z$ with coefficients in $\mathbb C(x)$ so that $\widehat F^{\:\mathsf{Mot}}_x(z)=\mathcal L^{-1}\{F^{\:\mathsf{Mot}}_x(1/t)/t\}(z)$, where the inverse Laplace transform is taken with respect to the variable $t$. 

\begin{theorem}\label{Thm14}
We have \[\sum_{n\geq 1}\left(\sum_{\mathcal T\in \mathsf{\overline{D}Mot}_{n-1}}x^{\des(\mathcal P(\mathcal T))+1}\right)\frac{z^n}{n!}=-\log(1+\widehat F^{\:\mathsf{Mot}}_x(z)).\] 
\end{theorem}

\begin{proof}
Define a sequence $(\kappa_n)_{n\geq 1}$ of free cumulants by $\kappa_n=-xM_{n-2}$, where $M_{n-2}$ is the $(n-2)^\text{th}$ Motzkin number. Let $(m_n)_{n\geq 1}$ and $(c_n)_{n\geq 1}$ be the corresponding sequences of moments and classical cumulants, respectively. We now repeat the same argument as in the proof of Theorem~\ref{Thm12}, except we invoke the equation \eqref{Eq26} instead of \eqref{Eq25}. This yields the identity \[\sum_{n\geq 1}\left(-\sum_{\mathcal T\in\mathsf{\overline{D}Mot}_{n-1}}x^{\des(\mathcal P(\mathcal T))+1}\right)\frac{z^n}{n!}=\log(1+\widehat M(z)),\] where $\displaystyle M(z)=\sum_{n\geq 1}m_nz^n$ is the moment series. It now suffices to show that $M(z)=F^{\mathsf{Mot}}_x(z)$. 

The $R$-transform is given by \[R(z)=\sum_{n\geq 1}\kappa_nz^n=-x\sum_{n\geq 1}M_{n-2}z^n=-x\frac{1-z-\sqrt{1-2z-3z^2}}{2},\] so $R^{\langle -1\rangle}(z)=\dfrac{xz\pm\sqrt{-3x^2z^2-4x^3z}}{2x^2}$. It follows from \eqref{Eq5} that $\displaystyle M^{\langle -1\rangle}(z)=\dfrac{xz\pm\sqrt{-3x^2z^2-4x^3z}}{2x^2(1+z)}$. From this, we find that \[M(z)=-x\frac{1-z+2xz^2\pm\sqrt{1-2z-3z^2+4xz^2}}{2(1-xz+x^2z^2)}.\]
We must choose the minus sign in the $\pm$ because $M(0)=0$. Thus, $M(z)=F^{\mathsf{Mot}}_x(z)$.  
\end{proof}

\begin{theorem}\label{Thm15}
Suppose we choose $\mathcal T$ uniformly at random from $\mathsf{\overline{D}Mot}_{n-1}$. As $n\to\infty$, \[\mathbb E(\des(\mathcal P(\mathcal T))+1)\sim\left(1-\dfrac{3\sqrt 3}{2\pi}\left(e^{\frac{\pi}{3\sqrt 3}}-1\right)\right)n.\]  
\end{theorem}

\begin{proof}
Repeating the argument from the proof of Theorem~\ref{Thm13}, we find that  
\begin{equation}\label{Eq45}
\sum_{n\geq 1}\sum_{\mathcal T\in\mathsf{\overline{D}Mot}_{n-1}}(\des(\mathcal P(\mathcal T))+1)\frac{z^n}{n!}=-\frac{\left[\frac{\partial}{\partial x}\widehat F^{\:\mathsf{Mot}}_x(z)\right]_{x=1}}{1+\widehat F^{\:\mathsf{Mot}}_1(z)}.
\end{equation} It follows directly from \eqref{Eq44} that $F^{\mathsf{Mot}}_1(z)=-\dfrac{z^2}{1-z+z^2}$. Therefore, 
\begin{equation}\label{Eq46}
\widehat F^{\:\mathsf{Mot}}_1(z)=\mathcal L^{-1}\left\{-\frac{(1/t)^2}{1-(1/t)+(1/t)^2}\frac{1}{t}\right\}(z)=e^{z/2}\left(\cos\left(\frac{\sqrt 3}{2}z\right)-\frac{1}{\sqrt{3}}\sin\left(\frac{\sqrt 3}{2}z\right)\right)-1.
\end{equation}
Next, we compute \[\left[\frac{\partial}{\partial x}\widehat F^{\:\mathsf{Mot}}_x(z)\right]_{x=1}=\left[\frac{\partial}{\partial x}\left(\mathcal L^{-1}\left\{F^{\:\mathsf{Mot}}_x(1/t)/t\right\}(z)\right)\right]_{x=1}=\mathcal L^{-1}\left\{\frac{\left[\frac{\partial}{\partial x}F^{\:\mathsf{Mot}}_x(1/t)\right]_{x=1}}{t}\right\}(z)\] \begin{equation}\label{Eq47}
=\mathcal L^{-1}\left\{\frac{t^2-2t}{(1-t)(1-t+t^2)^2}\right\}(z)=e^{z/2}\left(e^{z/2}-\cos\left(\frac{\sqrt 3}{2}z\right)-\frac{1}{\sqrt 3}(1+2z)\sin\left(\frac{\sqrt 3}{2}z\right)\right).
\end{equation} Combining \eqref{Eq45}, \eqref{Eq46}, and \eqref{Eq47}, we obtain 
\begin{equation}\label{Eq48}
\sum_{n\geq 1}\sum_{\mathcal T\in\mathsf{\overline{D}Mot}_{n-1}}(\des(\mathcal P(\mathcal T))+1)\frac{z^n}{n!}=-\frac{e^{z/2}-\cos\left(\frac{\sqrt 3}{2}z\right)-\frac{1}{\sqrt 3}(1+2z)\sin\left(\frac{\sqrt 3}{2}z\right)}{\cos\left(\frac{\sqrt 3}{2}z\right)-\frac{1}{\sqrt{3}}\sin\left(\frac{\sqrt 3}{2}z\right)}.
\end{equation} 
This last expression, viewed as a function of the complex variable $z$, is meromorphic. Its singularity nearest to the origin is a simple pole at $\dfrac{2\pi}{3\sqrt{3}}$. By Lemma~\ref{Lem1}, we have \[\sum_{\mathcal T\in\mathsf{\overline{D}Mot}_{n-1}}(\des(\mathcal P(\mathcal T))+1)\frac{1}{n!}\sim \left(\frac{2\pi}{3\sqrt 3}\right)^{-n-1}\Res_{z=\frac{2\pi}{3\sqrt 3}}\frac{e^{z/2}-\cos\left(\frac{\sqrt 3}{2}z\right)-\frac{1}{\sqrt 3}(1+2z)\sin\left(\frac{\sqrt 3}{2}z\right)}{\cos\left(\frac{\sqrt 3}{2}z\right)-\frac{1}{\sqrt{3}}\sin\left(\frac{\sqrt 3}{2}z\right)}\] \[=\left(\frac{2\pi}{3\sqrt 3}\right)^{-n-1}\left(1-e^{\frac{\pi}{3\sqrt 3}}+\frac{2\pi}{3\sqrt 3}\right).\] Finally, it is known (see OEIS sequence A080635) that $|\mathsf{\overline{D}Mot}_{n-1}|\sim\left(\dfrac{2\pi}{3\sqrt 3}\right)^{-n}(n-1)!$, so \[\frac{1}{|\mathsf{\overline{D}Mot}_{n-1}|}\sum_{\mathcal T\in\mathsf{\overline{D}Mot}_{n-1}}(\des(\mathcal P(\mathcal T))+1)\sim\frac{n!\left(\frac{2\pi}{3\sqrt 3}\right)^{-n-1}\left(1-e^{\frac{\pi}{3\sqrt 3}}+\frac{2\pi}{3\sqrt 3}\right)}{\left(\frac{2\pi}{3\sqrt 3}\right)^{-n}(n-1)!}\] \[=\left(1-\dfrac{3\sqrt 3}{2\pi}\left(e^{\frac{\pi}{3\sqrt 3}}-1\right)\right)n. \qedhere\] 
\end{proof}

Recall from Example~\ref{ExamMotzkin1} that the in-order reading $\mathcal I$ gives a bijection from $\mathsf{\overline{D}Mot}_{n-1}$ to the set $S_{n-1}\cap\EDP$ of permutations in $S_{n-1}$ in which every descent is a peak. This allows us to translate Theorem~\ref{Thm15} into the following theorem about stack-sorting. 

\begin{corollary}\label{Cor6}
Suppose we choose $\sigma$ uniformly at random from the set of permutations in $S_{n-1}$ whose descents are all peaks. As $n\to\infty$, \[\mathbb E(\des(s(\sigma))+1)\sim\left(1-\dfrac{3\sqrt 3}{2\pi}\left(e^{\frac{\pi}{3\sqrt 3}}-1\right)\right)n.\]
\end{corollary}

\begin{remark}\label{Rem5}
Numerical evidence suggests that the natural analogue of \eqref{Eq34} for the random variables associated to Motzkin trees (or alternatively, permutations whose descents are all peaks) considered in this section does not hold.
\end{remark}

To end this section, we state the following analogue of Conjecture~\ref{Conj3}. Using Theorem~\ref{Thm14} and Mathematica, we have checked this conjecture for all $n\leq 31$. 

\begin{conjecture}\label{Conj5} 
For every $n\geq 1$, the polynomial $\displaystyle\sum_{\mathcal T\in\mathsf{\overline{D}Mot}_{n-1}}x^{\des(\mathcal P(\mathcal T))+1}$ has only real roots. 
\end{conjecture} 

\subsubsection{Descents in Postorder Readings of Schr\"oder $2$-Colored Binary Trees} 

Let 
\begin{equation}\label{Eq49}
F^{\mathsf{Sch}}_x(z)=-x\frac{1-z+xz-\sqrt{1-6z+z^2+2xz+2xz^2+x^2z^2}}{2(1-xz)},
\end{equation} 
where we choose the branch of the square root that evaluates to $1$ when $z\to 0$. We view $F^{\mathsf{Sch}}_x(z)$ as a power series in the variable $z$ with coefficients in $\mathbb C(x)$.

\begin{theorem}\label{Thm20}
We have \[\sum_{n\geq 1}\left(\sum_{\mathcal T\in \mathsf{\overline{D}Sch}_{n-1}}x^{\des(\mathcal P(\mathcal T))+1}\right)\frac{z^n}{n!}=-\log(1+\widehat F^{\:\mathsf{Sch}}_x(z)).\] 
\end{theorem}

\begin{proof}
Define a sequence $(\kappa_n)_{n\geq 1}$ of free cumulants by $\kappa_n=-x\mathscr S_{n-1}$, where $\mathscr S_{n-1}$ is the $(n-1)^\text{th}$ large Schr\"oder number. Let $(m_n)_{n\geq 1}$ and $(c_n)_{n\geq 1}$ be the corresponding sequences of moments and classical cumulants, respectively. Repeating the same argument as in the proof of Theorem~\ref{Thm12}, except invoking \eqref{Eq50} instead of \eqref{Eq25}, we obtain \[\sum_{n\geq 1}\left(-\sum_{\mathcal T\in\mathsf{\overline{D}Sch}_{n-1}}x^{\des(\mathcal P(\mathcal T))+1}\right)\frac{z^n}{n!}=\log(1+\widehat M(z)),\] where $\displaystyle M(z)=\sum_{n\geq 1}m_nz^n$ is the moment series. It now suffices to show that $M(z)=F^{\mathsf{Sch}}_x(z)$. 

The $R$-transform is given by \[R(z)=\sum_{n\geq 1}\kappa_nz^n=-x\sum_{n\geq 1}\mathscr S_{n-1}z^n=-x\frac{1-z-\sqrt{1-6z+z^2}}{2},\] so $R^{\langle -1\rangle}(z)=\dfrac{z(z+x)}{x(z-x)}$. It follows from \eqref{Eq5} that $\displaystyle M^{\langle -1\rangle}(z)=\dfrac{z(z+x)}{x(z-x)(1+z)}$. From this, we find that \[M(z)=-x\frac{1-z+xz\pm\sqrt{1-6z+z^2+2xz+2xz^2+x^2z^2}}{2(1-xz)}.\]
We must choose the minus sign in the $\pm$ because $M(0)=0$. Thus, $M(z)=F^{\mathsf{Sch}}_x(z)$.  
\end{proof}

\begin{theorem}\label{Thm23}
Suppose we choose $\mathcal T$ uniformly at random from $\mathsf{\overline{D}Sch}_{n-1}$. As $n\to\infty$, \[\mathbb E(\des(\mathcal P(\mathcal T))+1)\sim\left(1-\frac{1}{2\log 2}\right)n.\]  
\end{theorem}

\begin{proof}
Repeating the argument from the proof of Theorem~\ref{Thm13}, we find that  
\begin{equation}\label{Eq51}
\sum_{n\geq 1}\sum_{\mathcal T\in\mathsf{\overline{D}Sch}_{n-1}}(\des(\mathcal P(\mathcal T))+1)\frac{z^n}{n!}=-\frac{\left[\frac{\partial}{\partial x}\widehat F^{\:\mathsf{Sch}}_x(z)\right]_{x=1}}{1+\widehat F^{\:\mathsf{Sch}}_1(z)}.
\end{equation} We see from \eqref{Eq49} that $F^{\mathsf{Sch}}_1(z)=-\dfrac{z}{1-z}$, so
\begin{equation}\label{Eq52}
\widehat F^{\:\mathsf{Sch}}_1(z)=\mathcal L^{-1}\left\{-\frac{1/t}{1-(1/t)}\frac{1}{t}\right\}(z)=1-e^z.
\end{equation}
Next, we compute \[\left[\frac{\partial}{\partial x}\widehat F^{\:\mathsf{Sch}}_x(z)\right]_{x=1}=\left[\frac{\partial}{\partial x}\left(\mathcal L^{-1}\left\{F^{\:\mathsf{Sch}}_x(1/t)/t\right\}(z)\right)\right]_{x=1}=\mathcal L^{-1}\left\{\frac{\left[\frac{\partial}{\partial x}F^{\:\mathsf{Sch}}_x(1/t)\right]_{x=1}}{t}\right\}(z)\] \begin{equation}\label{Eq54}
=\mathcal L^{-1}\left\{\frac{2-4t+t^2}{(2-t)(1-t)^2t}\right\}(z)=1-(2+z)e^z+e^{2z}.
\end{equation} Combining \eqref{Eq51}, \eqref{Eq52}, and \eqref{Eq54}, we obtain 
\begin{equation}\label{Eq55}
\sum_{n\geq 1}\sum_{\mathcal T\in\mathsf{\overline{D}Sch}_{n-1}}(\des(\mathcal P(\mathcal T))+1)\frac{z^n}{n!}=-\frac{1-(2+z)e^z+e^{2z}}{2-e^z}.
\end{equation} 
This last expression is meromorphic, and its singularity nearest to the origin is a simple pole at $\log 2$. Invoking Lemma~\ref{Lem1}, we find that \[\sum_{\mathcal T\in\mathsf{\overline{D}Sch}_{n-1}}(\des(\mathcal P(\mathcal T))+1)\frac{1}{n!}\sim \left(\log 2\right)^{-n-1}\Res_{z=\log 2}\frac{1-(2+z)e^z+e^{2z}}{2-e^z}=\left(\log 2\right)^{-n-1}\left(\log 2-\frac{1}{2}\right).\] Similarly, one can use Corollary~\ref{Cor12} and Lemma~\ref{Lem1} to show that $|\mathsf{\overline{D}Sch}_{n-1}|\sim(\log 2)^{-n}(n-1)!$. Thus, \[\frac{1}{|\mathsf{\overline{D}Sch}_{n-1}|}\sum_{\mathcal T\in\mathsf{\overline{D}Sch}_{n-1}}(\des(\mathcal P(\mathcal T))+1)\sim\frac{n!\left(\log 2\right)^{-n-1}\left(\log 2-1/2\right)}{(\log 2)^{-n}(n-1)!}=\left(1-\frac{1}{2\log 2}\right)n. \qedhere\] 
\end{proof}

\begin{remark}\label{Rem6}
Numerical evidence suggests that the natural analogue of \eqref{Eq34} for Schr\"oder $2$-colored binary trees fails to hold, but only by a little. Suppose we choose $\mathcal T\in\mathsf{\overline{D}Sch_{n-1}}$ uniformly at random. It appears that as $n\to\infty$, the probability that $1$ is a descent of $\mathcal P(\mathcal T)$ approaches a constant that is approximately $0.27259$, which is just slightly less than $1-\dfrac{1}{2\log 2}\approx 0.27865$. 
\end{remark}

Using Theorem~\ref{Thm20} and Mathematica, we have checked the following analogue of Conjecture~\ref{Conj3} for all $n\leq 31$. 

\begin{conjecture}\label{Conj6} 
For every $n\geq 1$, the polynomial $\displaystyle\sum_{\mathcal T\in\mathsf{\overline{D}Sch}_{n-1}}x^{\des(\mathcal P(\mathcal T))+1}$ has only real roots. 
\end{conjecture} 

\subsection{Asymptotics for Sorted Permutations} 

In \cite{Bousquet}, Bousquet-M\'elou investigated \emph{sorted} permutations, which are permutations in the image of the stack-sorting map $s$. She considered the numbers $e_{m,n}$, which count sorted permutations of length $m+n$ according to an additional statistic known as the Zeilberger statistic. For our purposes, it will be sufficient to know that these numbers satisfy the initial conditions $e_{m,-1}=0$ and $e_{0,n}=1$ and that $e_{m,0}=|s(S_m)|$ is the number of sorted permutations in $S_m$. The recurrence 
\begin{equation}\label{Eq35}
e_{m,n}=e_{m-1,n+1}+\sum_{i=1}^{m-1}\sum_{j=0}^{n-1}\binom{m-1}{i}e_{m-i-1,n-j}(e_{i,j}-e_{i,j-1}).
\end{equation} appears in \cite{Bousquet}. Unfortunately, this recurrence does not tell us anything immediately about the asymptotics of these numbers. The purpose of this section is to prove the following new estimates. 

\begin{theorem}
The limit $\displaystyle\lim_{n\to\infty}\left(\frac{|s(S_n)|}{n!}\right)^{1/n}$ exists and satisfies \[0.68631<\lim_{n\to\infty}\left(\frac{|s(S_n)|}{n!}\right)^{1/n}<0.75260.\] 
\end{theorem}

\begin{proof}
Suppose $\pi\in s(S_{m-1})$ and $\pi'\in s(S_{n-1})$. We can write $\pi=s(\sigma)$ and $\pi'=s(\sigma')$ for some $\sigma\in S_{m-1}$ and $\sigma'\in S_{n-1}$. Let $A$ be an $(m-1)$-element subset of $[m+n-2]$. Let $\widetilde\pi$ and $\widetilde\sigma$ be the permutations of $A$ whose standardizations (as defined in the proof of Theorem~\ref{Thm11}) are $\pi$ and $\sigma$, respectively. Let $\widetilde\pi'$ and $\widetilde\sigma'$ be the permutations of $[m+n-2]\setminus A$ whose standardizations are $\pi'$ and $\sigma'$, respectively. We have $s(\widetilde\sigma)=\widetilde\pi$ and $s(\widetilde\sigma')=\widetilde\pi'$, so it follows from the recursive description of $s$ given in \eqref{Eq36} that \[s(\widetilde\sigma(m+n-1)\widetilde\sigma')=s(\widetilde\sigma)s(\widetilde\sigma')(m+n-1)=\widetilde\pi\widetilde\pi'(m+n-1).\] This shows that $\widetilde\pi\widetilde\pi'(m+n-1)\in s(S_{m+n-1})$. The map sending the tuple $(\pi,\pi',A)$ to the permutation $\widetilde\pi\widetilde\pi'(m+n-1)$ is injective, so \[|s(S_{m-1})||s(S_{n-1})|\binom{m+n-2}{m-1}\leq |s(S_{m+n-1})|.\] Rearranging, this shows that 
\begin{equation}\label{Eq37}
\frac{|s(S_{m-1})|}{(m-1)!}\frac{|s(S_{n-1})|}{(n-1)!}\leq(m+n-1)\frac{|s(S_{m+n-1})|}{(m+n-1)!}.
\end{equation} 

We will make use of a generalization of Fekete's lemma due to de Bruijn and Erd\H{o}s \cite{deBruijn}, which states that if a sequence of nonnegative real numbers $(b_m)_{m\geq 1}$ satisfies $b_mb_n\leq b_{m+n}$ whenever $1/2\leq n/m\leq 2$, then $\lim\limits_{n\to\infty}b_n^{1/n}$ exists and equals $\sup\limits_{n\geq 1}b_n^{1/n}$. Now let $b_n=\dfrac{|s(S_{n-1})|}{n^2(n-1)!}$ for $n\geq 8$ and $b_n=0$ for $1\leq n\leq 7$. It is not difficult to check that $\dfrac{m+n-1}{m^2n^2}\leq\dfrac{1}{(m+n)^2}$ whenever $m,n\geq 8$ and $1/2\leq n/m\leq 2$. Therefore, it follows from \eqref{Eq37} that \[b_mb_n\leq\frac{m+n-1}{m^2n^2}\frac{|s(S_{m+n-1})|}{(m+n-1)!}\leq\frac{|s(S_{m+n-1})|}{(m+n)^2(m+n-1)!}=b_{m+n}\] whenever $m,n\geq 8$ and $1/2\leq n/m\leq 2$. The inequality $b_mb_n\leq b_{m+n}$ also certainly holds whenever $m$ or $n$ is at most $7$. According to the generalization of Fekete's lemma that we mentioned above, $\lim\limits_{n\to\infty}b_n^{1/n}$ exists and equals $\sup\limits_{n\geq 1}b_n^{1/n}$. It now follows from the definition of $b_n$ that $\displaystyle\lim_{n\to\infty}\left(\frac{|s(S_n)|}{n!}\right)^{1/n}$ exists and equals $\sup\limits_{n\geq 1}b_n^{1/n}$. We have used Bousquet-M\'elou's recurrence \eqref{Eq35} to compute $b_{802}^{1/802}$; its value is slightly larger than $0.68631$. This yields the desired lower bound for the limit. 

To prove the upper bound, note that it follows from the Fertility Formula \eqref{Eq22} that every sorted permutation in $S_n$ has a valid hook configuration. Therefore, $|s(S_n)|\leq|\VHC(S_n)|$. By Corollary~\ref{Cor14}, \[\frac{|s(S_n)|}{n!}\leq\frac{|\VHC(S_n)|}{n!}\sim\frac{1}{c^{n+1}},\] where $c\approx 1.32874$ is the constant appearing in that corollary. The desired upper bound is now immediate because $1/c<0.75260$.  
\end{proof} 

\subsection{The Degree of Noninvertibility of the Stack-Sorting Map}

Recently, Propp and the author \cite{DefantPropp} introduced the \dfn{degree of noninvertibility} of a function $f:X\to X$, where $X$ is a finite set, to be \[\deg(f:X\to X)=\frac{1}{|X|}\sum_{x\in X}|f^{-1}(x)|^2.\] This is a natural measure of how far the function $f$ is from being bijective. It is shown in \cite{DefantPropp} that $\deg(s:S_n\to S_n)$ grows exponentially in $n$, which is interesting because it contrasts the quadratic growth of $\deg({\bf B}:S_n\to S_n)$, where ${\bf B}$ is the bubble sort map. Indeed, Propp and the author showed that $\deg({\bf B}:S_n\to S_n)=\dfrac{(n+1)(n+2)}{6}$. While an exact formula for $\deg(s:S_n\to S_n)$ currently seems out of reach, we can at least obtain estimates for this quantity. It was shown in \cite{DefantPropp} that the limit $\lim\limits_{n\to\infty}\deg(s:S_n\to S_n)^{1/n}$ exists and satisfies 
\begin{equation}\label{Eq38}
1.12462\leq\lim_{n\to\infty}\deg(s:S_n\to S_n)^{1/n}\leq 4.
\end{equation} After running some experiments that used the Decomposition Lemma to compute the fertilities of large random permutations, Propp and the author conjectured that this limit lies in the interval $(1.68,1.73)$. In this section, we show how the tools involving free probability that we have developed allow us to greatly improve upon the lower bound in \eqref{Eq38}. This will bring the known lower bound close to the conjectured value of the limit. 

\begin{theorem}
We have $\displaystyle 1.62924<\lim_{n\to\infty}\deg(s:S_n\to S_n)^{1/n}$.
\end{theorem}   

\begin{proof}
For convenience, let $d_n=\deg(s:S_n\to S_n)$. Using the Fertility Formula \eqref{Eq22} and Remark~\ref{Rem2}, we find that \[(n-1)!d_{n-1}=\sum_{\pi\in S_{n-1}}|s^{-1}(\pi)|^2=\sum_{\pi\in S_{n-1}}\left(\sum_{\mathcal H\in\VHC(\pi)}(C_{\bullet-1})_{\vert\mathcal H}\right)^2\geq\sum_{\pi\in S_{n-1}}\sum_{\mathcal H\in\VHC(\pi)}((C_{\bullet-1})_{\vert\mathcal H})^2\] \[=\sum_{\pi\in S_{n-1}}\sum_{\mathcal H\in\VHC(\pi)}(C_{\bullet-1}^2)_{\vert\mathcal H}=\sum_{\mathcal H\in\VHC(S_{n-1})}(C_{\bullet-1}^2)_{\vert\mathcal H}.\] If we now define a sequence $(\kappa_n)_{n\geq 1}$ of free cumulants by $\kappa_n=-C_{n-1}^2$, then the VHC Cumulant Formula (Corollary~\ref{Cor4}) tells us that the corresponding classical cumulants are given by \[-c_n=\sum_{\mathcal H\in\VHC(S_{n-1})}(C_{\bullet-1}^2)_{\vert\mathcal H}.\] Hence, $(n-1)!d_{n-1}\geq -c_n$. In the proof of \eqref{Eq38} given in \cite{DefantPropp}, it is shown that \[\lim_{n\to\infty}d_n^{1/n}=\sup_{n\geq 1}a_n^{1/n},\] where $a_n=\dfrac{d_{n-1}}{n^2}$ for $n\geq 8$ and $a_n=0$ for $1\leq n\leq 7$. Thus, \[\lim_{n\to\infty}d_n^{1/n}\geq a_{1000}^{1/1000}=\left(\frac{d_{999}}{1000^2}\right)^{1/1000}\geq\left(\frac{-c_{1000}}{1000^2\cdot 999!}\right)^{1/1000}>1.62924.\] To obtain the last inequality, we computed the exact value of $c_{1000}$ using Mathematica. First, we computed the first $1000$ terms in the series expansion of the $R$-transform $\displaystyle R(z)=\sum_{n\geq 1}\kappa_nz^n=-\sum_{n\geq 1}C_{n-1}^2z^n$. Applying \eqref{Eq5}, then \eqref{Eq6}, and then \eqref{Eq4} allowed us to compute the first $1000$ terms of the series $\displaystyle\sum_{n\geq 1}c_n\dfrac{z^n}{n!}$. In particular, this yielded the value of $c_{1000}$.
\end{proof} 

\begin{remark}
The combination of the Fertility Formula and the VHC Cumulant Formula is useful in the proof of Theorem~\ref{Thm19} because it allows us to compute a lower bound for $\deg(s:S_{999}\to S_{999})$. Trying to compute $\deg(s:S_n\to S_n)$ by brute force, one will not be able to exceed small values of $n$ (around $n=14$). On the other hand, we lose something when we use the inequality \[\sum_{\pi\in S_{n-1}}\left(\sum_{\mathcal H\in\VHC(\pi)}(C_{\bullet-1})_{\vert\mathcal H}\right)^2\geq\sum_{\pi\in S_{n-1}}\sum_{\mathcal H\in\VHC(\pi)}((C_{\bullet-1})_{\vert\mathcal H})^2. \qedhere\]  
\end{remark}

\section{Other Cumulant Conversion Formulas}\label{Sec:OtherFormulas}

\subsection{A Sum over Noncrossing Partitions}
Recall how we derived the VHC Cumulant Formula (Corollary~\ref{Cor4}) in Section~\ref{Subsec:CumulantsVHCs}. We first showed how to obtain a connected set partition $\vert\mathcal H$ from a valid hook configurations $\mathcal H$. We then saw from Theorem~\ref{Thm18} that for each connected partition $\rho$, the number of valid hook configurations $\mathcal H$ such that $\vert\mathcal H=\rho$ is given by the evaluation $T_{G(\rho)}(1,0)$ of the Tutte polynomial of the crossing graph of $\rho$. The VHC Cumulant Formula then followed from Josuat-Verg\`es' formula (Theorem~\ref{Thm18}). Because Josuat-Verg\`es' formula extends to the multivariate setting, the VHC Cumulant Formula also extends to the multivariate formula \eqref{Eq7}.  

We also saw how to obtain a noncrossing set partition $\underline{\mathcal H}$ from a valid hook configuration $\mathcal H$. There is a natural size-preserving bijection from $\vert\mathcal H$ to $\underline{\mathcal H}$, which we denoted by $B\mapsto\wideparen B$. Therefore, we can rewrite the VHC Cumulant Formula as 
\begin{equation}\label{Eq56}
-c_n=\sum_{\mathcal H\in\VHC(S_{n-1})}(-\kappa_\bullet)_{\underline{\mathcal H}}.
\end{equation} By describing the number of valid hook configurations $\mathcal H$ such that $\underline{\mathcal H}$ is equal to a given noncrossing partition $\eta$, we will obtain an analogue of Josuat-Verg\`es' formula in which the sum ranges over noncrossing partitions. This new formula does \emph{not} extend to the multivariate setting. Indeed, it is impossible to find a formula that expresses multivariate classical cumulants in terms of multivariate free cumulants via a sum over noncrossing partitions (see, for example, Remark 8.1 in \cite{Arizmendi}).  

We are going to make use of a very important bijection $K:\NC(n)\to\NC(n)$, known as the \dfn{Kreweras complementation map}. Given $\eta\in\NC(n)$, we define $K'(\eta)$ to be the maximum (in the reverse refinement order) partition of the totally ordered set $\{1'<2'<\cdots<n'\}$ such that $\eta\,\cup K'(\eta)$ is a noncrossing partition of the totally ordered set $\{1<1'<2<2'<\cdots<n<n'\}$. We then define $K(\eta)$ to be the partition in $\NC(n)$ obtained by removing the primes from the elements of the blocks of $K'(\eta)$. For example, Figure~\ref{Fig17} shows that the Kreweras complement of the partition $\eta=\{\{1,4,5\},\{2,3\},\{6\},\{7,8\}\}$ is the partition $K(\eta)=\{\{1,3\},\{2\},\{4\},\{5,6,8\},\{7\}\}$. 

\begin{figure}[ht]
\begin{center} 
\includegraphics[height=1.1cm]{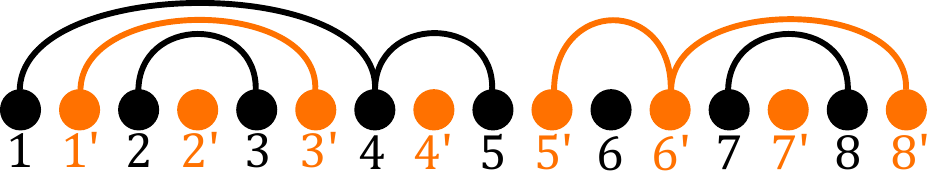}
\end{center}
\caption{The partition $\eta=\{\{1,4,5\},\{2,3\},\{6\},\{7,8\}\}$ (black) and its Kreweras complement $K(\eta)=\{\{1,3\},\{2\},\{4\},\{5,6,8\},\{7\}\}$ (orange).}\label{Fig17}
\end{figure}

Now consider a noncrossing partition $\eta\in\NC(n)$ and its Kreweras complement $K(\eta)$. 
We say $j$ is the \dfn{successor} of $i$ in $K(\eta)$ if $i$ and $j$ are in the same block $B$ of $K(\eta)$, $i<j$, and there are no elements $k\in B$ with $i<k<j$. Consider the elements of $[n]$ as the vertices of a directed graph. If $i\in[n-1]$ is the largest element of its block in $K(\eta)$, draw a directed edge from $i$ to $i+1$. If $i$ is not maximal in its block in $K(\eta)$, draw a directed edge from $i+1$ to $i$. If $j$ is the successor of $i$ in $K(\eta)$, then draw a directed edge from $i$ to $j$. We call the resulting directed graph the \dfn{arc graph} of $K(\eta)$ (see Figure~\ref{Fig20} for an example). Let $\mathcal L(K(\eta))$ denote the set of permutations $\sigma\in S_n$ such that $i$ appears to the left of $j$ in $\sigma$ whenever there is a directed edge from $i$ to $j$ in the arc graph of $K(\eta)$. The set $\mathcal L(K(\eta))$ is nonempty if and only if the arc graph of $K(\eta)$ is acyclic (i.e., has no directed cycles). One can check that this occurs if and only if there do not exist two consecutive integers in the same block of $K(\eta)$, which occurs if and only if $n=1$ or $\eta$ contains no singleton blocks. In this case, we can also view $\mathcal L(K(\eta))$ as the set of linear extensions of the poset $([n],\preceq)$ defined by declaring that $i\preceq j$ if and only if there is a directed path from $i$ to $j$ in the arc graph of $K(\eta)$. Let $\widetilde{\NC}(n)$ be the set of pairs $(\eta,\sigma)$ such that $\eta\in\NC(n)$ and $\sigma\in\mathcal L(K(\eta))$. 

\begin{figure}[ht]
\begin{center} 
\includegraphics[height=3.3cm]{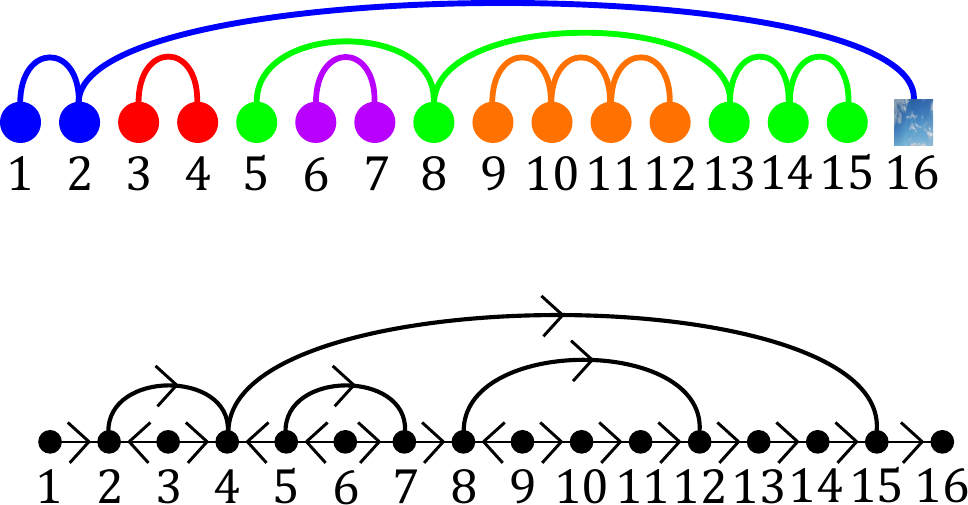}
\end{center}
\caption{On top is the noncrossing partition $\underline{\mathcal H}$, where $\mathcal H$ is the valid hook configuration whose modified diagram is shown in Figure~\ref{Fig10}. Its Kreweras complement is $K(\underline{\mathcal H})=\{\{1\},\{2,4,15\},\{3\},\{5,7\},\{6\},\{8,12\},\{9\},\{10\},\{11\},\{13\},\{14\},\{16\}\}$, whose arc graph is shown on the bottom. }\label{Fig20} 
\end{figure}

The reader may find it helpful to refer to Figures~\ref{Fig10} and \ref{Fig20} in the following proof. If $\pi=\pi_1\cdots\pi_{n-1}\in S_{n-1}$, then $\pi^{-1}$ denotes the permutation in $S_{n-1}$ whose $\pi_{i}^\text{th}$ entry is $i$ for all $i\in[n-1]$. 

\begin{theorem}\label{Thm54}
For $\pi\in S_{n-1}$ and $\mathcal H\in\VHC(\pi)$, let $\Psi(\mathcal H)=(\underline{\mathcal H},\pi^{-1}n)$, where $\pi^{-1} n\in S_n$ is the concatenation of $\pi^{-1}$ and $n$. The map $\Psi:\VHC(S_{n-1})\to\widetilde{\NC}(n)$ is a bijection. 
\end{theorem} 

\begin{proof}
If $n=1$, then the proof is trivial because $\pi$ is the empty permutation, $\mathcal H$ is the empty valid hook configuration, and $\underline{\mathcal H}=\{\{1\}\}$. Assume $n\geq 2$. We first prove that the image of $\Psi$ is contained in $\widetilde{\NC}(n)$. Let $\tau=\pi n=(\pi^{-1}n)^{-1}$. We need to show that $\tau^{-1}\in\mathcal L(K(\underline{\mathcal H}))$. This is equivalent to showing that $\tau_a<\tau_b$ whenever there is a directed edge from $a$ to $b$ in the arc graph of $K(\underline{\mathcal H})$. Choose $i\in[n-1]$, and consider the edge between $i$ and $i+1$ in the arc graph of $K(\underline{\mathcal H})$. It follows from the definition of the arc graph that this edge is directed from $i+1$ to $i$ if and only if $i$ is not maximal in its block in $K(\underline{\mathcal H})$. By inspecting the definition of the Kreweras complementation map, we see that this occurs if and only if $i+1$ is minimal in its block in $\underline{\mathcal H}$. This occurs if and only if $(i,\pi_i)$ is the southwest endpoint of a hook in $\mathcal H$, which happens if and only if $i$ is a descent of $\pi$. The descents of $\pi$ are the same as the descents of $\tau$, so the edge between $i$ and $i+1$ in the arc graph of $K(\underline{\mathcal H})$ is directed from $i+1$ to $i$ if and only if $\tau_i>\tau_{i+1}$. We also have a directed edge from $i$ to $j$ in the arc graph of $K(\underline{\mathcal H})$ if $j$ is the successor of $i$ in $K(\underline{\mathcal H})$. It is straightforward to check that this occurs if and only if there is a hook with southwest endpoint $(i,\pi_i)$ and northeast endpoint $(j,\pi_j)$ in $\mathcal H$. In this case, $\tau_i=\pi_i<\pi_j=\tau_j$, as desired. 

To see that $\Psi$ is bijective, we exhibit its inverse. Suppose we start with $(\eta,\sigma)\in\widetilde{\NC}(n)$. Consider $i\in[n-1]$. If $i$ is maximal in its block in $K(\eta)$, then there is an edge directed from $i$ to $i+1$ in the arc graph of $K(\eta)$. Otherwise, $i$ has a successor $j$ in $K(\eta)$, and there is a directed edge from $i$ to $j$ in the arc graph of $K(\eta)$. In either case, $i$ is not a sink (i.e., a vertex with outdegree $0$) in the arc graph of $K(\eta)$. As this is true for all $i\in[n-1]$, the number $n$ is forced to be the unique sink in this directed graph. It follows that $n$ appears last in the permutation $\sigma$, so we can write $\sigma=\pi^{-1}n$ for some $\pi\in S_{n-1}$. One can check that $d$ is a descent of $\pi$ if and only if $d$ is not maximal in its block in $K(\eta)$. Therefore, each descent $d$ of $\pi$ has a successor $b$ in $K(\eta)$. Let $d_1<\cdots<d_k$ be the descents of $\pi$. For all $1\leq\ell\leq k$, let $b_\ell$ be the successor of $d_\ell$ in $K(\eta)$; we have $\pi_{d_\ell}<\pi_{b_\ell}$ because $\pi^{-1}n\in\mathcal L(K(\eta))$. Therefore, we can draw a hook $H_\ell$ with southwest endpoint $(d_\ell,\pi_{d_\ell})$ and northeast endpoint $(b_\ell,\pi_{b_\ell})$ for all $1\leq\ell\leq k$; this produces a configuration $\mathcal H=(H_1,\ldots,H_k)$ of hooks of $\pi$. The fact that $K(\eta)$ is noncrossing implies that none of these hooks cross or intersect each other, except possibly when the southwest endpoint of one is the northeast endpoint of another. One can check by induction on $m$ that $H_{k-m}$ does not pass underneath any points in the plot of $\pi$. It follows that $\mathcal H$ is a valid hook configuration. The resulting map $(\eta,\sigma)\mapsto\mathcal H$ is the inverse of $\Psi$.   
\end{proof}

Combining \eqref{Eq56} with Theorem~\ref{Thm54} yields the following corollary. 

\begin{corollary}\label{Cor15}
If $(\kappa_n)_{n\geq 1}$ is a sequence of free cumulants, then the corresponding classical cumulants are given by \[-c_n=\sum_{\eta\in\NC(n)}|\mathcal L(K(\eta))|(-\kappa_\bullet)_{\eta}.\]
\end{corollary}

\begin{remark}\label{Rem3}
It might be interesting to see if there are alternative methods for computing (or possibly other combinatorial descriptions of) the numbers $|\mathcal L(K(\eta))|$, especially when the noncrossing partitions $\eta$ have special forms. For example, if $n$ is even and we put \[\eta=\{\{1,n\},\{2,3\},\{4,5\},\ldots,\{n-2,n-1\}\}\quad\text{and}\quad\eta'=\{\{1,n\},\{2,n-1\},\ldots,\{n/2,n/2+1\}\},\] then one can show that $|\mathcal L(K(\eta))|=(n-3)!!$ and $|\mathcal L(K(\eta'))|=C_{n/2-1}$. 
\end{remark} 

\subsection{A Sum over $231$-Avoiding Valid Hook Configurations} 

Recall the definition of a $231$-avoiding permutation from Example~\ref{ExamBinary2}. Let $\Av_{n-1}(231)$ denote the set of $231$-avoiding permutations in $S_{n-1}$. We say a valid hook configuration is \dfn{$231$-avoiding} if its underlying permutation is $231$-avoiding, and we write $\VHC(\Av_{n-1}(231))$ for the set of $231$-avoiding valid hook configurations in $\VHC(S_{n-1})$. In this section, we will rewrite the VHC Cumulant Formula as a sum over $231$-avoiding valid hook configurations. 

To begin, suppose $P$ is an $(n-1)$-element poset, and let $\mathcal L(P)$ denote the set of linear extensions of $P$, which we view as labelings of $P$ with the elements of $[n-1]$. Given $L\in\mathcal L(P)$, define $p_\ell(L)\in\mathcal L(P)$ as follows. If the element of $P$ with label $\ell$ in $L$ is comparable to (equivalently, covered by) the element with label $\ell+1$, then let $p_\ell(L)=L$. Otherwise, let $p_\ell(L)$ be the linear extension obtained from $L$ by swapping the labels $\ell$ and $\ell+1$. This defines an involution $p_\ell:\mathcal L(P)\to\mathcal L(P)$. Let $\mathfrak S_{\mathcal L(P)}$ denote the set of bijections from $\mathcal L(P)$ to itself. It is not difficult to show (by induction on $n$) that the subgroup of $\mathfrak S_{\mathcal L(P)}$ generated by $p_1,\ldots,p_{n-2}$ acts transitively on $\mathcal L(P)$.  

We now once again make use of the in-order bijection $\mathcal I:\mathsf{\overline{D}BPT}_{n-1}\to S_{n-1}$. Let $T\in\mathsf{BPT}$. The tree $T$ represents a poset in which a vertex $u$ is less than a vertex $v$ whenever $u$ is a descendant of $v$. From this point of view, standardized decreasing binary plane trees with skeleton $T$ correspond to linear extensions of $T$. We can use this correspondence to transfer the maps $p_\ell$ to the set $S_{n-1}$. Doing so, we obtain the following alternative description. 

Suppose $\pi\in S_{n-1}$. If there exists an entry $a>\ell+1$ that appears between $\ell$ and $\ell+1$ in $\pi$, let $p_\ell(\pi)$ be the permutation obtained from $\pi$ by swapping the entries $\ell$ and $\ell+1$. If no such entry $a$ exists, let $p_\ell(\pi)=\pi$. This defines an involution $p_\ell:S_{n-1}\to S_{n-1}$, which we can view as an element of the group $\mathfrak S_{S_{n-1}}$ of all bijections from $S_{n-1}$ to $S_{n-1}$. Let $\mathscr P_{n-1}=\langle p_1,\ldots,p_{n-2}\rangle$ be the subgroup of $\mathfrak S_{S_{n-1}}$ generated by $p_1,\ldots, p_{n-2}$. The result mentioned in the previous paragraph implies that for every permutation $\pi\in S_{n-1}$, the $\mathscr P_{n-1}$-orbit of $\pi$ is the set of permutations $\pi'$ such that $\skel(\mathcal I^{-1}(\pi'))=\skel(\mathcal I^{-1}(\pi))$.\footnote{In \cite{Hivert}, the $\mathscr P_{n-1}$-orbits are called \dfn{sylvester classes}. In \cite{Bjorner}, it is shown that $\mathscr P_{n-1}$-orbits form intervals in the weak order on $S_n$. In \cite{Postnikov}, it is shown that $\mathscr P_{n-1}$-orbits naturally label vertices of associahedra when associahedra are viewed as generalized permutohedra.}  

Notice that the set of descents of $\pi$ is the same as the set of descents of $p_\ell(\pi)$. Suppose $\mathcal H=(H_1,\ldots,H_k)\in\VHC(\pi)$. If $H_r$ has southwest endpoint $(i,\pi_i)$ and northeast endpoint $(j,\pi_j)$, then let $\widetilde H_r$ be the hook of $p_\ell(\pi)$ with southwest endpoint $(i,(p_\ell(\pi))_i)$ and northeast endpoint $(j,(p_\ell(\pi))_j)$. One can easily verify that $\widetilde{\mathcal H}:=(\widetilde H_1,\ldots,\widetilde H_k)$ is a valid hook configuration of $p_\ell(\pi)$ satisfying $\underline{\widetilde{\mathcal H}}=\underline{\mathcal H}$ and that the resulting map $\VHC(\pi)\to\VHC(p_\ell(\pi))$ given by $\mathcal H\mapsto\widetilde{\mathcal H}$ is a bijection. It follows that the set $\{\underline{\mathcal H}:\mathcal H\in\VHC(\pi)\}$ of noncrossing partitions associated to valid hook configurations of $\pi$ only depends on the $\mathscr P_{n-1}$-orbit of $\pi$. Therefore, we can rewrite \eqref{Eq56} as 
\begin{equation}\label{Eq58}
-c_n=\sum_{\pi}\mathscr T_\pi\sum_{\mathcal H\in\VHC(\pi)}(-\kappa_\bullet)_{\underline{\mathcal H}},
\end{equation} where the first sum ranges over a set of representatives for the $\mathscr P_{n-1}$-orbits in $S_{n-1}$ and $\mathscr T_\pi$ denotes the size of the $\mathscr P_{n-1}$-orbit containing $\pi$. 

For each valid hook configuration $\mathcal H$ of a permutation $\pi\in S_{n-1}$, let $\mathscr T_{\mathcal H}=\mathscr T_\pi$. Letting $T_\pi=\skel(\mathcal I^{-1}(\pi))$, we see from the above remarks that $\mathscr T_{\mathcal H}$ is the number of standardized decreasing binary plane trees with skeleton $T_\pi$. This description is useful because there is a well-known hook length formula for the number of linear extensions of a poset whose Hasse diagram is a rooted tree (originally due to Knuth in \cite{Knuth2}). For each of the $n-1$ vertices $v$ of $T_\pi$, let $h_v$ denote the size of the subtree of $T_v$ with root $v$ (including $v$ itself). Then \[\mathscr T_{\mathcal H}=\frac{(n-1)!}{\prod_vh_v}.\] 

For every binary plane tree $T$ with $n-1$ vertices, there is a unique decreasing binary plane tree $\ddot T$ with skeleton $T$ such that $\mathcal I(\ddot T)\in\Av_{n-1}(231)$. Indeed, $\ddot T$ is obtained by labeling the vertices of $T$ so that the postorder reading $\mathcal P(\ddot T)$ is $123\cdots(n-1)$. Consequently, $\Av_{n-1}(231)$ is a set of representatives for the $\mathscr P_{n-1}$-orbits in $S_{n-1}$. Referring back to \eqref{Eq58}, we obtain the following theorem.  

\begin{theorem}\label{Thm25}
If $(\kappa_n)_{n\geq 1}$ is a sequence of free cumulants, then the corresponding classical cumulants are given by \[-c_n=\sum_{\mathcal H\in\VHC(\Av_{n-1}(231))}\mathscr T_{\mathcal H}(-\kappa_\bullet)_{\underline{\mathcal H}}.\]
\end{theorem}

Let us remark that $231$-avoiding valid hook configurations were enumerated in \cite{DefantMotzkin}, where they were shown to be in bijection with $132$-avoiding valid hook configurations (this essentially follows from the above remarks because every $\mathscr P_{n-1}$-orbit contains a unique $231$-avoiding permutation and a unique $132$-avoiding permutation). In \cite{Maya}, Sankar gave an intricate bijection between $132$-avoiding valid hook configurations and intervals in Motzkin-Tamari posets. 

\section{$2$-Stack-Sortable and $3$-Stack-Sortable Permutations}\label{Sec:2-stack}

A permutation $\pi$ is called \dfn{$t$-stack-sortable} if $s^t(\pi)$ is increasing. Let $\mathcal W_t(n)$ denote the set of $t$-stack-sortable permutation in $S_n$. In \cite{Knuth}, Knuth proved that $\mathcal W_1(n)$ is the set $\Av_n(231)$ of $231$-avoiding permutations in $S_n$ and that $|\mathcal W_1(n)|=|\Av_n(231)|=C_n$. In his thesis, West \cite{West} conjectured that $|\mathcal W_2(n)|=\frac{2}{(n+1)(2n+1)}\binom{3n}{n}$. This was later proved by Zeilberger \cite{Zeilberger}. Since then, several papers devoted to the enumerative properties of $2$-stack-sortable permutations have emerged \cite{BonaSimplicial, Bousquet98, Branden3, Cori, DefantCounting, Dulucq, Dulucq2, Egge, Fang, Goulden}. The problem of finding a polynomial-time algorithm for enumerating $3$-stack-sortable permutations was open for 30 years, and was solved recently in \cite{DefantCounting}. 

If ${\bf T}\subseteq\mathsf{BPT}$, then the in-order map $\mathcal I$ gives a bijection between $\mathsf{\overline D}{\bf T}$ and a set $\mathcal I(\mathsf{\overline D}{\bf T})$ of standardized permutations associated to ${\bf T}$. For example, $\mathcal I(\mathsf{\overline{D}BPT})$ is the set of all standardized permutations, $\mathcal I(\mathsf{\overline{D}FBPT})$ is the set of standardized alternating permutations of odd length, and $\mathcal I(\mathsf{\overline{D}Mot})$ is the set of standardized permutations in which every descent is a peak. A consequence of the main theorem in this section will provide a way to enumerate $2$-stack-sortable permutations in $\mathcal I(\mathsf{\overline D}{\bf T})$ whenever ${\bf T}\subseteq\mathsf{BPT}$ is a troupe. We will also see that in many cases, the generating function that counts trees in $\mathcal P^{-1}(\Av(231))\cap\mathsf{\overline D}{\bf T}$ according to a collection of insertion-additive tree statistics is algebraic. There has been a great deal of work devoted to proving the algebraicity of various generating functions of combinatorial interest (see the survey \cite{BousquetAlgebraic}). We will also obtain a recurrence that counts $3$-stack-sortable permutations in $\mathcal I(\mathsf{\overline D}{\bf T})$.

The \dfn{tail length} of a permutation $\pi=\pi_1\cdots\pi_n\in S_n$, denoted $\tl(\pi)$, is the largest integer $\ell\in\{0,\ldots,n\}$ such that $\pi_i=i$ for all $i\in\{n-\ell+1,\ldots,n\}$. For example, $\tl(324156)=2$, $\tl(3421)=0$, and $\tl(12345)=5$. Let $\mathcal D_{\geq\ell}(n)=\{\pi\in\Av_{n+\ell}(231):\tl(\pi)\geq\ell\}$. In particular, $D_{\geq 0}(n)=\Av_n(231)$. Let ${\bf T}$ be a troupe, and let $f_1,\ldots,f_r$ be insertion-additive tree statistics. Let \[{\bf G}^{(x_1,\ldots,x_r)}(y)=\sum_{\ell\geq 0}{\bf G}_\ell(x_1,\ldots,x_r)y^\ell=\sum_{\ell\geq 0}\sum_{T\in{\bf T}_\ell}x_1^{f_1(T)}\cdots x_r^{f_r(T)}y^\ell.\] We are interested in the generating function \[I_{x_1,\ldots,x_r}(z,y)=\sum_{\ell\geq 0}\sum_{n\geq 0}\sum_{\mathcal T\in\mathcal P^{-1}(\mathcal D_{\geq\ell}(n))\cap\mathsf{\overline D}{\bf T}}x_1^{\ddot f_1(\mathcal T)}\cdots x_r^{\ddot f_r(\mathcal T)}z^ny^\ell.\] In truth, we will be primarily interested in the specialization \[I_{x_1,\ldots,x_r}(z,0)=\sum_{n\geq 0}\sum_{\mathcal T\in\mathcal P^{-1}(\Av_n(231))\cap\mathsf{\overline D}{\bf T}}x_1^{\ddot f_1(\mathcal T)}\cdots x_r^{\ddot f_r(\mathcal T)}z^n.\]

\begin{theorem}\label{Thm21}
With notation as above, we have \[(I_{x_1,\ldots,x_r}(z,y)-I_{x_1,\ldots,x_r}(z,0))(I_{x_1,\ldots,x_r}(z,y)-{\bf G}^{(x_1,\ldots,x_r)}(y))\] 
\[=\frac{I_{x_1,\ldots,x_r}(z,y)-{\bf G}^{(x_1,\ldots,x_r)}(y)}{z}-\frac{I_{x_1,\ldots,x_r}(z,y)-I_{x_1,\ldots,x_r}(z,0)}{y}.\] 
\end{theorem}

\begin{proof}
The specific case in which ${\bf T}=\mathsf{BPT}$, $f_1(T)=\des(T)+1$, $f_2(T)=\peak(T)+1$ was proven in Section 4 of \cite{DefantCounting}, except that the proof there was written in the language of stack-sorting instead of postorder readings. The exact same proof applies, \emph{mutatis mutandis}, in this more general setting. The main difference is that one must now use the Refined Tree Decomposition Lemma (Theorem~\ref{Thm17}) instead of the Refined Decomposition Lemma (Corollary~\ref{Cor2}). We omit the details. 
\end{proof}

Let $\mathbb K=\mathbb C(x_1,\ldots,x_r)$. Suppose the generating function ${\bf G}^{(x_1,\ldots,x_r)}(y)$ is algebraic over $\mathbb K(y)$, meaning that it satisfies a polynomial equation with coefficients in $\mathbb K(y)$. Then we can solve the equation in Theorem~\ref{Thm21} for ${\bf G}^{(x_1,\ldots,x_r)}(y)$, substitute the result into the polynomial satisfied by ${\bf G}^{(x_1,\ldots,x_r)}(y)$, and clear denominators in order to obtain a polynomial equation of the form \[Q(I_{x_1,\ldots,x_r}(z,y),I_{x_1,\ldots,x_r}(z,0),z,y)=0.\] This is a polynomial equation with one ``catalytic variable'' $y$. Therefore, the next result follows immediately from Theorem 3 in \cite{BousquetJehanne}. 

\begin{corollary}\label{Cor8}
Preserve the notation from above. If ${\bf G}^{(x_1,\ldots,x_r)}(y)$ is algebraic over $\mathbb K(y)$, then $I_{x_1,\ldots,x_r}(z,y)$ is algebraic over $\mathbb K(z,y)$ and, consequently, $I_{x_1,\ldots,x_r}(z,0)$ is algebraic over $\mathbb K(z)$. 
\end{corollary}

The preceding corollary yields the algebraicity of several generating functions associated to troupes. For specific examples, suppose ${\bf T}$ and $f_1,\ldots, f_r$ are as in one of the Examples \ref{ExamBinary1}, \ref{ExamFull1}, \ref{ExamMotzkin1}, or \ref{ExamSchroder1}. In each of these cases, we saw that the generating function ${\bf G}^{(x_1,\ldots,x_r)}(y)$ is algebraic over $\mathbb K(y)$, so it follows that $I_{x_1,\ldots,x_r}(z,0)$ is algebraic over $\mathbb K(z)$. To make this even more concrete, we will show how to use the methods from \cite{BousquetJehanne} to find an explicit algebraic equation satisfied by $I_{x_1,\ldots,x_r}(z,0)$ when ${\bf T}=\mathsf{FBPT}$ and $r=0$. In this case, \[I_{x_1,\ldots,x_r}(z,0)=I(z,0)=\sum_{n\geq 0}|\mathcal P^{-1}(\Av_n(231))\cap\mathsf{\overline{D}FBPT}|z^n.\] Note that $s^{-1}(\Av_n(231))=s^{-1}(\mathcal W_1(n))=\mathcal W_2(n)$. Using \eqref{Eq11}, we find that the in-order reading gives a bijection between $\mathcal P^{-1}(\Av_n(231))\cap\mathsf{\overline{D}FBPT}$ and the set $\mathcal W_2(n)\cap\ALT$ when $n$ is odd. Thus, $I(z,0)$ is the generating function for (standardized) $2$-stack-sortable alternating permutations of odd length. 

\begin{corollary}\label{Cor7}
Let \[I(z,0)=\sum_{n\geq 0}|\mathcal P^{-1}(\Av_n(231))\cap\mathsf{\overline{D}FBPT}|z^n=\sum_{k\geq 0}|\mathcal W_2(2k+1)\cap\ALT|z^{2k+1}\] be the generating function that counts (standardized) $2$-stack-sortable alternating permutations of odd length. Then $\mathcal R(I(z,0),z)=0$, where \[\mathcal R(v,z)=(-z + 27 z^3) + (1 - 33 z^2) v + (4 z + 33 z^3) v^2 + (6 z^2 + 
    z^4) v^3 + 4 z^3 v^4 + z^4 v^5.\] 
\end{corollary}

\begin{proof}
We have ${\bf G}_\ell=|\mathsf{FBPT}_\ell|=C_{(\ell-1)/2}$, so \[{\bf G}(y)=\sum_{\ell\geq 0}{\bf G}_\ell y^\ell=\sum_{k\geq 0}C_kz^{2k+1}=\frac{1-\sqrt{1-4z^2}}{2z}.\] Therefore, $y{\bf G}(y)^2+y-{\bf G}(y)=0$. We can solve the equation in Theorem~\ref{Thm21} for ${\bf G}(y)$, substitute the result into the identity $y{\bf G}(y)^2+y-{\bf G}(y)=0$, and clear denominators to find that \begin{equation}\label{Eq62}
Q(I(z,y),I(z,0),z,y)=0,
\end{equation} where 
\[\begin{split}
Q(u,v,z,y)=\: &(1 - u z + vz)^2 y^2 - (1 - uz + vz) (vz - u^2 z y + 
    u (y + z (-1 + v y))) \\ 
    &+ (v z - u^2 z y + u 
    (y + z (-1 + vy)))^2.
    \end{split}\]

Let $Q_u'=\dfrac{\partial}{\partial u}Q(u,v,z,y)$. There is a unique fractional power series (Puiseux series) $Y=Y(z)$ such that $Y(z)=z+O(z^2)$ and 
\begin{equation}\label{Eq63}
Q_u'(I(z,Y),I(z,0),z,Y)=0.
\end{equation} Indeed, one can calculate the coefficients of $Y(z)$ one at a time from the equation \eqref{Eq63} after initially computing sufficiently many terms of $I(z,y)$ via its combinatorial definition. Let $\Delta_uQ(v,z,y)$ be the discriminant of $Q(u,v,z,y)$ with respect to the variable $u$. Using Mathematica, we find that this discriminant is $\Delta_uQ(v,z,y)=z^6(1 - 4 y^2)^2y^3 \widehat Q(v,z,y)$, where \[\begin{split}
\widehat Q(v,z,y)=\: &-4 z^3 + y z^2 (-3 + v z)^2 + y^3 (1 + 2 v z + z^2 + v^2 z^2)^2  \\ 
&+ 
 2 y^2 z (-3 + 5 z^2 - v^2 z^2 + v^3 z^3 + v z (-5 + z^2)).
\end{split}
\] We now use Theorem 14\footnote{In the notation of \cite{BousquetJehanne}, we are applying Theorem 14 with $k=1$. Our polynomial $Q(u,v,z,y)$, power series $I(z,y)$, and power series $I(z,0)$ are playing the roles of $P(x_0,\ldots,x_k,t,v)$, $F(t,u)$, and $F_1(t)$, respectively, from that article.} from the paper \cite{BousquetJehanne}, which allows us to deduce from \eqref{Eq62} and \eqref{Eq63} that $y=Y(z)$ is a repeated root of $\Delta_uQ(I(z,0),z,y)$. Since $Y(z)=z+O(z^2)$, we know that $z^6(1-4Y^2)^2Y^3\neq 0$. Consequently, $y=Y(z)$ is a repeated root of $\widehat Q(I(z,0),z,y)$. The discriminant of a polynomial with a repeated root must be $0$. This means that $\Delta_y\widehat Q(I(z,0),z)=0$, where $\Delta_y\widehat Q(v,z)$ is the discriminant of $\widehat Q(v,z,y)$ with respect to $y$. Computing $\Delta_y\widehat Q(v,z)$ explicitly and ignoring extraneous factors, we find that $\mathcal R(I(z,0),z)=0$, where $\mathcal R(v,z)$ is as in the statement of the corollary.  
\end{proof}

\begin{remark}
Using the techniques from \cite[Chapter VII]{Flajolet}, one can deduce from Corollary~\ref{Cor7} that for odd $n$, the number $|\mathcal W_2(n)\cap\ALT|$ of $2$-stack-sortable alternating permutations in $S_n$ satisfies the asymptotic formula \[|\mathcal W_2(n)\cap\ALT|\sim\beta n^{-5/2}\gamma^n,\] where $\beta\approx0.68444$ and $\gamma\approx 4.10868$.  The exponential growth rate for the number of $2$-stack-sortable permutations in $S_n$ is $\lim\limits_{n\to\infty}|\mathcal W_2(n)|^{1/n}=6.75$, so the probability that a randomly-chosen $2$-stack-sortable permutation in $S_n$ is alternating is roughly $(\gamma/6.75)^n\approx0.60869^n$. On the other hand, the probability that a randomly-chosen permutation in $S_n$ is alternating is roughly $(2/\pi)^n\approx 0.63662^n$. Therefore, if we choose $\pi\in S_n$ uniformly at random, where $n$ is large and odd, then the events ``$\pi$ is $2$-stack-sortable" and ``$\pi$ is alternating" are negatively correlated. 
\end{remark}

Next, suppose ${\bf T}=\mathsf{Mot}$. Set $r=0$ so that \[I(z,0)=\sum_{n\geq 0}|\mathcal P^{-1}(\Av_n(231))\cap\mathsf{\overline{D}Mot}|z^n.\] Note that the in-order reading gives a bijection from $\mathcal P^{-1}(\Av_n(231))\cap\mathsf{\overline{D}Mot}$ to the set $\mathcal W_2(n)\cap\EDP$ of $2$-stack-sortable permutations in $S_n$ in which every descent is a peak. 

\begin{corollary}\label{Cor9}
Let \[I(z,0)=\sum_{n\geq 0}|\mathcal P^{-1}(\Av_n(231))\cap\mathsf{\overline{D}Mot}|z^n=\sum_{n\geq 0}|\mathcal W_2(n)\cap\EDP|z^n\] be the generating function the counts (standardized) $2$-stack-sortable permutations whose descents are all peaks. Then $\mathcal R(I(z,0),z)=0$, where \[\mathcal R(v,z)=(-z + 3 z^2 + 24 z^3 + 
   z^4) + (1 - 4 z - 27 z^2 + 26 z^3 + 4 z^4) v + (4 z - 4 z^2 + 
    29 z^3 + 7 z^4) v^2 \] \[+ (6 z^2 + 4 z^3 + 7 z^4) v^3 + 
 4 (z^3 + z^4) v^4 + z^4 v^5.\] 
\end{corollary}

\begin{proof}
We have ${\bf G}_\ell=|\mathsf{Mot}_\ell|=M_{\ell-1}$, so \[{\bf G}(y)=\sum_{\ell\geq 0}{\bf G}_\ell y^\ell=\sum_{\ell\geq 0}M_{\ell-1}y^\ell=\frac{1-y-\sqrt{1-2y-3y^2}}{2y}.\]  It follows that $y{\bf G}(y)^2+(y-1){\bf G}(y)+y=0$. We can solve the equation in Theorem~\ref{Thm21} for ${\bf G}(y)$, substitute the result into the identity $y{\bf G}(y)^2+(y-1){\bf G}(y)+y=0$, and clear denominators to find that $Q(I(z,y),I(z,0),z,y)=0$, where 
\[\begin{split}
Q(u,v,z,y)=\: &y^2 (1 - u z + v z)^2 + (-1 + u) (-1 + u z - v z) (-v z + u^2 y z - 
    u (y - z + v y z))  \\ 
    &+ (v z - u^2 y z + u (y - z + v y z))^2.
    \end{split}\]

We now proceed exactly as in the proof of Corollary~\ref{Cor7}, computing the discriminant $\Delta_uQ(v,z,y)$ $=z^6y^3(-1 + 2 y + 3 y^2)^2  \widehat Q(v,z,y)$, where \[\begin{split}
\widehat Q(v,z,y)=\: &-4 z^3 + y z^2 (9 + (2 - 6 v) z + (1 + v)^2 z^2) + 
 y^3 (1 + z + 2 v z + (1 + v + v^2) z^2)^2  \\ 
&+ 
 2 y^2 z (-3 - (2 + 5 v) z - (-4 + v^2) z^2 + (1 + 2 v + 2 v^2 + 
       v^3) z^3),
\end{split}
\] and deducing that $\Delta_y\widehat Q(I(z,0),z)=0$. Computing $\Delta_y\widehat Q(v,z)$ explicitly and ignoring extraneous factors, we find that $\mathcal R(I(z,0),z)=0$, where $\mathcal R(v,z)$ is as desired.  
\end{proof}

\begin{remark}
Using the techniques from \cite[Chapter VII]{Flajolet}, one can deduce from Corollary~\ref{Cor9} that the number $|\mathcal W_2(n)\cap\EDP|$ of $2$-stack-sortable permutations in $S_n$ whose descents are all peaks satisfies the asymptotic formula \[|\mathcal W_2(n)\cap\EDP|\sim\beta n^{-5/2}\gamma^n,\] where $\beta\approx0.42022$ and $\gamma\approx 5.46152$. Since $\lim\limits_{n\to\infty}|\mathcal W_2(n)|^{1/n}=6.75$, the probability that a randomly-chosen $2$-stack-sortable permutation in $S_n$ is in $\EDP$ is roughly $(\gamma/6.75)^n\approx0.80911^n$. On the other hand, the probability that a randomly-chosen permutation in $S_n$ is in $\EDP$ is roughly $\left(\dfrac{3\sqrt{3}}{2\pi}\right)^n\approx 0.82699^n$. Therefore, if we choose $\pi\in S_n$ uniformly at random, where $n$ is large, then the events ``$\pi$ is $2$-stack-sortable" and ``every descent of $\pi$ is a peak" are negatively correlated. 
\end{remark}

\begin{remark}
One could easily refine Corollary~\ref{Cor9} by taking into account the statistic $\des$. This would simply amount to replacing the generating function ${\bf G}(y)=\sum_{\ell\geq 0}M_{\ell-1} y^\ell$ with ${\bf G}^{(x_1)}(y)=\sum_{\ell\geq 0}M_{\ell-1}(x_1) y^\ell$, where $M_{\ell-1}(x_1)$ denotes a Motzkin polynomial. 
\end{remark} 

We end this section with a theorem about postorder preimages of $2$-stack-sortable permutations. Using the in-order reading, one can transfer the statement of the theorem to a statement about the enumeration of $3$-stack-sortable permutations associated to troupes. For example, taking ${\bf T}=\mathsf{FBPT}$, one obtains a recurrence for counting $3$-stack-sortable alternating permutations of odd length. 

\begin{theorem}
Let ${\bf T}$ be a troupe, and let $f_1,\ldots,f_r$ be insertion-additive tree statistics. Let \[{\bf G_\ell}={\bf G}_\ell(x_1,\ldots,x_r)=\sum_{T\in{\bf T}_\ell}x_1^{f_1(T)}\cdots x_r^{f_r(T)}.\] If $n\geq 1$, then \[\sum_{\mathcal T\in\mathcal P^{-1}(\mathcal W_2(n))\cap\mathsf{\overline{D}}{\bf T}_n}x_1^{\ddot f_1(\mathcal T)}\cdots x_r^{\ddot f_r(\mathcal T)}=\sum_{g=1}^{n+1}E_{\geq 0}^{(g)}(n),\] where $E_{\geq \ell}^{(g)}(n)=E_{\geq \ell}^{(g)}(n)(x_1,\ldots,x_r)$ are polynomials in $\mathbb C[x_1,\ldots,x_r]$ satisfying the following relations. We have $E_{\geq\ell}^{(0)}(n)=0$ and \[E_{\geq\ell}^{(g)}(1)=\begin{cases} 0, & \mbox{if } g\neq 2; \\ {\bf G}_{\ell+1}, & \mbox{if } g=2. \end{cases}\] If $n,g\geq 1$ and $\ell\geq 0$, then 
\[E_{\geq\ell}^{(g)}(n+1)=\sum_{j=1}^\ell\left(\sum_{a=2}^{n}\sum_{b=\max\{2,g-a\}}^{g-1}\sum_{i=a-1}^{n-b+1}E_{\geq j-1}^{(a)}(i) E_{\geq \ell-j+1}^{(b)}(n-i)+ E_{\geq j-1}^{(g-1)}(n){\bf G}_{\ell-j+1}\right)\] \[+ E_{\geq\ell+1}^{(g-1)}(n).\]
\end{theorem}

\begin{proof}
This theorem appears as Theorem 5.3 in \cite{DefantCounting} in the specific case in which ${\bf T}=\mathsf{BPT}$, $f_1(T)=\des(T)+1$, and $f_2(T)=\peak(T)+1$, although it is phrased in terms of stack-sorting in that article. The exact same proof applies, \emph{mutatis mutandis}, in this more general setting. The main difference is that one must now use the Refined Tree Decomposition Lemma (Theorem~\ref{Thm17}) instead of the Refined Decomposition Lemma (Corollary~\ref{Cor2}). We omit the details. 
\end{proof}

\section{The Troupe Transform}\label{Sec:Transform}
We saw in Theorem~\ref{Thm19} that a troupe ${\bf T}$ is completely determined by its set of branch generators ${\bf T}\cap\mathsf{Branch}$. We would like to know more about the enumerative relationships between a troupe and its set of branch generators. As a starting point, let us prove that the sequence $(|{\bf T}_n|)_{n\geq 0}$ is determined by the sequence $(|{\bf T}_n\cap\mathsf{Branch}|)_{n\geq 0}$. 

\begin{theorem}\label{Thm24}
Let ${\bf T}$ and $\widetilde {\bf T}$ be troupes. If $|{\bf T}_n\cap\mathsf{Branch}|=|\widetilde{\bf T}_n\cap\mathsf{Branch}|$ for all $n\geq 0$, then $|{\bf T}_n|=|\widetilde{\bf T}_n|$ for all $n\geq 0$. 
\end{theorem}

\begin{proof}
Given a colored binary plane tree $T$ with $n$ vertices, we can view the postorder as a total ordering on the set of vertices of $T$. More precisely, the vertex of $T$ that is read $i^\text{th}$ in postorder is the vertex with label $i$ in the unique decreasing colored binary plane tree $\ddot T$ that satisfies $\skel(\ddot T)=T$ and $\mathcal P(\ddot T)=123\cdots n$. Let ${\bf T}_n^p$ (respectively, $\widetilde{\bf T}_n^p$) be the set of trees in ${\bf T}_n$ (respectively, $\widetilde{\bf T}_n$) in which exactly $p$ vertices have $2$ children. Let us say two trees $T$ and $\widetilde T$ with $n$ vertices have the same \dfn{shape} if for every $i\in[n]$, the vertex of $T$ that is read $i^\text{th}$ in postorder has the same number of children as the vertex of $\widetilde T$ that is read $i^\text{th}$ in postorder. We say a map $\varphi:{\bf T}_n^p\to\widetilde{\bf T}_n^p$ is \dfn{shape-preserving} if for every $T\in{\bf T}_n^p$, the trees $T$ and $\varphi(T)$ have the same shape. We will prove that for all $n,p\geq 0$, there is a shape-preserving bijection $\varphi_n^p:{\bf T}_n^p\to\widetilde{\bf T}_n^p$. When $p=0$, this is immediate from the hypothesis that $|{\bf T}_n\cap\mathsf{Branch}|=|\widetilde{\bf T}_n\cap\mathsf{Branch}|$ for all $n\geq 0$. We now proceed by induction on $p$. 

Let $T\in{\bf T}_n^p$ for some $p\geq 1$. Among the $p$ vertices of $T$ that have $2$ children, let $v^*$ be the one that is read last in postorder. Let $\Delta_{v^*}(T)=(T_1,T_2)$. Because ${\bf T}$ is decomposition-closed, we have $T_1\in{\bf T}_{n_1}^{p_1}$ and $T_2\in{\bf T}_{n_2}^{p_2}$ for some $n_1,n_2<n$ and $p_1,p_2<p$. Let $\widetilde T_1=\varphi_{n_1}^{p_1}(T_1)$ and $\widetilde T_2=\varphi_{n_2}^{p_2}(T_2)$. Let $v$ be the left child of $v^*$ in $T$. Then $v$ is also a vertex in $T_1$; say it is the vertex of $T_1$ that is read $r^\text{th}$ in postorder. Let $\widetilde v$ be the vertex of $\widetilde T_1$ that is read $r^\text{th}$ in postorder. Let $\varphi_n^p(T)=\widetilde T=\nabla_{\widetilde v}(\widetilde T_1,\widetilde T_2)$. Let $\widetilde v^*$ be the parent of $\widetilde v$ in $\widetilde T$. By the induction hypothesis, $\varphi_{n_1}^{p_1}$ and $\varphi_{n_2}^{p_2}$ are shape-preserving. This means that $T_1$ and $\widetilde T_1$ have the same shape and that $T_2$ and $\widetilde T_2$ have the same shape. Notice that $v^*$ is the vertex of $T$ read $(r+n_2+1)^\text{th}$ in postorder and that $\widetilde v^*$ is the vertex of $\widetilde T$ read $(r+n_2+1)^\text{th}$ in postorder. It follows that $T$ and $\widetilde T$ have the same shape, so the resulting map $\varphi_n^p:{\bf T}_n^p\to\widetilde{\bf T}_n^p$ is shape-preserving. Furthermore, among the $p$ vertices of $\widetilde T$ that have $2$ children, $\widetilde v^*$ is the one that is read last in postorder. 

By induction, the maps $\varphi_n^{p'}$ with $p'<p$ are bijections; let $\psi_n^{p'}$ denote their inverses. If we perform the same construction as above, except with the roles of ${\bf T}$ and $\widetilde{\bf T}$ switched and with the maps $\varphi_n^{p'}$ with $p'<p$ replaced by the maps $\psi_n^{p'}$, then we obtain a map $\psi_n^p:\widetilde{\bf T}_n^p\to{\bf T}_n^p$. Using the observation made in the last sentence of the preceding paragraph, we find that $\psi_n^p$ is the inverse of $\varphi_n^p$.
\end{proof}

The preceding theorem yields a new transform on sequences of nonnegative integers, which we call the \dfn{troupe transform}. Indeed, for any sequence of nonnegative integers $(\omega_n)_{n\geq 0}$, we can find a set $B$ of branches that has $\omega_n$ elements with $n$ vertices for all $n\geq 0$. For this, we might have to use an infinite set of colors, but this will not cause any harm as long as there are only finitely many elements of $B$ with each fixed number of vertices. We can then consider the troupe $\InsCl(B)$ generated by $B$ and define the new sequence $(\widecheck{\omega}_n)_{n\geq 0}$ by letting $\widecheck{\omega}_n$ be the number of elements of $\InsCl(B)$ with $n$ vertices. For example, if $(\omega_n)_{n\geq 0}=1,1,2,4,8,16,\ldots$ is the sequence enumerating the set $\mathsf{BPT}\cap\mathsf{Branch}$, then $(\widecheck\omega_n)_{n\geq 0}=1,1,2,5,14,42,\ldots$ is the sequence enumerating $\mathsf{BPT}$. Similar considerations for the troupes $\mathsf{FBPT}$, $\mathsf{Mot}$, and $\mathsf{Sch}$ show that $0,1,0,0,0,0,\ldots$ transforms into the sequence $0,1,0,1,0,2,0,5,\ldots$ of aerated Catalan numbers, that $0,1,1,1,1,1,\ldots$ transforms into the sequence $0,1,1,2,4,9,\ldots$ of Motzkin numbers, and that $0, 2, 6, 18, 54, 162,\ldots$ (whose $n^\text{th}$ term is $2\cdot 3^{n-1}$ for $n\geq 1$) transforms into the sequence $1, 2, 6, 22, 90, 394,\ldots$ of large Schr\"oder numbers.

\section{Concluding Remarks and Open Problems}\label{Sec:Conclusion}

In Section~\ref{Subsec:Troupes}, we defined insertion and decomposition, with which we initiated the development of a theory of troupes. We believe that there is likely much more to be done in this line of work. Here, we state some specific open problems and conjectures.

We saw in Theorem~\ref{Thm24} that the sequence that enumerates a troupe ${\bf T}$ is determined by the sequence that enumerates the set ${\bf T}\cap\mathsf{Branch}$ of branch generators of ${\bf T}$. This led us to define the troupe transform of a sequence of nonnegative integers. It would interesting to have a better understanding of this transform, especially on the level of generating functions. In view of Corollary~\ref{Cor8}, we are also interested in the algebraicity of the generating function that enumerates a troupe. 

\begin{question}\label{Quest1} 
Let ${\bf T}$ be a troupe. What can we deduce about the generating function $\displaystyle\sum_{n\geq 0}|{\bf T}_n|z^n$ from the generating function $\displaystyle\sum_{n\geq 0}|{\bf T}_n\cap\mathsf{Branch}|z^n$? Under what conditions will the former be algebraic? 
\end{question}  

There has been interest in binary plane trees and decreasing binary plane trees in algebraic settings \cite{Hivert, Loday1, Loday2}. It could be interesting to see if there are algebraic aspects of the insertion and decomposition operations or of troupes.

It would certainly be interesting to prove any of Conjectures~\ref{Conj3}, \ref{Conj4}, \ref{Conj5}, or \ref{Conj6}, which concern the real-rootedness of the polynomials that count specific families of decreasing colored binary plane trees according to the number of descents in their postorder readings. In fact, it would be nice just to have a proof that one of these polynomials has unimodal coefficients. In general, there exist troupes ${\bf T}$ and positive integers $n$ such that the polynomials $\displaystyle\sum_{\mathcal T\in\mathsf{\overline{D}}{\bf T}_{n-1}}x^{\des(\mathcal P(\mathcal T))+1}$ do not have unimodal coefficients and, consequently, have some nonreal roots. For an example, let $\Upsilon$ be the set of branches with $7$ vertices in which each vertex is either black or white. Let ${\bf T}$ be the troupe whose branch generators are the elements of $\Upsilon$ and the tree consisting of a single black vertex. Then ${\bf T}_7=\Upsilon\cup\mathsf{FBPT}_7$, and one can compute that \[\sum_{\mathcal T\in\mathsf{\overline{D}}{\bf T}_7}x^{\des(\mathcal P(\mathcal T))+1}=8197x+71x^2+140x^3+56x^4.\] To produce this example, we have exploited our freedom to color vertices. Thus, we have the following question concerning troupes whose trees only have black vertices. 
\begin{question}\label{Quest2}
Do there exist a troupe ${\bf T}\subseteq \mathsf{BPT}$ and a positive integer $n$ such that the coefficients of $\displaystyle\sum_{\mathcal T\in\mathsf{\overline{D}}{\bf T}_{n-1}}x^{\des(\mathcal P(\mathcal T))+1}$ are not unimodal? 
\end{question}

Recall Conjecture~\ref{Conj2}, which states that the random variables $D_n$ are asymptotically normally distributed. While explaining a potential approach to this conjecture (which is likely to fail), we observed the strange fact that $\lim\limits_{n\to\infty}\mathbb E(D_{n,1})=\lim\limits_{n\to\infty}\dfrac{\mathbb E(D_n)}{n}$ (see \eqref{Eq34}). This says that if $\sigma\in S_{n-1}$ is chosen uniformly at random, then the probability that $1$ is a descent of $s(\sigma)$ is asymptotically equal to the probability that a random index $i\in[n-2]$ is a descent of $s(\sigma)$. This is suspiciously similar to the fact that if $\pi$ is chosen uniformly at random from the set of uniquely sorted permutations in $S_{2k+1}$, then the expected value of the first entry of $\pi$ is $k+1$ (this follows from Theorem 5.7 in \cite{DefantEngenMiller}), which is also the expected value of a random entry of $\pi$. It would be very interesting to provide a deeper explanation for these observations. On the other hand, we mentioned in Remark~\ref{Rem4} that the analogue of \eqref{Eq34} for $\mathsf{FBPT}$ does not hold. We also saw in Remarks~\ref{Rem5} and \ref{Rem6} that analogues of \eqref{Eq34} for $\mathsf{Mot}$ and $\mathsf{Sch}$ are probably false as well. 

\begin{question}\label{Quest3}
Suppose we choose $\mathcal T\in\mathsf{\overline{D}Mot_{n-1}}$ uniformly at random. As $n\to\infty$, does the probability that $1$ is a descent of $\mathcal T$ approach a limit? If so, what is its value?  
\end{question}

\begin{question}\label{Quest4}
Suppose we choose $\mathcal T\in\mathsf{\overline{D}Sch_{n-1}}$ uniformly at random. As $n\to\infty$, does the probability that $1$ is a descent of $\mathcal T$ approach a limit? If so, what is its value?  
\end{question}

Define the \dfn{fertility distribution} on $S_{n-1}$ to be the probability distribution on $S_{n-1}$ in which the probability of a permutation $\pi$ is $|s^{-1}(\pi)|$. With this alternative terminology, one can view the results in Section~\ref{Subsec:DescentsSorted} as an analysis of the distribution of the descent statistic with respect to the fertility distribution. It could be interesting to consider the distributions of other permutation statistics with respect to this distribution. 

It is likely that several of the results concerning the stack-sorting map, especially those in \cite{DefantCounting, DefantEnumeration, DefantFertility, DefantFertilityWilf, DefantClass}, could be generalized to the setting of troupes using the Refined Tree Decomposition Lemma and the Refined Tree Fertility Formula. We illustrated this in Section~\ref{Sec:2-stack}, but it is possible that pushing this line of work further could lead to some interesting results. For example, it should be possible to enumerate standardized permutations $\pi$ whose descents are all peaks and with the property that $s(\pi)$ avoids some collection of patterns (say, $132$ and $231$). 

Recall from Theorem~\ref{Thm7} that Lassalle's sequence counts uniquely sorted permutations. Lassalle \cite{Lassalle} proved that for $k\geq 3$, the number $\mathscr{A}_k$ is odd if and only if $k+1$ is a power of $2$. This is analogous to our Conjecture~\ref{Conj1}, which states that if $n\geq 3$, then $|\VHC(S_{n-1})|$ is odd if and only if $n+1$ is a power of $2$. It would be interesting to have a combinatorial proof of Lassalle's result and/or Conjecture~\ref{Conj1}.

Finally, let us recall Problem~\ref{Prob1} and Remark~\ref{Rem3}. The former asks for a formula for the number of alternating permutations in $s^{-1}(\pi)$ when $\pi$ has even length. The latter asks for alternative methods for computing the numbers $|\mathcal L(K(\eta))|$ for $\eta\in\NC(n)$.

\section{Acknowledgments}
The author thanks Octavio Arizmendi, Mikl\'os B\'ona, Darij Grinberg, and Takahiro Hasebe for interesting comments. The author was supported by a Fannie and John Hertz Foundation Fellowship and an NSF Graduate Research Fellowship.

\end{document}